\documentclass[pdflatex,sn-mathphys-num]{sn-jnl}

\usepackage{times}
\usepackage{listings}
\usepackage{array} 
\usepackage{url}
\usepackage{graphicx}
\usepackage{xcolor}
\usepackage{booktabs}
\usepackage[T1]{fontenc}
\usepackage{lmodern}
\usepackage{amsfonts}
\usepackage{amssymb}
\usepackage{amsmath}
\usepackage{hyperref}
\usepackage{cleveref}
\usepackage{enumitem}

\newcommand{\R}{\mathbb{R}}
\newcommand{\N}{\mathbb{N}}

\theoremstyle{thmstyleone}%
\newtheorem{theorem}{Theorem}%
\newtheorem{proposition}[theorem]{Proposition}%
\newtheorem{lemma}[theorem]{Lemma}%
\newtheorem{corollary}[theorem]{Corollary}%

\theoremstyle{thmstyletwo}%
\newtheorem{remark}{Remark}%

\theoremstyle{thmstylethree}%
\newtheorem{definition}{Definition}%
\newtheorem{assum}{Assumption}

\usepackage{amsopn}
\DeclareMathOperator{\diag}{diag}

\begin{document}

\title[Stability Connection Between Discrete-Time Algorithms and Their Resolution ODEs]{On the Stability Connection Between Discrete-Time Algorithms and Their Resolution ODEs: Applications to Min-Max Optimisation}

\author*[1]{\fnm{Amir Ali} \sur{Farzin}}\email{amirali.farzin@anu.edu.au}

\author[2]{\fnm{Yuen-Man} \sur{Pun}}\email{yuenman.pun@unimelb.edu.au}

\author[1]{\fnm{Philipp} \sur{Braun}}\email{philipp.braun@anu.edu.au}

\author[2]{\fnm{Iman} \sur{Shames}}\email{iman.shames@unimelb.edu.au}

 \affil[1]{\orgdiv{School of Engineering}, \orgname{Australian National University}, \orgaddress{\city{Canberra}, \state{ACT}, \country{Australia}}}

 \affil[2]{\orgdiv{Department of Electrical and Electronic Engineering}, \orgname{University of Melbourne}, \orgaddress{\city{Melbourne}, \state{VIC}, \country{Australia}}}

\abstract{
We establish a connection between the stability of discrete-time algorithms (DTAs) and continuous-time dynamical systems, in particular their \( O(s^r) \)-resolution ordinary differential equations (ODEs). We show that exponential stability of a common equilibrium transfers to the DTA whenever the one-step error is small relative to the distance to the equilibrium and the one-step Lyapunov decay. This condition is sharp in order, and one-step consistency alone is insufficient. The result is local and allows step-size-dependent ODEs and Lyapunov functions, covering \( O(s) \)-resolution ODEs whose decay rate vanishes as the step size tends to zero. Asymptotic stability is not inherited in general, but a comparison criterion and a checkable compatibility condition ensure it carries over for equilibria and compact sets. For resolution ODEs, joint regularity of the update map guarantees the required error bound. We apply the framework to two-timescale gradient descent ascent (TT-GDA), generalised extragradient, two-timescale proximal point, and (regularised) damped Newton, showing that for small step sizes and suitable parameters every saddle point with invertible Hessian is locally exponentially stable for all but TT-GDA. 
For degenerate saddle points, the framework yields asymptotic stability and convergence rates.
}

\maketitle

\section{Introduction}\label{sec:intro}
Many iterative algorithms in optimisation and game theory can be viewed as discrete-time dynamical systems, and analysing the stability of their equilibria is fundamental to understanding convergence. Direct analysis of discrete-time dynamics is often technically involved, especially for algorithms arising in nonconvex-nonconcave min-max optimisation problems, where convergence to local saddle points is not guaranteed and spurious attractors may appear~\cite{hsieh2021limits}.
Continuous-time dynamical systems are typically easier to analyse than their discrete-time counterparts. A powerful idea is to associate a discrete-time algorithm (DTA) with a continuous-time flow that it tracks closely for small step sizes. This is formalised in~\cite{lu2022sr} through the notion of an \(O(s^r)\)-resolution ordinary differential equation (resolution ODE) 
which approximates the discrete-time updates. Resolution ODEs have been used to gain insight into the behaviour of optimisation algorithms by transferring the analysis from the discrete-time domain to a continuous-time setting. However, a rigorous understanding of when and how local stability properties transfer between these two domains has remained largely implicit.

In this paper, we establish general theorems rigorously connecting 
local stability properties of  {equilibria of} a discrete-time system and  {stability properties of equilibria of} its associated resolution ODE. 
Using a Lyapunov approach, we show that if the one-step error between the two systems is small relative to the distance to a common equilibrium and to the decay of a Lyapunov function of the continuous-time system (Assumption~\ref{as:error}), then exponential stability of the equilibrium is inherited by the discrete-time system. Asymptotic stability is not inherited in general, and we give conditions under which asymptotic stability of equilibria and compact sets, together with convergence rates, carries over to the discrete-time system.

The practical value of this contribution is a systematic tool: rather than directly analysing the often complicated discrete-time dynamics, one can study the corresponding resolution ODE, determine its stability properties, and transfer these conclusions back to the original algorithm. 
 {We illustrate the application of the results on the example of min-max optimisation algorithms.}

\paragraph{Application to min-max optimisation algorithms}
Min-max optimisation problems model the interaction between two decision-makers, where one aims to maximise an objective function while the other seeks to minimise it. These problems have traditionally been studied in economics \cite{myerson2013game}, and have recently gained prominence in areas such as adversarial training \cite{madry2018towards} and multi-agent reinforcement learning \cite{omidshafiei2017deep}. A standard formulation of a min-max problem is given by
\begin{equation}\label{eq:eq20}
	\min_{x \in \mathbb{R}^n}\;\max_{y \in \mathbb{R}^m}\;f(x,y),
\end{equation}
where $f: \mathbb{R}^n\times\mathbb{R}^m\rightarrow\mathbb{R}$ is the objective function. We apply the aforementioned local stability results to investigate stability properties of  min-max algorithms.

Iterative numerical algorithms, specifically first-order gradient-based methods such as Gradient Descent Ascent (GDA)
\cite{chen1997convergence} and their successors, 
including
Extragradient Algorithms (EG)~\cite{korpelevich1977extragradient}, proximal point methods~\cite{rockafellar1976monotone}, optimistic GDA
(OGDA)~\cite{wei2021last}, and stochastic GDA
(SGDA)~\cite{beznosikov2023stochastic}, and second-order methods, such as Newton-type algorithms are the go-to methods of choice
for solving problem \eqref{eq:eq20}.
Generally, in the literature of min-max optimisation, the algorithms have been
studied from the perspective of first-order optimality or convergence to an
$\epsilon$-stationary point, i.e., a point $(x,y)$ such that $\|\nabla
f(x,y)\|\leq\epsilon$
\cite{diakonikolas2021efficient,farzin2025min,lin2020gradient}. It is also
known that these iterative algorithms may fail to converge to 
local saddle points (local Nash equilibria), e.g., GDA will fail to converge to 
saddle  {points} for a bilinear objective function~\cite{hsieh2021limits}.

While it is known that 
finding Nash equilibria is generally PPAD-complete~\cite{daskalakis2009complexity,daskalakis2021complexity}, one can certify whether the saddle points
are a subset of the limit points of a given algorithm or not. To do so,
we can analyse the discrete-time algorithm from the viewpoint of discrete-time
dynamical systems and study the stability properties of its fixed points,
similar to~\cite{daskalakis2018limit}. A possible framework is to find an
accurate enough continuous flow that the discrete-time algorithms follow,
analyse them, and draw conclusions about the DTA from the analysis of the
continuous flow. In \cite{lu2022sr}, the author introduced a framework to find
$O(s^r)$-resolution ODEs 
 of a DTA 
with a mild assumption on the DTA. 
For prominent iterative numerical algorithms, we show how to 
obtain their 
$O(s^r)$-resolution ODE 
and analyse 
if saddle points are a subset of the stable 
equilibria of the ODE.
If they are, and the one-step error is controlled as in Section~\ref{sec:stab}, we conclude that they are stable equilibria of the DTA for sufficiently small step sizes. Using this framework, one can characterise 
the relationship between 
saddle points and 
stable equilibria of a DTA.
Moreover, traditional convergence analyses often rely on linearisation and spectral conditions around equilibria, which typically require invertibility of the Hessian at the saddle point \cite{chae2023two,jin2020local,daskalakis2018limit}. At degenerate saddle points, the linearisation of the dynamics has a zero eigenvalue and is inconclusive. Our framework addresses this case through Lyapunov functions of the resolution ODE, and we identify conditions under which asymptotic stability of degenerate saddle points, and of sets of saddle points, together with sublinear convergence rates, is inherited by the discrete-time algorithm.

Classical numerical-analysis literature has established a number of
connections between the qualitative properties of continuous-time
dynamical systems and those of their numerical discretisations. In
particular, stability and invariant-manifold properties near hyperbolic
equilibria have been studied for one-step discretisations
\cite{beyn1987numerical}, the behaviour of centre manifolds under discretisation has been analysed in \cite{beyn1987center}, while related results establish
the persistence of attracting sets under discretisation
\cite{kloeden1986stable}. These results typically concern a fixed
continuous-time system and a discretisation of that system, with the
analysis relying on properties such as hyperbolicity, structural
stability, or uniform asymptotic stability. Our setting is different.
The resolution ODE itself depends on the discretisation parameter
$s$, and its stability margin may vanish as $s\to0$. Consequently, a
standard stability-preservation argument for a fixed ODE does not
directly yield a stability-transfer result for the resolution-ODE
setting considered here.
\subsection{Contributions} 
        The main contributions of this paper are as follows:

        \begin{enumerate}[label=(\roman*)]
            \item We develop a stability-transfer theory for general pairs of discrete-time systems and ODEs. Exponential stability of an equilibrium is inherited by the discrete-time system when the one-step error is small relative to the distance to the equilibrium and to the decrease of a Lyapunov function of the ODE over one step (Theorem~\ref{thm:stab}). We demonstrate that this condition is sharp in order and that one-step consistency alone is not sufficient (Remarks~\ref{rem:ce1} and~\ref{rem:ce2}). The result uses only local information near the equilibrium, and both the ODE and the Lyapunov function may depend on the step size, so that it applies to $O(s)$-resolution ODEs whose decay rate vanishes as the step size tends to zero.

            \item We show that asymptotic stability is not inherited in general (Remark~\ref{rem:ce3}), and we give a comparison criterion (Theorem~\ref{thm:stab_set_general}) and a checkable second-order compatibility condition (Theorem~\ref{thm:stab1_A}) under which asymptotic stability of equilibria and compact sets is inherited for a fixed step size, together with sublinear convergence rates (Corollary~\ref{cor:rates}). This covers non-hyperbolic equilibria and sets of equilibria, for which linearisation is inconclusive.

            \item We identify the conditions under which the resolution ODEs of \cite{lu2022sr} can be used for stability transfer. The one-step estimate in the definition of resolution ODEs holds pointwise in the initial condition, and we show that this is not sufficient. Joint regularity of the update map in the step size and the state is sufficient, as it yields the required error bound (Proposition~\ref{prop:defect}). As a consequence, we obtain transfer results for the $O(1)$-resolution ODE and for the $O(s)$-resolution ODE (Corollaries~\ref{cor:res_fixed} and~\ref{cor:res_vanishing}). The latter covers the case in which the $O(1)$-resolution ODE is not exponentially stable and the $O(s)$ correction decides stability.

            \item We apply the framework to several first- and second-order min-max algorithms through their resolution ODEs. We establish that, for sufficiently small step sizes, saddle points with invertible Hessian are locally exponentially stable fixed points of the Generalised Extragradient method with $\gamma\in(0,2)$ and of the Proximal Point method, including saddle points whose linearisation has purely imaginary eigenvalues, where the $O(s)$-resolution ODE is decisive (Proposition~\ref{prop:saddle_transfer}). The same holds for the (Regularised) Damped Newton methods, whose $O(1)$-resolution ODE is already exponentially stable at every such saddle point. We highlight the inherent limitations of the Two-Timescale Gradient Descent Ascent and Jacobian methods. For degenerate saddle points, including a non-separable example, and for compact sets of critical points, we obtain asymptotic stability and convergence rates without Hessian invertibility assumptions (Theorem~\ref{thm:degenerate} and Propositions~\ref{prop:nonsep} and~\ref{prop:bump}).
        \end{enumerate}

	Overall, we provide a rigorous bridge between continuous-time and discrete-time stability analysis, offering both theoretical insights and practical tools for studying iterative algorithms.
    
\subsection{Related Works}\label{sec:related}
The relationship between the stability properties of discrete and continuous-time dynamical systems has long been studied in control theory and numerical analysis \cite{deuflhard2012scientific,hairer1993solving}.  In the numerical analysis literature, the stability of numerical discretisation methods has been extensively characterised through the concept of $A$-stability, $L$-stability, and related notions~\cite{hairer1993solving,deuflhard2012scientific, butcher2016numerical}. For linear ODEs, these results are based on how well a discretisation method's rational function can approximate an exponential function. Based on consistency (order of error in approximating the exponential function) of the discretisation method, one can find a small enough step size to ensure that the origin is asymptotically stable for the discrete-time dynamics. The results ensure that if the continuous-time system is asymptotically stable, then the discrete-time system obtained via an $A$-stable method (e.g., implicit Euler) preserves stability for all step sizes, while explicit methods preserve it only for sufficiently small step sizes. For nonlinear ODEs, the results show that one can linearise around the equilibrium and, conditional on the negativity of the real parts of the eigenvalues of the Jacobian of the dynamics at the equilibrium, one can use the results for 
linear systems to 
conclude properties 
for nonlinear systems. 
The results on hyperbolic equilibria and centre manifolds \cite{beyn1987numerical,beyn1987center} discussed in Section~\ref{sec:intro} concern discretisations of a fixed ODE. In contrast, we work directly with Lyapunov functions, which allows  {for} ODEs that depend on the step size, sets of equilibria, and non-hyperbolic equilibria.

For compact attractors that are asymptotically but not necessarily
exponentially stable, one-step discretisations may possess
attracting sets close to the attractor, which converge to it as
the step size tends to zero
\cite{kloeden1986stable,stuart1996dynamical}.
In \cite{grune2003attraction}, this convergence is characterised
in terms of uniform attraction rates and robustness. In general,
however, the original attractor of the ODE need not itself remain
asymptotically stable for the discretised system, and only
practical stability may be expected (see Remark~\ref{rem:ce3}).
Our results address the stability of the original common
invariant set itself, rather than the existence of a nearby
attracting set. In particular, Theorems~\ref{thm:stab_set_general}
and~\ref{thm:stab1_A} provide conditions under which this set
is asymptotically stable for the discrete-time system at a fixed
step size, while Corollary~\ref{cor:rates} gives corresponding
convergence rates. These results therefore complement existing
attractor-persistence results by providing conditions for
preservation of the original set.
Runge--Kutta discretisations of gradient systems, and conditions
under which they inherit a Lyapunov function of the continuous-time
system, are studied in \cite{humphries1994runge}. In sampled-data
control, \cite{nevsic2005lyapunov} redesigns a continuous-time
controller with correction terms in powers of the sampling period
to improve the Lyapunov decrease of the exact discrete-time model,
and obtains semiglobal practical stability. The use of the
$O(s)$-resolution ODE in Section~\ref{sec:main} is similar in
spirit, but there the correction terms are determined by the
algorithm, and our conclusions concern exponential or asymptotic
stability of the equilibrium or invariant set itself.
In the discrete control literature, there are papers 
such as \cite{nevsic1999sufficient,nevsic2009stability}, which 
study controller designs with practical stability guarantees by analysing 
an exact discrete-time model of a nonlinear plant.
Practical stability for a
family of approximate discrete-time models of a sampled nonlinear plant is achieved 
under one-step consistency or multi-step consistency assumptions on 
the approximate and exact discrete-time system. 
More recently, discrete-to-continuous stability connections in special cases have been studied for learning and optimisation algorithms viewed as dynamical systems, e.g.,~\cite{fazlyab2018analysis,su2016differential}, which exploit continuous-time analogues to understand convergence rates and robustness. 


In the min-max optimisation literature, the relationship between the saddle points and fixed points of some of the discrete-time algorithms has been studied before. 
In~\cite{daskalakis2018limit}, the fixed-point properties of the GDA and OGDA methods are analysed. It is shown that, under certain assumptions, a certain set of saddle points of the objective function is contained within the set of limit points of the GDA method, which, in turn, is contained within the set of limit points of the OGDA method. 
In \cite{jin2020local}, the authors show 
that the same property holds for Two-Timescale GDA (TT-GDA)\footnote{The choice of the quantifier ``two-timescale'' to describe the algorithm is imprecise, as the algorithm does not operate in two distinct timescales. However, ``two-timescale'' is  commonly used in the literature to describe this algorithm and similar ones. Using ``two-stepsize'' would possibly lead to a more precise name.}. When the ratio of the maximiser step size to minimiser step size goes to infinity ($\infty$-GDA), the local min-max points are a subset of $\infty$-GDA stable points.
However, it is well known that the GDA method may converge to limit cycles or even diverge~\cite{daskalakis2018limit}. This phenomenon is not unexpected. Analogous issues arise in continuous-time saddle-point dynamics for min-max problems, which converge only under additional structural assumptions (see, e.g.,~\cite[Ch.~11.5]
{goebel2024set}). Furthermore, \cite{hsieh2021limits} analyses GDA-type min-max algorithms within the Robbins–Monro framework and demonstrates the existence of spurious attractors that contain no stationary points of the underlying problem.
In \cite{chae2023two}, the authors analyse 
Two-Timescale EG and prove 
that when the ratio of the maximiser step size to minimiser step size goes to infinity ($\infty$-EG), the local min-max points are a subset of $\infty$-EG stable points with relaxed
assumptions compared to \cite{jin2020local}.
In \cite{farzin2025properties}, the authors analyse 
a Generalised 
Extragradient Algorithm (GEG) and prove 
that under certain assumptions, all 
saddle points of the objective function are a subset of the stable equilibria of the GEG algorithm under proper hyperparameter selection.

In \cite{lu2022sr}, the $O(s^r)$-resolution ODE of a discrete-time
algorithm is introduced, together with a construction procedure and
a uniqueness result (see Theorem~\ref{th:ODE}). It is further shown
in \cite[Thm.~2]{lu2022sr} that linear convergence of the resolution
ODE with respect to a suitable energy function implies linear
convergence of the discrete-time algorithm. The extension to
sublinear convergence is posed as an open direction in
\cite[Sec.~6]{lu2022sr}. Our work addresses a different and more
general question. Under what conditions are stability properties
of equilibria and compact invariant sets of a continuous-time
system inherited by a corresponding discrete-time system?
We develop a Lyapunov-based theory for general pairs of
discrete-time systems and ODEs, without requiring the construction
of a particular energy function. In particular, our results allow
the ODE and its Lyapunov function to depend on the step size,
cover non-hyperbolic equilibria and compact sets of equilibria,
and provide conditions for transferring asymptotic stability
together with sublinear convergence rates. We also establish
conditions under which the one-step approximation property of
resolution ODEs yields the error bounds required by our
stability-transfer results.

\smallskip

\textit{Outline:}  The main stability results are given in Section~\ref{sec:stab}. Necessary preliminaries and definitions for the analysis of DTAs for \eqref{eq:eq20} are introduced in Section~\ref{sec:per}, 
before they are used in 
Section~\ref{sec:main}, which presents 
results 
on equilibria of the dynamical systems associated with min-max algorithms. 
Conclusions and 
future research directions are discussed in Section~\ref{sec:conc}. 
Appendix~\ref{app:extra} collects results from the literature. Proofs are given in Appendix~\ref{app:main}, and Appendix~\ref{app:o} contains the constructions of the resolution ODEs.
Code reproducing the numerical experiments that illustrate the theoretical findings is available at \url{https://github.com/amirali78frz/MCSS-minmax-stability-numerics.git} and can be run online via \href{https://mybinder.org/v2/gh/amirali78frz/MCSS-minmax-stability-numerics.git/HEAD?urlpath=%2Fdoc%2Ftree%2FMCSS_binder.ipynb}{\texttt{Binder}}.

\smallskip

\textit{Notation:}
For  
$x,y\in \mathbb{R}^d$, $\delta>0$ and $A\in \R^{m\times d}$, let
$\langle x , y \rangle=x^\top y$, 
$|x|=\sqrt{\langle x , x \rangle}$, $|x|_{y}=|x-y|$, $\mathcal{B}_\delta(x)=\{y\in \R^d| \ |y|_{x} \leq \delta\}$ and $\|A\|=\sup_{|x|=1}|Ax|$.   
For $c=a+\jmath b \in\mathbb{C}$,
$\Re(c)=a,$ $\Im(c)=b,$ and $|c|=\sqrt{a^2+b^2}$. For $v\in\mathbb{C}^d$ and $A\in\mathbb{C}^{m\times d}$, $v^{\mathsf H}$ and $A^{\mathsf H}$ denote the conjugate transposes, and $|v|=\sqrt{v^{\mathsf H}v}$. 
For $r,n,m\in \N$, 
$f:\R^{n}\rightarrow \R^m$ is a $C^r$-function if it is $r$-times continuously differentiable.
Gradient and Hessian of 
$f: \mathbb{R}^{n+m} \to \mathbb{R}$ are denoted by 
$\nabla f$ and  $\nabla^2 f$. 
For $(x,y)\in \R^{n+m}$, $ \nabla_x f$, $\nabla_y f$, $\nabla^2_{xx} f$, $ \nabla^2_{xy} f$, and $\nabla^2_{yy} f$ 
denote parts of the gradient/Hessian where the derivative of $f$ is taken with respect to the subscript.
For $A\in \R^{n\times n}$, 
$\sigma(A) = \{\lambda \in \mathbb{C} | \ \det(A-\lambda I) = 0 \}$ and $
\rho(A) = \max_{\lambda \in \sigma(A)} |\lambda|$ denote the spectrum and the spectral radius. 
For $A$ symmetric, $A \succ 0$ ($A \succeq 0$) indicates that  \( A \) is positive (semi-)definite and $I$ denotes the identity matrix. We denote the distance of a point $z$ to the set $\mathcal{A}\subset \R^d$ by $|z|_\mathcal{A} = \inf_{a\in\mathcal{A}}|z-a|$ and $\mathcal{B}_\delta(\mathcal{A})=\{y\in \R^d| \ |y|_{\mathcal{A}} \leq \delta\}$ extends the definition of a closed ball.
For $\bar s>0$ and functions $h:(0,\bar s]\to\R^d$ and $\phi:(0,\bar s]\to\R_{>0}$, we write $h(s)=O(\phi(s))$ if $\limsup_{s\to0^+}|h(s)|/\phi(s)<\infty$, and $h(s)=o(\phi(s))$ if $\lim_{s\to0^+}|h(s)|/\phi(s)=0$. When $h$ depends on further arguments, we state explicitly whether such a bound holds pointwise or uniformly with respect to them. A continuous function $\alpha:\R_{\geq0}\to\R_{\geq0}$ is of class $\mathcal{K}_\infty$ if it is strictly increasing, unbounded, and satisfies $\alpha(0)=0$.

\section{Equilibria of continuous-time and discrete-time systems}\label{sec:stab}

In this section, we analyse the connection between the stability properties of common equilibria of discrete-time and continuous-time dynamical systems satisfying a specific property.
While the main results focus 
on
the stability properties of an equilibrium, we additionally consider stability properties of compact sets.
Consider a discrete-time dynamical system $w_{s}: \R^{d} \rightarrow \R^d$,
\begin{align}
	z_{k+1} 
    = w_s(z_k) \qquad (\text{or } z^+=w_s(z))  \label{eq:discrete_time_sys}
\end{align}
with state $z\in \R^d$, which depends on 
a positive constant $s\in \R_{>0}$.
The solution of \eqref{eq:discrete_time_sys} at time $k\in \N$ with respect to an initial condition $z_0\in \R^d$ is denoted by $z_k(z_0)$ or simply by $z_k$ if the initial condition is clear from the context.
The constant $s$ can, for example, represent the step size in optimisation algorithms. 
In addition, we consider a continuous-time dynamical system $W_{s}: \R^d\rightarrow\R^d$,
\begin{align}\label{eq:continuous_sys}
	\dot{Z} 
    = W_s(Z)
\end{align}
with state $Z\in \R^d$ and positive constant $s\in \R_{>0}$.
Similar to the discrete-time setting, the solution of \eqref{eq:continuous_sys} at time $t\in \R_{>0}$ w.r.t. an initial condition $z_0 \in \R^d$
is denoted by $Z_{s}(t;z_0)$
or simply by $Z(t)$ if the initial condition and the parameter $s$ are
clear from the context.

Throughout this section, $s$ ranges over an interval $(0,\bar s]$ with $\bar s>0$, and, for every $s \in(0,\bar s]$, $W_s$ is defined and locally Lipschitz continuous on an open set that contains all balls and neighbourhoods considered below. 
Whenever $Z_s(\cdot\,;z)$ exists on $[0,s]$, the \emph{one-step error}
\begin{align}\label{eq:defect}
	e_s(z) := w_s(z) - Z_s(s;z)
\end{align}
compares one iteration of \eqref{eq:discrete_time_sys} with the flow of \eqref{eq:continuous_sys} over a time interval of length $s$. In contrast to classical one-step consistency conditions, the stability results below require the error to be controlled \emph{relative to the distance to the equilibrium} and \emph{relative to the decay rate of the continuous-time dynamics}. Remarks~\ref{rem:ce1} and~\ref{rem:ce2} in Section~\ref{sec:exp_stab_prop} show that neither requirement can be dropped.

\subsection{Exponential stability of common equilibria} \label{sec:exp_stab_prop}

In this section, we discuss conditions 
under which an exponentially stable equilibrium of the dynamics
\eqref{eq:continuous_sys} is also exponentially stable for the discrete-time dynamics \eqref{eq:discrete_time_sys}. Before we present the result, we introduce the necessary definitions.

\begin{definition}
	Consider the discrete-time system \eqref{eq:discrete_time_sys}. 
	For $s\in \R_{>0}$, a 
    point $z^e\in \R^d$ is called 
	an equilibrium of \eqref{eq:discrete_time_sys} 
    if  $ w_s(z^e) = z^e$ holds. Moreover, for $s\in \R_{>0}$, it is an equilibrium of the continuous-time system \eqref{eq:continuous_sys} if $ W_s(z^e) = 0$ is satisfied.
\end{definition}

Note that an equilibrium of a discrete-time system is a fixed point of the mapping $w_{s}(\cdot)$. While the term equilibrium is common in the dynamical systems literature, fixed points are more common in the literature on iterative optimisation.  
Here, we use both terms interchangeably.

\begin{definition} \label{def:stab_discr}
Consider the discrete-time dynamics \eqref{eq:discrete_time_sys} and let
    $z^e \in \R^d$ be a corresponding equilibrium. 
	The equilibrium $z^e$ is said to be locally
	stable for 
    \eqref{eq:discrete_time_sys} if, for every $\varepsilon>0$, there exists $\delta>0$ such that 
	if $|z_0|_{z^e} < \delta$ then, for all $k\geq 0$, 
	\(
		|z_k|_{z^e} < \varepsilon.
	\) 
	Otherwise, $z^e$ is called unstable. 
	The equilibrium $z^e$ is locally asymptotically stable if it is locally stable and there exists $\delta>0$ such that if $|z_0|_{z^e} < \delta$, then
	\(
		\lim_{k\to\infty}|z_k|_{z^e} = 0.
	\)
	The equilibrium $z^e$ is locally exponentially stable if it is stable and there exist $\delta>0$, $M>0$ and $\gamma\in(0,1)$ such that if
    $|z_0|_{z^e} < \delta$, then
	 \(
	 	|z_k|_{z^e} \leq M \, |z_0|_{z^e} \, \gamma^k
	 \)
     for all  $k\geq 0.$
\end{definition}


\begin{definition} \label{def:stab_cont}
	Consider the ODE 
    \eqref{eq:continuous_sys} and let 
    $Z^e \in \R^d$ be an 
    equilibrium. 
	The equilibrium  $Z^e$ is said to be locally
	stable for 
    \eqref{eq:continuous_sys} if, for every $\varepsilon>0$, there exists $\delta>0$ such that 
	if $|Z_0|_{Z^e} < \delta$, then for all $t\geq 0$, 
	\(
		|Z(t)|_{Z^e} < \varepsilon.
	\) 
	Otherwise, $Z^{e}$ is called unstable. 
	The equilibrium $Z^e$ is locally asymptotically stable if it is stable and there exists $\delta>0$ such that if $|Z_0|_{Z^e} < \delta$ then
	\(
		\lim_{t\to\infty}|Z(t)|_{Z^e} = 0.
	\)
	The equilibrium $Z^e$ is locally exponentially stable if it is stable and there exist $\delta>0$, $M>0$ and $\lambda>0$ such that, for all $|Z_0|_{Z^e} < \delta$,
	\(
		|Z(t)|_{Z^e} \leq M e^{-\lambda t} |Z_0|_{Z^e}, 
	\)
    for all  $t\geq 0.$
\end{definition}

An equilibrium $Z^e$ of \eqref{eq:continuous_sys} with $C^1$ vector field is called \emph{hyperbolic} if $\frac{\partial W_s}{\partial z}(Z^e)$ has no eigenvalue on the imaginary axis, and \emph{non-hyperbolic} otherwise. Stability of hyperbolic equilibria is determined by linearisation (see Theorem~\ref{thm:stability_of_equilibria} in Appendix \ref{app:extra}), whereas linearisation is inconclusive for non-hyperbolic equilibria.
We impose two assumptions. The first one bounds the one-step error near a common equilibrium.

\begin{assum}[Local error bound]\label{as:error}
	There exist $\delta_e>0$ and a function $\eta:(0,\bar s]\to\R_{\geq0}$ such that, for every $s\in(0,\bar s]$, the point $z^e\in\R^d$ is an equilibrium of both \eqref{eq:discrete_time_sys} and \eqref{eq:continuous_sys}, the solution $Z_s(\cdot\,;z)$ of \eqref{eq:continuous_sys} exists on $[0,s]$ for all $z\in\mathcal B_{\delta_e}(z^e)$, and
	\begin{align}\label{eq:defect_bound}
		|e_s(z)|\leq \eta(s)\,|z-z^e| \qquad \forall\, z\in\mathcal B_{\delta_e}(z^e).
	\end{align}
\end{assum}

Assumption~\ref{as:error} differs from the classical one-step consistency condition $|e_s(z)|=O(s^{r})$ used in
\cite{deuflhard2012scientific,nevsic1999sufficient} in two ways. The bound holds uniformly on a neighbourhood of $z^e$ and is proportional to $|z-z^e|$, and the function $\eta$ is arbitrary rather than a power of $s$. Since $e_s(z^e)=0$, the bound \eqref{eq:defect_bound} controls the error near the equilibrium.    The following example shows that controlling only the size of the error, uniformly in $z$, is not sufficient.

\begin{remark}[A uniformly small error is not sufficient]\label{rem:ce1}
	Let $d=1$, $W_s(z)=-z$ and $w_s(z)=(1-s)z+s^3\sin(z/s^4)$ for $s\in(0,1)$. Then $Z_s(s;z)=e^{-s}z$. Since $e^{-s}=1-s+\tfrac{s^2}{2}e^{-\xi}$ for some $\xi\in(0,s)$, we obtain
	\begin{align*}
		|e_s(z)|=\big|(1-s-e^{-s})z+s^3\sin(z/s^4)\big|\leq \tfrac{s^2}{2}|z|+s^3 \qquad\forall\, z\in\R .
	\end{align*}
	Thus $e_s(z)=O(s^2)$ uniformly on bounded sets and $\lim_{s\to0^+}|e_s(z)|/s^2=|z|/2$ for every $z$. Moreover, the origin is a common equilibrium, both maps are $C^\infty$ in $z$, and the origin is globally exponentially stable for $\dot Z=-Z$. However,
	\begin{align*}
		\frac{dw_s}{dz}(0)=1-s+s^{-1}>1 \qquad \forall\, s\in(0,1),
	\end{align*}
	so the origin is unstable for \eqref{eq:discrete_time_sys} by Theorem~\ref{thm:stability_of_equilibria}(ii) (in Appendix \ref{app:extra}). The error has amplitude of order $s^2$ but slope $\frac{de_s}{dz}(0)=1-s-e^{-s}+s^{-1}$ at the equilibrium, and it is this slope that governs stability. Since $e_s(0)=0$, \eqref{eq:defect_bound} implies $|{\frac{de_s}{dz}}(0)|\leq\eta(s)$, and {since} $1-s-e^{-s}\geq-\tfrac{s^2}{2}$, the bound \eqref{eq:defect_bound} can only hold with $\eta(s)\geq s^{-1}-\tfrac{s^2}{2}$, which tends to infinity as $s\to0$. This does not violate Assumption~\ref{as:error}, which allows any function $\eta$. However, \eqref{eq:defect_bound} is only useful for stability if $\eta(s)$ is small compared with the decrease of the distance to the origin along the flow over one time interval of length $s$. Here $|Z_s(s;z)|=e^{-s}|z|$, so this decrease equals $(1-e^{-s})|z|$. Indeed, $|w_s(z)|\leq|Z_s(s;z)|+|e_s(z)|\leq(e^{-s}+\eta(s))|z|$, so \eqref{eq:defect_bound} yields $|w_s(z)|<|z|$ only if $\eta(s)<1-e^{-s}<s$. In this example, $\eta(s)$ exceeds $s$ by a factor of order $s^{-2}$, so no stability conclusion can be drawn from \eqref{eq:defect_bound}, in agreement with the instability of the origin. Theorem~\ref{thm:stab} below makes this requirement precise.

\end{remark}

The second assumption requires a Lyapunov function for the continuous-time dynamics whose bounds are uniform in $s$, while the function itself and its decay rate may depend on $s$.

\begin{assum}[$s$-uniform Lyapunov function]\label{as:lyap}
	There exist constants $c_1,c_2,c_3>0$ and functions $\delta_V:(0,\bar s]\to\R_{>0}$ and $\rho:(0,\bar s]\to\R_{>0}$ such that, for every $s\in(0,\bar s]$, there is a $C^1$ function $V_s:\mathcal B_{\delta_V(s)}(z^e)\to\R_{\geq0},$ with $z^e$ as an equilibrium of both \eqref{eq:discrete_time_sys} and \eqref{eq:continuous_sys}, satisfying, for all $z\in\mathcal B_{\delta_V(s)}(z^e)$,
	\begin{align}
		c_1|z-z^e|^2\leq V_s(z)&\leq c_2|z-z^e|^2,\label{eq:lyap_bounds}\\
		|\nabla V_s(z)|&\leq c_3|z-z^e|,\label{eq:lyap_grad}\\
		\langle\nabla V_s(z),W_s(z)\rangle&\leq-\rho(s)V_s(z).\label{eq:lyap_decay}
	\end{align}
\end{assum}

If $W_s=W$ does not depend on $s$, any locally exponentially stable equilibrium of $\dot Z=W(Z),$ with $W$ of class $C^1,$ satisfies Assumption~\ref{as:lyap} with a constant rate $\rho$. Allowing $V_s$ and $\rho(s)$ to depend on $s$ is essential when the continuous-time dynamics themselves depend on $s$, as for the $O(s)$-resolution ODEs of Section~\ref{sec:per}.  A typical case is $W_s=W_0+sW_1$, where the linearisation of $W_0$ at $z^e$ has eigenvalues on the imaginary axis, so that the decay is provided only by the term $sW_1$ and $\rho(s)$ is of order $s$. For instance, for $W_s(z)=(E-sI)z$ with $E=-E^\top$, the function $V_s(z)=|z|^2$ satisfies Assumption~\ref{as:lyap} with $c_1=c_2=1$, $c_3=2$ and $\rho(s)=2s$. If $W_s$ also contains nonlinear terms, their contribution to $\langle\nabla V_s,W_s\rangle$ is of order $|z-z^e|^3$ and is dominated by the decrease of order $s|z-z^e|^2$ only for $|z-z^e|$ of order $s$. In that case, the radius $\delta_V(s)$ shrinks as $s\to0$, while $c_1,c_2,c_3$ can still be chosen independent of $s$.

The key step in the analysis is the following comparison between one step of \eqref{eq:discrete_time_sys} and the flow of \eqref{eq:continuous_sys} over the same time interval. It is also the basis of the results on compact sets in the next subsection.

\begin{lemma}[One-step comparison]\label{lem:comparison}
	Let $s>0$, $D\subset\R^d$, $V:D\to\R$ and $z\in D$ be such that $Z_s(s;z), 
    w_s(z)\in D$. If there exists $\theta\in[0,1)$ with
	\begin{align}\label{eq:comparison_hyp}
		V(w_s(z))-V(Z_s(s;z))\leq\theta\big[V(z)-V(Z_s(s;z))\big],
	\end{align}
	then
	\begin{align}\label{eq:comparison_concl}
		V(w_s(z))-V(z)\leq-(1-\theta)\big[V(z)-V(Z_s(s;z))\big].
	\end{align}
\end{lemma}

A proof of Lemma~\ref{lem:comparison} is given in Appendix~\ref{app:main}. Lemma~\ref{lem:comparison} states that one iteration of the discrete-time system retains at least the fraction $1-\theta$ of the decrease of $V$ along the flow over one time interval of length $s$. Since \eqref{eq:comparison_hyp} compares $w_s(z)$ with the endpoint $Z_s(s;z)$ of the flow rather than with $z$, no second-order expansion of $V$ along $W_s$ is involved. This is what allows decay rates $\rho(s)$ that vanish as $s\to0$ in the following main result.

\begin{theorem}\label{thm:stab}
	Let Assumptions~\ref{as:error} and~\ref{as:lyap} hold, let $\delta(s)=\min\{\delta_e,\delta_V(s)\}$, and define
	\begin{align}\label{eq:thm_constants}
		M=\sqrt{c_2/c_1},\qquad \kappa=\frac{c_1}{4c_3(M+1)},\qquad \tilde\rho(s)=\min\{\rho(s),s^{-1}\}.
	\end{align}
	If $s\in(0,\bar s]$ satisfies
	\begin{align}\label{eq:step_condition}
		\eta(s)\leq\kappa\, s\,\tilde\rho(s),
	\end{align}
	then every solution of \eqref{eq:discrete_time_sys} with $|z_0-z^e|<\delta(s)/(M(M+1))$ satisfies
	\begin{align}\label{eq:exp_estimate}
		|z_k-z^e|\leq M\Big(1-\tfrac{s\tilde\rho(s)}{4}\Big)^{k/2}|z_0-z^e|\qquad\forall\, k\in\N .
	\end{align}
	In particular, $z^e$ is a locally exponentially stable equilibrium of \eqref{eq:discrete_time_sys}.
\end{theorem}

A proof of Theorem~\ref{thm:stab} is given in Appendix~\ref{app:main}. 
Condition \eqref{eq:step_condition} involves only the ratio between the error and the decrease $s\rho(s)$ of the Lyapunov function over one time interval of length $s$, and the estimate \eqref{eq:exp_estimate} gives the contraction factor of the discrete-time system explicitly. The next example shows that \eqref{eq:step_condition} cannot be weakened in order.

\begin{remark}[Tightness of \eqref{eq:step_condition}]\label{rem:ce2}
	Let $d=1$, $W_s(z)=-sz$ and $w_s(z)=(1+s^2)z$ for $s\in(0,1)$. The function $V_s(z)=z^2$ satisfies Assumption~\ref{as:lyap} with $c_1=c_2=1$, $c_3=2$ and $\rho(s)=2s$, so that $s\rho(s)=2s^2$. Moreover, $Z_s(s;z)=e^{-s^2}z$ and $e_s(z)=(1+s^2-e^{-s^2})z$. Since $e^{-s^2}=1-s^2+\tfrac{s^4}{2}e^{-\xi}$ for some $\xi\in(0,s^2)$, we have $2s^2-\tfrac{s^4}{2}\leq 1+s^2-e^{-s^2}\leq 2s^2$, and Assumption~\ref{as:error} holds with $\eta(s)=2s^2$. Both maps are $C^\infty$ jointly in $(s,z)$, and the origin is exponentially stable for \eqref{eq:continuous_sys} for every $s>0$. Nevertheless, $z_k=(1+s^2)^kz_0$ diverges for every $z_0\neq0$. In this example $\eta(s)$ and $s\rho(s)$ are of the same order, so \eqref{eq:step_condition} cannot be relaxed to $\eta(s)=O(s\rho(s))$. It is the ratio $\eta(s)/(s\rho(s))$, and not the order of $\eta(s)$ in $s$, that decides whether stability is inherited.
\end{remark}

When the decay rate of the continuous-time dynamics does not vanish as $s\to0$, in particular when $W_s=W$ does not depend on $s$, it suffices to verify exponential stability of $z^e$ for the ODE, since a converse Lyapunov theorem \cite[Thm.~4.14]{khalil2002nonlinear} then provides the function required in Assumption~\ref{as:lyap}.

\begin{corollary}\label{cor:fixed_margin}
	Let Assumption~\ref{as:error} hold with $\eta(s)=o(s)$. Suppose that there exist $\delta_0>0$, $K\geq1$, $\lambda>0$ and $L>0$ such that, for every $s\in(0,\bar s]$, the function $W_s$ in \eqref{eq:continuous_sys} is $C^1$ on $\mathcal B_{\delta_0}(z^e)$ with $\|\frac{\partial W_s}{\partial z}(z)\|\leq L$ for all $z\in\mathcal B_{\delta_0}(z^e)$, the solution $Z_s(t;z)$ exists for all $t\geq0$ and satisfies
\begin{align}\label{eq:unif_exp}
	|Z_s(t;z)-z^e|\leq K e^{-\lambda t}|z-z^e| \qquad \forall\, t\geq0,\ \forall\, z\in\mathcal B_{\delta_0}(z^e).
\end{align}
	Then there exists $s^\ast\in(0,\bar s]$ such that $z^e$ is a locally exponentially stable equilibrium of \eqref{eq:discrete_time_sys} for every $s\in(0,s^\ast]$.  In particular, this holds if $W_s=W$ is independent of $s$, $W$ is $C^1$, and $z^e$ is a locally exponentially stable equilibrium of $\dot Z=W(Z)$.
\end{corollary}

A proof of Corollary~\ref{cor:fixed_margin} is given in Appendix~\ref{app:main}. In the proof, Assumption~\ref{as:lyap} is verified with the rate $\rho=\lambda/K^2$ and constants depending only on $K$, $\lambda$ and $L$. 
Theorem~\ref{thm:stab} is a statement for a fixed step size. Its proof uses Assumptions~\ref{as:error} and~\ref{as:lyap} only at the given $s$, so it applies to every $s$ that satisfies \eqref{eq:step_condition}, in particular when \eqref{eq:step_condition} holds only on an interval $(\underline{s},\bar s]$ with $\underline{s}>0$, for instance because $\rho(s)$ vanishes too fast as $s\to0$. 

\begin{remark}[Relation to linearisation]\label{rem:linearisation}
	If $W$ is $C^1$ with locally Lipschitz Jacobian, exponential stability of $z^e$ for $\dot Z=W(Z)$ is equivalent to all eigenvalues of $\frac{\partial W}{\partial z}(z^e)$ having negative real parts \cite[Thm.~4.15]{khalil2002nonlinear}. For $s$-independent dynamics, Corollary~\ref{cor:fixed_margin} can therefore be verified through eigenvalues, and Theorem~\ref{thm:stab} complements the classical linearisation argument \cite[Sec.~6.1]{deuflhard2012scientific} with explicit constants and with an error bound $\eta(s)$ in \eqref{eq:defect_bound} that may be any nonnegative function of $s$, rather than a power $Cs^p$ as in classical consistency estimates. For $s$-dependent dynamics whose decay rate $\rho(s)$ vanishes as $s\to0$, stability of the discrete-time system is not determined by the linearisation of any single $s$-independent vector field, and Theorem~\ref{thm:stab} provides the corresponding quantitative criterion. Equilibria and compact sets that are asymptotically but not exponentially stable, for which linearisation is inconclusive, are treated in the next subsection.
\end{remark}

\subsection{Asymptotic stability of compact sets}

While we are in general interested in properties of equilibria, sometimes it might only be possible to show convergence to a compact set $\mathcal{A}\subset \R^d$, but not necessarily convergence to a single point. To cover these cases, we extend the results in Section \ref{sec:exp_stab_prop} to compact attractive sets.

\begin{definition} \label{def:stab_discr_a}
Consider the discrete-time dynamics \eqref{eq:discrete_time_sys} and let
    $\mathcal{A}\subset \R^d$ be a compact 
    set. 
	The set $\mathcal{A}$ is said to be locally
	stable for the system \eqref{eq:discrete_time_sys} if, for every $\varepsilon>0$, there exists $\delta>0$ such that 
	if $|z_0|_\mathcal{A} < \delta$ then, for all $k\geq 0$, 
	\(
		|z_k|_\mathcal{A} < \varepsilon.
	\) 
	The set $\mathcal{A}$ is locally asymptotically stable if it is locally stable and there exists $\delta>0$ such that if $|z_0|_\mathcal{A} < \delta$ then
	\(
		\lim_{k\to\infty}|z_k|_\mathcal{A} = 0.
	\)
\end{definition}

\begin{definition} \label{def:stab_cont_a}
	Consider the ODE 
    \eqref{eq:continuous_sys} and let
    $\mathcal{A}\subset \R^d$ be a compact  
    set.
	The set $\mathcal{A}$ is said to be locally
	stable for 
    \eqref{eq:continuous_sys} if, for every $\varepsilon>0$, there exists $\delta>0$ such that 
	if $|Z_0|_\mathcal{A} < \delta$ then, for all $t\geq 0$, 
	\(
		|Z(t)|_\mathcal{A} < \varepsilon.
	\) 
	The set $\mathcal{A}$ is locally asymptotically stable if it is stable and there exists $\delta>0$ such that if $|Z_0|_\mathcal{A} < \delta$ then
	\(
		\lim_{t\to\infty}|Z(t)|_\mathcal{A} = 0.
	\)
\end{definition}

Unlike exponential stability, asymptotic stability is not robust with respect to perturbations of the vector field that vanish on $\mathcal{A}$ only linearly. The following example shows that it is not inherited by the discrete-time system in general, even when the error satisfies Assumption~\ref{as:error} and both maps are smooth.

\begin{remark}[Asymptotic stability is not inherited in general]\label{rem:ce3}
	Let $d=1$, $W(z)=-z^3$ and $w_s(z)=z/\sqrt{1+2sz^2}+s^2\tanh z$. The solution of the continuous-time dynamics satisfies $Z(t;z)=z/\sqrt{1+2tz^2}$, and thus the error satisfies $e_s(z)=s^2\tanh z$ and $|e_s(z)|\leq s^2|z|$ for all $z\in\R$. Hence Assumption~\ref{as:error} holds with $\eta(s)=s^2$, both maps are jointly $C^\infty$ in $(s,z)$, and the origin is globally asymptotically stable for $\dot Z=W(Z)$. However, $\frac{dw_s}{dz}(0)=1+s^2>1$, so the origin is unstable for the discrete-time system for every $s>0$. 
    To see the mechanism, let $V(z)=z^2$ and $Z=Z(s;z)$. Over one time interval of length $s$, $V(z)-V(Z)=2sz^4/(1+2sz^2)\leq2sz^4$, whereas $V(w_s(z))-V(Z)=2s^2Z\tanh z+s^4\tanh^2z\geq2s^2Z\tanh z$. For $0<|z|\leq1$ and $s\in(0,1]$, we have $\tanh z/z\geq\tanh1$ and $\sqrt{1+2sz^2}\leq\sqrt3$, so that $2s^2Z\tanh z\geq\frac{2\tanh1}{\sqrt3}s^2z^2\geq\frac45s^2z^2$. Hence
\begin{align*}
	V(w_s(z))-V(z)\geq sz^2\big(\tfrac45s-2z^2\big)>0\qquad\text{whenever } 0<|z|<\sqrt{2s/5},
\end{align*}
that is, close to the origin the increase of $V$ caused by the error exceeds the decrease along the flow, and the iterates are pushed away from the origin.
\end{remark}

Remark~\ref{rem:ce3} shows that, beyond exponential stability, the error has to be compared with the decrease of a Lyapunov function along the flow rather than with the distance to $\mathcal{A}$. For $\delta>0$, we write $\mathcal{B}_\delta(\mathcal{A})=\{z\in\R^d \mid |z|_\mathcal{A}\leq\delta\}$. The first result is a direct consequence of Lemma~\ref{lem:comparison}.

\begin{theorem}\label{thm:stab_set_general}
	Let $\mathcal{A}\subset\R^d$ be nonempty and compact, let $s\in(0,\bar s]$ and $0<\delta\leq\delta_2$, and let $V:U\to\R_{\geq0}$ be $C^1$ on an open set $U\supset\mathcal{B}_{\delta_2}(\mathcal{A})$. Suppose that
	\begin{itemize}
		\item[(a)] there exist $\alpha_1,\alpha_2\in\mathcal{K}_\infty$ with $\alpha_1(|z|_\mathcal{A})\leq V(z)\leq\alpha_2(|z|_\mathcal{A})$ for all $z\in\mathcal{B}_{\delta_2}(\mathcal{A})$,
		\item[(b)] $\langle\nabla V(z),W_s(z)\rangle<0$ for all $z\in\mathcal{B}_{\delta_2}(\mathcal{A})\setminus\mathcal{A}$,
		\item[(c)] for every $z\in\mathcal{B}_\delta(\mathcal{A})$, the solution $Z_s(t;z)$ exists and lies in $\mathcal{B}_{\delta_2}(\mathcal{A})$ for all $t\in[0,s]$, $w_s(z)\in\mathcal{B}_{\delta_2}(\mathcal{A})$, and \eqref{eq:comparison_hyp} holds with $V$ and a constant $\theta\in[0,1)$ that does not depend on $z$.
	\end{itemize}
	Then $V(w_s(z))-V(z)\leq-(1-\theta)\big[V(z)-V(Z_s(s;z))\big]\leq0$ for all $z\in\mathcal{B}_\delta(\mathcal{A})$, with strict inequality for $z\notin\mathcal{A}$, and $\mathcal{A}$ is locally asymptotically stable for \eqref{eq:discrete_time_sys}.
\end{theorem}

A proof of Theorem~\ref{thm:stab_set_general} is given in Appendix~\ref{app:main}. Theorem~\ref{thm:stab_set_general} applies to arbitrary compact sets $\mathcal{A}$ that are asymptotically stable for the continuous-time dynamics, including sets that do not consist of equilibria, and to ODEs that depend on $s$. The proof of Theorem~\ref{thm:stab} verifies \eqref{eq:comparison_hyp} with $\theta=\tfrac12$ in the exponential case. Since \eqref{eq:comparison_hyp} at a point $a\in\mathcal{A}$ yields $V(w_s(a))\leq0$, the set $\mathcal{A}$ in Theorem~\ref{thm:stab_set_general} is necessarily forward invariant for \eqref{eq:discrete_time_sys}. In contrast to Theorem~\ref{thm:stab1_A} below, whose compatibility condition forces every point of $\mathcal{A}$ to be a fixed point of $w_s$, it does not need to consist of fixed points. 

In general, \eqref{eq:comparison_hyp} may fail in a neighbourhood of $\mathcal{A}$ whose size shrinks as $s\to0$. In Remark~\ref{rem:ce3}, for instance, let $V(z)=z^2$, $s\in(0,1]$ and $0<|z|\leq1$. The bounds derived there give $V(w_s(z))-V(Z_s(s;z))\geq\frac45s^2z^2$ and $V(z)-V(Z_s(s;z))\leq2sz^4$. Thus, for every $\theta\in(0,1),$ \eqref{eq:comparison_hyp} fails whenever $0<|z|<\min\{1,\sqrt{2s/(5\theta)}\}$. Where \eqref{eq:comparison_hyp} holds, Lemma~\ref{lem:comparison} still yields a decrease of $V$, so the iterates are driven into a neighbourhood of $\mathcal{A}$ that shrinks as $s\to0$. This is consistent with the practical stability results in
\cite{kloeden1986stable,stuart1996dynamical,nevsic1999sufficient}, which
yield a nearby attracting set rather than asymptotic stability of
$\mathcal{A}$ itself.

Condition \eqref{eq:comparison_hyp} involves the solution of \eqref{eq:continuous_sys}, which is usually not available in closed form. The next result replaces it by a condition on the update map. Many discrete-time algorithms can be written as
\begin{align}\label{eq:euler_decomp}
	w_s(z)=z+sW(z)+s^2R_s(z),
\end{align}
where $W$ does not depend on $s$. For instance, $w_s(z)=z+sW(z)$ is the explicit Euler discretisation of $\dot Z=W(Z)$, with $R_s\equiv0$. More generally, if $(s,z)\mapsto w_s(z)$ is $C^2$ and $w_0(z)=z$, then Taylor's theorem in $s$ yields \eqref{eq:euler_decomp} with $W(z)=\frac{\partial w_s}{\partial s}(z)|_{s=0}$ and $R_s(z)=\int_0^1(1-\tau)\,\frac{\partial^2w_{\tau s}}{\partial s^2}(z)\,d\tau$.

\begin{theorem}\label{thm:stab1_A}
		Let $\mathcal{A}\subset\R^d$ be nonempty and compact, let $\delta>0$, let $W:\R^d\to\R^d$ be continuous, and for every $s\in(0,\bar s]$ let $w_s$ satisfy \eqref{eq:euler_decomp} for all $z\in\mathcal{B}_\delta(\mathcal{A})$, with $R_s$ continuous on $\mathcal{B}_\delta(\mathcal{A})$. Suppose that
	\begin{itemize}
				\item[(i)] there exists $M_w>0$ such that $|W(z)+sR_s(z)|\leq M_w$ for all $z\in\mathcal{B}_\delta(\mathcal{A})$ and $s\in(0,\bar s]$,
		\item[(ii)] $V:U\to\R_{\geq0}$ is $C^1$ on an open set $U\supset\mathcal{B}_{2\delta}(\mathcal{A})$, $\nabla V$ is Lipschitz continuous on $\mathcal{B}_{2\delta}(\mathcal{A})$ with constant $L_V$, and there exist $\alpha_1,\alpha_2\in\mathcal{K}_\infty$ with $\alpha_1(|z|_\mathcal{A})\leq V(z)\leq\alpha_2(|z|_\mathcal{A})$ for all $z\in\mathcal{B}_{2\delta}(\mathcal{A})$,
		\item[(iii)] there is a continuous function $\sigma:\mathcal{B}_\delta(\mathcal{A})\to\R_{\geq0}$ with $\sigma(z)>0$ for all $z\in\mathcal{B}_\delta(\mathcal{A})\setminus\mathcal{A}$ and $\langle\nabla V(z),W(z)\rangle\leq-\sigma(z)$ for all $z\in\mathcal{B}_\delta(\mathcal{A})$,
		\item[(iv)] there exists $C>0$ such that
		\begin{align}\label{eq:compatibility}
			\langle\nabla V(z),R_s(z)\rangle+\tfrac{L_V}{2}\,|W(z)+sR_s(z)|^2\leq C\,\sigma(z)\quad\forall\, z\in\mathcal{B}_\delta(\mathcal{A}),\ \forall s\in(0,\bar s].
		\end{align}
	\end{itemize}
	Let $s^\ast=\min\{\bar s,\delta/M_w,1/(2C)\}$.
    Then, for every $s\in(0,s^\ast]$, every point of $\mathcal{A}$ is a fixed point of $w_s$,
	\begin{align}\label{eq:set_decrease}
		V(w_s(z))-V(z)\leq-\tfrac{s}{2}\,\sigma(z)\qquad\forall\, z\in\mathcal{B}_\delta(\mathcal{A}),
	\end{align}
	and $\mathcal{A}$ is locally asymptotically stable for \eqref{eq:discrete_time_sys}.
\end{theorem}

A proof of Theorem~\ref{thm:stab1_A} is given in Appendix~\ref{app:main}. Assumptions (ii) and (iii) state that $V$ is a Lyapunov function for the ODE $\dot Z=W(Z)$ with respect to $\mathcal{A}$, so that $\mathcal{A}$ is locally asymptotically stable for the ODE. Assumption (iv) is the additional structural requirement. It demands that the second-order terms of the update map are dominated by the decrease of $V$ along the ODE, uniformly near $\mathcal{A}$.
Theorem~\ref{thm:stab1_A} applies to non-hyperbolic equilibria and to sets of equilibria, for which linearisation is inconclusive. 
The decrease \eqref{eq:set_decrease} also transfers convergence rates.

\begin{corollary}\label{cor:rates}
	Let the assumptions of Theorem~\ref{thm:stab1_A} hold, and suppose that there exist $c>0$ and $p\geq1$ with $\sigma(z)\geq c\,V(z)^p$ for all $z\in\mathcal{B}_\delta(\mathcal{A})$. Let $s\in(0,s^\ast]$ and $0<\ell<\alpha_1(\delta)$. Then every solution of \eqref{eq:discrete_time_sys} with $z_0\in\mathcal{B}_\delta(\mathcal{A})\setminus\mathcal{A}$ and $V(z_0)\leq\ell$ satisfies
	\begin{alignat*}{3}
		V(z_k)&\leq\Big(1-\tfrac{cs}{2}\Big)^kV(z_0)\;\; \quad &&\text{if } p=1, \text{ and} \\
        V(z_k)&\leq\Big(V(z_0)^{1-p}+(p-1)\tfrac{cs}{2}\,k\Big)^{-\frac{1}{p-1}}\;\; \quad &&\text{if } p>1,
	\end{alignat*}
	for all $k\in\N$. If $z_0\in\mathcal{A}$, then $z_k=z_0$ for all $k\in\N$.
\end{corollary}



\begin{proposition}[Implicit discretisations]\label{prop:implicit}
	Let $W:\R^d\to\R^d$ be continuous, let $V:\R^d\to\R_{\geq0}$ be convex and $C^1$, and let $\sigma:\R^d\to\R_{\geq0}$ be continuous and satisfy $\langle\nabla V(z),W(z)\rangle\leq-\sigma(z)$ for all $z\in\R^d$. Let $s>0$, and suppose that $w_s:\R^d\to\R^d$ satisfies $w_s(z)=z+sW(w_s(z))$ for all $z\in\R^d$. Then
	\begin{align}\label{eq:implicit_decrease}
		V(w_s(z))\leq V(z)-s\,\sigma(w_s(z))\qquad\forall\, z\in\R^d .
	\end{align}
	If, in addition, $w_s$ is continuous, $\mathcal{A}=\{z \mid V(z)=0\}$ is compact, $V$ satisfies the bounds of Theorem~\ref{thm:stab1_A}(ii) on $\R^d$, and $\sigma(z)>0$ for $z\notin\mathcal{A}$, then $\mathcal{A}$ is locally asymptotically stable for $z_{k+1}=w_s(z_k)$.
\end{proposition}

In contrast to Theorems~\ref{thm:stab_set_general} and~\ref{thm:stab1_A}, Proposition~\ref{prop:implicit} requires no restriction on the step size beyond the existence of the implicit update, and no bound on the one-step error. It is the Lyapunov-function counterpart of the unconditional stability of the implicit Euler method for linear systems discussed in Section~\ref{sec:related}, and it applies to \eqref{dyn:TTPPM}. The argument is standard and is included for completeness.

\section{Fixed points of DTAs and saddle points of \texorpdfstring{\eqref{eq:eq20}}{(1)}} \label{sec:per}

In 
Section \ref{sec:main}, we analyse 
connections between saddle points of the min-max problem \eqref{eq:eq20} and fixed points  
of 
state-of-the-art min-max algorithms. Here, we discuss 
preliminary results for this analysis. 
Instead of analysing the algorithms directly, we discuss 
algorithms through the perspective 
of dynamical systems and their stability properties. 
In particular, we write DTAs in the general form of discrete-time dynamical systems \eqref{eq:discrete_time_sys} and we consider the concept of $O(s^r)$-resolution ODEs to connect discrete-time systems \eqref{eq:discrete_time_sys} with continuous-time dynamics \eqref{eq:continuous_sys}. We leverage the results from Section~\ref{sec:stab} to draw connections between saddle points of \eqref{eq:eq20}, fixed points of DTAs, and stability properties of equilibria and invariant sets.

To approximate the flow
that a DTA follows, we 
adopt the definitions and results given in \cite{lu2022sr} on how to obtain ODEs 
from a DTA and define 
$O(s^r)$-resolution ODEs.

\begin{definition}[\mbox{\cite[Def. 1]{lu2022sr}}]\label{def:ODE}
	For $r\in\N\cup\{0\}$ and $C^1$ functions $f_i:\R^d\rightarrow \R^d$, $i\in \{0,\ldots,r\}$, an 
    ODE of the form 
	\begin{align}\label{eq:ode}
		\dot{Z} = f^{(r)}(Z,s) = f_0(Z)+sf_1(Z)+\cdots+s^rf_r(Z)
	\end{align}
	is said to be the $O(s^r)$-resolution ODE of the discrete-time algorithm of $z_{k+1}=w_s
    (z_k)$, $k\in \N$, with step size $s\in \R_{>0}$ if it satisfies
	\begin{align}
      |Z(s;z_0)-w_s(z_0)|=o(s^{r+1}) \qquad \forall \ z_0\in \R^d. 
    \end{align}
\end{definition}

The $O(s^r)$-resolution ODE of \eqref{eq:discrete_time_sys} can be constructed using the following result.
\begin{theorem}[\mbox{\cite[Thm. 1]{lu2022sr}}]\label{th:ODE}
	Consider a DTA with iterate update $z^+=w_s(z)$
    where $w_0(z)=z$ for all $z\in \R^d$ 
    and
    $w_{\cdot}(\cdot)$
    is sufficiently often continuously differentiable in $s$ and in $z$. Then its $O(s^r)$-resolution ODE is unique, and the $i$-th coefficient function in the $O(s^r)$-resolution ODE \eqref{eq:ode} 
    can be obtained recursively by
	\begin{align*}
    f_i(z) = \frac{1}{(i+1)!}\frac{\partial^{i+1} w_s(z)}{\partial s^{i+1}}\bigg|_{s=0}-\sum_{l=2}^{i+1}\frac{1}{l!}h_{l,i+1-l}(z),\quad \forall \  i = 0,1,\ldots,r,
    \end{align*}
	where $h_{0,0}(z)=z$, $h_{1,i}(z) = f_i(z)$ for $i\in\{0,\ldots,r\},$ $ h_{0,i}(z) = 0$ for $i\in\{1,\ldots,r\},$ and
    \begin{align*}
    h_{j+1,i}(z) = \sum_{l=0}^{i}\nabla h_{j,l}(z)h_{1,i-l}(z) \quad \text{for $i\in\{0,\ldots,r\}$ and $j\in\{1,\ldots,r\}.$}
    \end{align*}
\end{theorem}

Definition~\ref{def:ODE} imposes the one-step estimate pointwise in $z_0$. This is not sufficient for the transfer of stability properties.
\begin{remark}[Pointwise consistency is not sufficient]\label{rem:pointwise}
	In the example of Remark~\ref{rem:ce1}, $|e_s(z)|\leq\tfrac{s^2}{2}|z|+s^3=o(s)$ for every $z\in\R$. Hence $\dot Z=-Z$ is an $O(1)$-resolution ODE of $w_s(z)=(1-s)z+s^3\sin(z/s^4)$ in the sense of Definition~\ref{def:ODE}, its origin is exponentially stable, and yet the origin is unstable for the discrete-time system. 
\end{remark}
The next result shows that joint regularity of the update map, which is the setting in which the construction of Theorem~\ref{th:ODE} is carried out, yields Assumption~\ref{as:error} with an explicit order.

\begin{proposition}\label{prop:defect}
	Let $r\in\N\cup\{0\}$, $k\in\N$ with $k\geq r+2$, $\bar s>0$, and let $U\subset\R^d$ be open. Suppose that $(s,z)\mapsto w_s(z)$ is of class $C^k$ on $(-\bar s,\bar s)\times U$ with $w_0(z)=z$ for all $z\in U$, and that the coefficient functions $f_0,\ldots,f_r$ of the $O(s^r)$-resolution ODE in Theorem~\ref{th:ODE} are of class $C^k$ on $U$. Let $z^e\in U$ satisfy $w_s(z^e)=z^e$ for all $s\in[0,\bar s)$, and let $W_s=f^{(r)}(\cdot\,,s)$. Then there exist $\delta>0$ and $s_1\in(0,\bar s)$ such that Assumption~\ref{as:error} holds on $(0,s_1]$ with $\delta_e=\delta$ and
	\begin{itemize}
		\item[(i)] $\eta(s)=s^{r+1}\omega(s)$ for some function $\omega$ with $\lim_{s\to0^+}\omega(s)=0$, if $k=r+2$,
		\item[(ii)] $\eta(s)=K_es^{r+2}$ for some constant $K_e>0$, if $k\geq r+3$.
	\end{itemize}
\end{proposition}

A proof of Proposition~\ref{prop:defect} is given in Appendix~\ref{app:main}. 
Proposition~\ref{prop:defect} converts the resolution-ODE expansion into a
relative one-step error estimate suitable for stability analysis. In
contrast to a pointwise consistency estimate, this bound controls
the error relative to the distance from the equilibrium. 
Combining Proposition~\ref{prop:defect} with Theorem~\ref{thm:stab} and Corollary~\ref{cor:fixed_margin} gives the following two results for resolution ODEs.

\begin{corollary}[$O(1)$-resolution ODE]\label{cor:res_fixed}
	Let the assumptions of Proposition~\ref{prop:defect} hold with $r=0$ and $k=2$. If $z^e$ is a locally exponentially stable equilibrium of the $O(1)$-resolution ODE $\dot Z=f_0(Z)$, then there exists $s^\ast>0$ such that $z^e$ is a locally exponentially stable equilibrium of $z_{k+1}=w_s(z_k)$ for every $s\in(0,s^\ast]$.
\end{corollary}

\begin{corollary}[$O(s)$-resolution ODE]\label{cor:res_vanishing}
	Let the assumptions of Proposition~\ref{prop:defect} hold with $r=1$ and $k=3$. Suppose that Assumption~\ref{as:lyap} holds for $W_s=f_0+sf_1$ with $\rho(s)\geq\rho_1 s$ for some $\rho_1>0$ and all $s\in(0,\bar s]$. Then there exists $s^\ast>0$ such that $z^e$ is a locally exponentially stable equilibrium of $z_{k+1}=w_s(z_k)$ for every $s\in(0,s^\ast]$.
\end{corollary}

Corollary~\ref{cor:res_vanishing} covers the situation in which the $O(1)$-resolution ODE is not exponentially stable because its linearisation at $z^e$ has eigenvalues on the imaginary axis, and the $O(s)$ correction $sf_1$ provides a decay rate of order $s$. This is the case for \eqref{GEG} and \eqref{TTPPM} at saddle points where $\Lambda_\tau H(z^\ast)$ has purely imaginary eigenvalues (Section~\ref{sec:main}).


In the context of min-max problems \eqref{eq:eq20}, fixed points of DTAs correspond to critical points of $f$, which include its local saddle points.


\begin{definition}[Critical point] Let $f: \mathbb{R}^{n+m}
	\rightarrow \mathbb{R}$ be a 
	$C^1$ function.
	A point $(x^\ast,y^\ast)\in \R^{n+m}$ is called a critical point of $f$ if $\nabla f(x^\ast,y^\ast)=0$.
\end{definition}
\begin{definition}[
	{\cite[Def. 3.4.1]{bertsekas2009convex}}]
	A point $(x^\ast,y^\ast)\in \R^{n+m}$ is a local saddle point of $f:\mathbb{R}^{n+m} \rightarrow \mathbb{R}$ 
	if there exists a neighbourhood $U$ around $(x^\ast,y^\ast)$ so that 
	\begin{align}
		f(x^\ast,y)\leq f(x^\ast,y^\ast)\leq f(x,y^\ast) \qquad \forall \ (x,y)\in U. \label{eq:saddle_local}
	\end{align}

\end{definition}
\begin{proposition}[\mbox{\cite[Prop. 4 and 5]{jin2020local}}]\label{locNash}
	Let $f: \mathbb{R}^{n+m}
	\rightarrow \mathbb{R}$ be a $C^2$ function. Then
	any 
	saddle point
	$(x,y)$ is a critical point of $f$ and satisfies $\nabla_{xx}^2 f(x,y) \succeq 0$ and $\nabla_{yy}^2 f(x,y) \preceq 0$. 
\end{proposition}
A critical point $z^\ast$ of a $C^2$ function $f$ is called \emph{nondegenerate} if $\nabla^2 f(z^\ast)$ is invertible, and \emph{degenerate} otherwise. In particular, a \emph{degenerate saddle point} is a local saddle point at which $\nabla^2 f$ is singular.
The proof of Proposition \ref{locNash} follows from standard local optimality conditions (e.g., see \cite{ratliff2013characterization} after Definition 3).
With these definitions and results, we investigate the relationship between the local exponential/asymptotic stability of equilibria of the dynamical systems associated with min-max algorithms and the saddle points of the objective function $f$ in the following section. 
A standard way to verify exponential stability or instability of equilibria $z^*$ of dynamical systems \eqref{eq:discrete_time_sys} and \eqref{eq:continuous_sys}, is to look at the eigenvalues of $\frac{\partial w_s}{\partial z}(z^*)$ or $\frac{\partial W_s}{\partial z}(z^*)$, respectively (see 
Theorem~\ref{thm:stability_of_equilibria}).
In the context of DTAs for min-max problems, conditions on the eigenvalues of $\frac{\partial w_s}{\partial z}(z^*)$ can be translated into conditions on the Hessian $\nabla^2 f(x^*,y^*)$, and thus, a common assumption in the analysis of iterative optimisation algorithms (see  \cite{daskalakis2018limit,chae2023two}, for example), is to assume that the Hessian is invertible.

\section{Analysis of the Fixed Points of Min-Max Algorithms}\label{sec:main}

\subsection{An analysis of commonly used algorithms in the literature}\label{sec:main_p1}

In this section, we analyse the fixed points of Two-Timescale Gradient Descent Ascent (TT-GDA), Generalised Extragradient (GEG), the Two-Timescale Proximal Point Method (TT-PPM), Damped Newton (DN), and Regularised Damped Newton (RDN) at saddle points with invertible Hessian, and we comment on the Jacobian method (JM) of \cite{lu2022sr} in Remark~\ref{rem:JM}. 
The stability properties of several of the individual algorithms
considered below are known or have been studied previously. Our purpose
here is not to re-establish these algorithm-specific facts in
isolation, but to interpret them through the resolution-ODE framework
and to show how the general stability-transfer results above explain
their discrete-time behaviour.
For each method, we introduce the DTA and its $O(1)$- and $O(s)$-resolution ODEs, determine the stability of saddle points for these ODEs, and transfer the conclusions to the DTA using Sections~\ref{sec:stab} and~\ref{sec:per}.
For the first-order methods and RDN, we also show or recall that, with random initialisation, the iterates almost surely do not converge to linearly unstable fixed points. Here, a fixed point $\bar z$ of a $C^1$ map $w_s$ is called \emph{linearly unstable} if $\rho(\frac{\partial w_s}{\partial z}(\bar z))>1$. By Theorem~\ref{thm:stability_of_equilibria}(ii), linearly unstable fixed points are unstable.
The results in this section rely on the following assumption.

\begin{assum}\label{as:inv}
	Let $f: \mathbb{R}^{n+m}\rightarrow \mathbb{R}$ be a $C^2$ function and let $z^\ast=(x^\ast,y^\ast)$ be a local saddle point of $f$.
	\begin{itemize}
		\item[(i)] The Hessian $\nabla^2 f(z^\ast)$ is invertible.
		\item[(ii)] There exists $L>0$ such that $\|\nabla^2f(z)\|\leq L$ for all $z\in\R^{n+m}$, that is, $\nabla f$ is globally Lipschitz continuous with constant $L$.
	\end{itemize}
\end{assum}

Higher regularity of $f$, which is needed for the resolution ODEs, is stated explicitly in each result.
The condition on the Hessian allows us to draw conclusions from the eigenvalues of $\nabla^2 f(x^*,y^*)$.
To make the condition on the Hessian more precise, we introduce the following
variables and notation:
\begin{align}
	z=\left[ \! \begin{array}{c}
		x  \\
		y
	\end{array} \! \right],  \quad
	\Lambda_\tau =\left[ \! \begin{array}{cc}
		\frac{1}{\tau} I & 0  \\
		0 & I
	\end{array} \! \right], \quad
	F(z)=\left[ \! \begin{array}{r}
		\nabla_x f(z)\\
		-\nabla_y f(z)
	\end{array} \! \right] \label{eq:F_definition}
\end{align}
for $\tau \in \R_{>0}$, and $d=n+m.$
Moreover, we denote the derivative of $-F(z)$ as
 \begin{align}
 	H(z) &= \left[\begin{array}{rr}
 		-\nabla_{xx} f(z)& -\nabla_{xy} f(z)\\ \nabla_{yx} f(z)& \nabla_{yy} f(z)
 	\end{array} \right]  =
 	\left[ \begin{array}{rr}
 		-I & 0  \\
 		0 & I
 	\end{array} \right] \nabla^2 f(z),
 	\label{eq:DF_def}
 \end{align}
which has the following property when evaluated at local saddle points.

\begin{lemma}\label{negeig}
	Let $f: \mathbb{R}^{n+m}\rightarrow \mathbb{R}$ be a $C^2$ function, let $\tau>0$, and let $z^\ast$ be a local saddle point of $f$. Then $N=\Lambda_\tau H(z^\ast)$ is similar to a matrix of the form $-S+K$, where $S=S^\top\succeq0$ and $K=-K^\top$. Consequently,
	\begin{itemize}
		\item[(i)] $\Re(\kappa)\leq0$ for all $\kappa\in\sigma(N)$ (see also \cite[Lem.~2]{farzin2025properties}),
		\item[(ii)] every eigenvalue of $N$ on the imaginary axis is semisimple, that is, its algebraic and geometric multiplicities coincide.
	\end{itemize}
	If, in addition, $\nabla^2f(z^\ast)$ is invertible, then $0\notin\sigma(N)$.
\end{lemma}

A proof of Lemma~\ref{negeig} is given in Appendix~\ref{app:main}. Part (ii) is used below to construct Lyapunov functions for the $O(s)$-resolution ODEs whose constants do not depend on $s$. Throughout Section~\ref{sec:main_p1}, we write
\begin{align}\label{eq:N_notation}
	N=\Lambda_\tau H(z^\ast),\qquad \ell_\tau=\max\{1,\tau^{-1}\}=\|\Lambda_\tau\|,\qquad L_\tau=\ell_\tau L,
\end{align}
so that $|\kappa|\leq\|N\|\leq L_\tau$ for all $\kappa\in\sigma(N)$ under Assumption~\ref{as:inv}(ii). Moreover, we define $\mu=\min\{|\Re(\kappa)| \mid \kappa\in\sigma(N)\}$ and $\nu=\max\{|\Re(\kappa)| \mid \kappa\in\sigma(N),\ \Re(\kappa)<0\}$, with the conventions $\max\emptyset=0$ and $1/0=+\infty$. Note that $\nu\leq L_\tau$.

All algorithms considered below can be written as $z_{k+1}=w_s(z_k)$ with $w_0(z)=z$, and since $F(z^\ast)=0$, they satisfy $w_s(z^\ast)=z^\ast$ for every $s$ for which the update is defined. These are the structural assumptions of Proposition~\ref{prop:saddle_transfer}, and we use them in the proofs below without further mention. For TT-GDA, GEG and TT-PPM, the Jacobians of the $O(1)$- and $O(s)$-resolution ODEs at $z^\ast$ are polynomials in $N$ (see Appendix~\ref{app:o}). The following result, which combines Corollaries~\ref{cor:res_fixed} and~\ref{cor:res_vanishing}, reduces the stability analysis of these first-order DTAs to the signs of two coefficients. For DN and RDN, the Jacobian of the $O(1)$-resolution ODE at $z^\ast$ is Hurwitz at every saddle point with invertible Hessian, and Corollary~\ref{cor:res_fixed} applies directly (see Corollaries~\ref{coro:DN} and~\ref{coro:RDN}).

\begin{proposition}\label{prop:saddle_transfer}
	Let Assumption~\ref{as:inv}(i) hold and let $\tau>0$. Let $U$ be an open neighbourhood of $z^\ast$, let $\bar s>0$, and let $(s,z)\mapsto w_s(z)$ be defined on $(-\bar s,\bar s)\times U$ with $w_0(z)=z$ for all $z\in U$ and $w_s(z^\ast)=z^\ast$ for all $s\in[0,\bar s)$. Let $f_0$ and $f_1$ be the coefficient functions of the $O(s)$-resolution ODE of $z_{k+1}=w_s(z_k)$ from Theorem~\ref{th:ODE}, and suppose that
	\begin{align}\label{eq:poly_jacobian}
		\tfrac{\partial f_0}{\partial z}(z^\ast)=a_1N,\qquad \tfrac{\partial f_1}{\partial z}(z^\ast)=a_2N^2,
	\end{align}
	with constants $a_1>0$ and $a_2\in\R$.
	\begin{itemize}
		\item[(i)] If $(s,z)\mapsto w_s(z)$ is $C^2$, $f_0$ is $C^2$, and $\Re(\kappa)<0$ for all $\kappa\in\sigma(N)$, then there exists $s^\ast>0$ such that $z^\ast$ is a locally exponentially stable fixed point of $z_{k+1}=w_s(z_k)$ for every $s\in(0,s^\ast]$.
		\item[(ii)] If $(s,z)\mapsto w_s(z)$ is $C^3$, $f_0$ and $f_1$ are $C^3$, and $a_2>0$, then the conclusion of (i) holds, including the case where $N$ has eigenvalues on the imaginary axis.
		\item[(iii)] If $f_0$ and $f_1$ are $C^1$, $a_2<0$, and $N$ has an eigenvalue on the imaginary axis, then $z^\ast$ is an unstable equilibrium of the $O(s)$-resolution ODE $\dot z=f_0(z)+sf_1(z)$ for every $s>0$.
	\end{itemize}
\end{proposition}

A proof of Proposition~\ref{prop:saddle_transfer} is given in Appendix~\ref{app:main}. In case (ii), the $O(1)$-resolution ODE $\dot z=f_0(z)$ is not exponentially stable when $N$ has eigenvalues on the imaginary axis, and the decay rate of the $O(s)$-resolution ODE is of order $s$. The proof constructs a quadratic Lyapunov function with $s$-independent constants from Lemma~\ref{negeig}(ii) and applies Corollary~\ref{cor:res_vanishing}. 
This is the regime of Remark~\ref{rem:ce2}. The decay rate is of order $s$, and, as in the proof of Corollary~\ref{cor:res_vanishing}, \eqref{eq:step_condition} is satisfied for small $s$ if $\eta(s)\leq\kappa\rho_1s^2$, that is, if the error is of order $s^2$ with a sufficiently small constant. Proposition~\ref{prop:defect} with $r=1$ gives $\eta(s)=s^2\omega(s)$ with $\omega(s)\to0$, so this holds for all sufficiently small $s$. The $O(1)$-resolution ODE cannot be used here, since it is not exponentially stable at $z^\ast$.
A summary of the findings in this section is given in Table~\ref{tab:sum}.

\begin{remark}[Pointwise statements and common step sizes]\label{rem:pointwise_saddles}
	The step size $s^\ast$ in Proposition~\ref{prop:saddle_transfer} and in the corollaries below depends on the saddle point $z^\ast$. For finitely many saddle points satisfying the respective hypotheses, the minimum of the individual step sizes is a common step size. Saddle points with invertible Hessian are nondegenerate critical points and hence isolated, but they may accumulate at a degenerate critical point, where the admissible step sizes may tend to zero. However, if $\mathcal{K}\subset\R^{n+m}$ is compact and every critical point of $f$ in $\mathcal{K}$ is nondegenerate, then the set of critical points in $\mathcal{K}$ is compact and consists of isolated points, hence it is finite. Consequently, a common step size $s^\ast>0$ can be chosen for all saddle points in such a set $\mathcal{K}$ that satisfy the respective hypotheses.
\end{remark}

\begin{table}[htbp]
	\caption{Local behaviour of the analysed algorithms at a saddle point $z^\ast$ with invertible Hessian for sufficiently small step sizes. Here, $N=\Lambda_\tau H(z^\ast)$ for TT-GDA, GEG and TT-PPM, $N=H(z^\ast)$ for DN, RDN and JM, and LES abbreviates locally exponentially stable.}
	\label{tab:sum}
	\renewcommand{\arraystretch}{1.3}
	\small
	\begin{tabular}{|>{\raggedright\arraybackslash}p{2.0cm}|>{\raggedright\arraybackslash}p{2.1cm}|>{\raggedright\arraybackslash}p{2.1cm}|>{\raggedright\arraybackslash}p{3.8cm}|}
		\hline
		\textbf{Algorithm} & \textbf{$\Re(\kappa)<0$ for all $\kappa\in\sigma(N)$} & \textbf{$\jmath b\in\sigma(N)$ for some $b\neq0$} & \textbf{Remarks} \\ \hline
		TT-GDA & LES & unstable & the $O(s)$ term destabilises \\ \hline
		GEG & LES ($\gamma>0$) & LES ($\gamma\in(0,2)$) & additional gradient evaluation \\ \hline
		TT-PPM & LES & LES & implicit update \\ \hline
		DN & LES & LES & second-order information, every nondegenerate critical point is LES \\ \hline
		RDN, $\phi>0$ & LES & LES & second-order information \\ \hline
		JM & \multicolumn{2}{>{\raggedright\arraybackslash}p{4.6cm}|}{LES if $|\Re(\kappa)|<|\Im(\kappa)|$ for all $\kappa$, unstable if $|\Re(\kappa)|\geq|\Im(\kappa)|$ for some $\kappa$} & unstable for $f(x,y)=x^2-y^2$ \\ \hline
	\end{tabular}
\end{table}

\subsubsection{Two-Timescale Gradient Descent Ascent}

In this section, we investigate the properties of the fixed points of the TT-GDA, which is defined as
\begin{equation}\tag{TT-GDA}
	\begin{aligned}\label{TTGDA}
		x_{k+1} = x_k -\alpha_{x} \nabla_x f(x_k,y_k)\quad \text{and}\quad
		y_{k+1} = y_k +\alpha_{y} \nabla_y f(x_k,y_k),
	\end{aligned}
\end{equation}
where $\alpha_{x}$ and $\alpha_{y}$ are positive step sizes. Let $\alpha_{y}=s$ and $\alpha_{x}=s/\tau$ for positive constants $s$ and $\tau$.  This can be formulated as a discrete-time dynamical system
\begin{align}\label{dyn:TTGDA}
	z_{k+1} = z_k - s\Lambda_\tau F(z_k)
\end{align}
where $z$, $\Lambda_\tau$ and $F$ are defined in \eqref{eq:F_definition}.
The \eqref{TTGDA} algorithm has been studied in previous works, e.g., \cite{jin2020local,chae2023two}.
In \cite[Thm 2.2]{daskalakis2018limit}, it has been proved for GDA (i.e., TT-GDA with $\tau=1$) that with random initialisation, GDA will almost surely escape unstable fixed points.
We can extend this result for the general choice of $\tau.$

\begin{proposition}\label{prop:TTGDA_uns}
	Let $f$ be $C^2$ and let Assumption~\ref{as:inv}(ii) hold. Let $\tau>0$, $\alpha_y=s\in(0,1/L_\tau)$ and $\alpha_x=s/\tau$. Then the set of initial points $z_0\in\R^{n+m}$ for which the iterates of \eqref{TTGDA} converge to  a linearly unstable fixed point has Lebesgue measure zero.
\end{proposition}

A proof of Proposition~\ref{prop:TTGDA_uns} can be found in Appendix~\ref{app:main}.
Next, we analyse the ODE approximations of \eqref{TTGDA}. We use Theorem~\ref{th:ODE} to construct the $O(1)$- and $O(s)$-resolution ODEs of \eqref{TTGDA}, and, following \cite{lu2022sr},
where
$O(1)$- and $O(s)$-resolution ODEs of GDA have been derived, we obtain
the $O(1)$- and the $O(s)$-resolution ODEs
of \eqref{TTGDA}:
\begin{align} \label{o1:TTGDA}
	\dot{z} &= -\Lambda_\tau F(z) \quad \text{and}\\
	\dot{z} &= -\Lambda_\tau F(z)-\tfrac{s}{2}\Lambda_\tau \nabla F(z)\Lambda_\tau F(z) = -(\mathrm{I}+\tfrac{s}{2}\Lambda_\tau\nabla F(z))\Lambda_\tau F(z), \label{os:TTGDA}
\end{align}
respectively.
The details of constructing \eqref{o1:TTGDA} and \eqref{os:TTGDA} can be found in Appendix~\ref{app:o}. In the next theorem, we
analyse the properties of the limit points of \eqref{o1:TTGDA} and \eqref{os:TTGDA}.

\begin{theorem}\label{thm:TTGDA_con}
	Let Assumption~\ref{as:inv} hold, let $f$ be $C^3$, and let $\tau>0$.
	\begin{itemize}
		\item[(i)] If $\Re(\kappa)<0$ for all $\kappa\in\sigma(N)$, then $z^\ast$ is a locally exponentially stable equilibrium of \eqref{o1:TTGDA} and, for every $s\in(0,2\mu/L_\tau^2)$, of \eqref{os:TTGDA}.
		\item[(ii)] If $N$ has an eigenvalue on the imaginary axis, then $z^\ast$ is not exponentially stable for \eqref{o1:TTGDA}, and it is unstable for \eqref{os:TTGDA} for every $s>0$.
	\end{itemize}
\end{theorem}
A proof of Theorem~\ref{thm:TTGDA_con} can be found in Appendix~\ref{app:main}. The $O(1)$-resolution ODE cannot decide stability when $N$ has eigenvalues on the imaginary axis, whereas the $O(s)$-resolution ODE shows that the $O(s)$ term of \eqref{os:TTGDA} is destabilising. The corresponding statements for the DTA are as follows.

\begin{corollary}\label{coro:TTGDA}
	Let Assumption~\ref{as:inv} hold and let $\tau>0$.
	\begin{itemize}
		\item[(i)] If $f$ is $C^3$ and $\Re(\kappa)<0$ for all $\kappa\in\sigma(N)$, then there exists $s^\ast>0$ such that $z^\ast$ is a locally exponentially stable equilibrium of \eqref{dyn:TTGDA} for every $s\in(0,s^\ast]$.
		\item[(ii)] If $N$ has an eigenvalue on the imaginary axis, then $z^\ast$ is an unstable equilibrium of \eqref{dyn:TTGDA} for every $s>0$.
	\end{itemize}
\end{corollary}

Corollary~\ref{coro:TTGDA}(i) states the same conclusion as \cite[Prop.~26]{jin2020local}. 
Corollary~\ref{coro:TTGDA}(ii) is the discrete-time counterpart of Theorem~\ref{thm:TTGDA_con}(ii). The instability of simultaneous gradient methods at equilibria whose linearisation has purely imaginary eigenvalues is well known, see, e.g., \cite{mescheder2018training}.
Next, we analyse the fixed points of the Generalised Extragradient algorithm.

\subsubsection{Generalised Extragradient}

In this section, we investigate the properties of the fixed points of the GEG algorithm, which is defined through the updates
\begin{align}\tag{GEG}
	\begin{aligned}\label{GEG}
		\hat{x}_k &= x_k -\alpha_{1x} \nabla_x f(x_k,y_k)\quad \text{and}\quad
		\hat{y}_k = y_k +\alpha_{1y} \nabla_y f(x_k,y_k)\\
		x_{k+1} &= x_k -\alpha_{2x} \nabla_x f(\hat{x}_k,\hat{y}_k)\quad \text{and}\quad
		y_{k+1} = y_k +\alpha_{2y} \nabla_y f(\hat{x}_k,\hat{y}_k),
	\end{aligned}
\end{align}
where $\alpha_{1x}$, $\alpha_{2x}$, $\alpha_{1y}$ and $\alpha_{2y}$ are positive step sizes.
Here, following \cite{farzin2025properties}, we focus on three degrees of freedom in the step size selection and
define $\gamma=\frac{\alpha_{2x}}{\alpha_{1x}}=\frac{\alpha_{2y}}{\alpha_{1y}}$, $\tau=\frac{\alpha_{1y}}{\alpha_{1x}}=\frac{\alpha_{2y}}{\alpha_{2x}}$ and $s=\alpha_{1y}$.
Considering \eqref{eq:F_definition},
this
can be formulated as
\begin{align}\label{dyn:GEG}
	z_{k+1} = z_k - \gamma s\Lambda_\tau F(z_k-s\Lambda_\tau F(z_k)).
\end{align}
To study the limit point properties of \eqref{GEG}, we focus on stable fixed points of \eqref{GEG}, since with random initialisation, \eqref{GEG} will almost surely escape unstable fixed points. This property has been proved for
proper choice of $\gamma$ and $s$
in \cite[Thm 2]{farzin2025properties}.
Next, we
analyse the ODE approximations of \eqref{GEG}.
In \cite{lu2022sr}, the author constructed the $O(1)$- and $O(s)$-resolution ODEs of GEG for the special case
$\gamma=\tau=1$.
Applying Theorem~\ref{th:ODE} to the general case, the
$O(1)$- and $O(s)$-resolutions
of \eqref{GEG} are
\begin{align}\label{o1:GEG}
	\dot{z} &= -\gamma\Lambda_\tau F(z), \quad \text{and}\\
\begin{split}\label{os:GEG}
	\dot{z} & = -\gamma\Lambda_\tau F(z) +s(\gamma\Lambda_\tau\nabla F(z)\Lambda_\tau F(z)-\tfrac{\gamma^2}{2}\Lambda_\tau\nabla F(z)\Lambda_\tau F(z))\\
	&=(-I+(1-\tfrac{\gamma}{2})s\Lambda_\tau \nabla F(z))\gamma\Lambda_\tau F(z),
\end{split}
\end{align}
respectively.
The details of constructing \eqref{o1:GEG} and \eqref{os:GEG} can be found in Appendix~\ref{app:o}. In the next theorem, we analyse the properties of the limit points of \eqref{o1:GEG} and \eqref{os:GEG}.

\begin{theorem}\label{thm:GEG_con}
	Let Assumption~\ref{as:inv} hold, let $f$ be $C^3$, and let $\tau>0$.
	\begin{itemize}
		\item[(i)] Let $\gamma>0$. If $\Re(\kappa)<0$ for all $\kappa\in\sigma(N)$, then $z^\ast$ is a locally exponentially stable equilibrium of \eqref{o1:GEG}. If $N$ has an eigenvalue on the imaginary axis, then $z^\ast$ is not exponentially stable for \eqref{o1:GEG}.
		\item[(ii)] Let $\gamma\in(0,2)$. Then $z^\ast$ is a locally exponentially stable equilibrium of \eqref{os:GEG} for every $s\in\big(0,1/((1-\tfrac{\gamma}{2})\nu)\big)$, an interval that contains $\big(0,1/((1-\tfrac{\gamma}{2})L_\tau)\big)$. In particular, if all eigenvalues of $N$ lie on the imaginary axis, this holds for every $s>0$.
	\end{itemize}
\end{theorem}
A proof of Theorem~\ref{thm:GEG_con} can be found in Appendix~\ref{app:main}. Unlike for TT-GDA, the $O(s)$ term in \eqref{os:GEG} is stabilising for $\gamma\in(0,2)$, also at saddle points where $N$ has purely imaginary eigenvalues. This leads to the following result for the DTA.

\begin{corollary}\label{coro:GEG}
	Let Assumption~\ref{as:inv} hold and let $\tau>0$.
	\begin{itemize}
		\item[(i)] If $\gamma>0$, $f$ is $C^3$, and $\Re(\kappa)<0$ for all $\kappa\in\sigma(N)$, then there exists $s^\ast>0$ such that $z^\ast$ is a locally exponentially stable equilibrium of \eqref{dyn:GEG} for every $s\in(0,s^\ast]$.
		\item[(ii)] If $\gamma\in(0,2)$ and $f$ is $C^5$, then there exists $s^\ast>0$ such that $z^\ast$ is a locally exponentially stable equilibrium of \eqref{dyn:GEG} for every $s\in(0,s^\ast]$. This includes saddle points at which $N$ has eigenvalues on the imaginary axis.
	\end{itemize}
\end{corollary}

 The regularity assumptions in Corollary~\ref{coro:GEG} are those required by the resolution-ODE route through Proposition~\ref{prop:saddle_transfer}. The stability of GEG fixed points has been studied previously \cite[Thm.~4]{farzin2025properties}. The contribution here is instead to obtain this stabilisation mechanism as a consequence of the $O(s)$
resolution-ODE correction and to place it within the general stability-transfer framework developed above.

\subsubsection{(Two-Timescale) Proximal Point Method}

We investigate the properties of the fixed points of the TT-PPM, whose iterates are defined as
\begin{align}\tag{TT-PPM}
    \begin{split}\label{TTPPM}
		x_{k+1} = x_k -\alpha_{x} \nabla_x f(x_{k+1},y_{k+1})\\
		y_{k+1} = y_k +\alpha_{y} \nabla_y f(x_{k+1},y_{k+1}),
	\end{split}
\end{align}
where $\alpha_{x}$ and $\alpha_{y}$ are positive step sizes. Here, we consider
$\alpha_{y}=s$ and $\alpha_{x}=s/\tau$ for
positive constants $s,\tau \in \mathbb{R}_{>0}$.
Then, \ref{TTPPM}
can be formulated
as
\begin{align}
	z_{k+1} = z_k - s\Lambda_\tau F(z_{k+1})
\qquad \text{or} \qquad
\label{dyn:TTPPM}
	z_{k+1} = (I+s\Lambda_\tau F)^{-1}(z_k)
\end{align}
where $\Lambda_\tau$ and $F$ are defined in \eqref{eq:F_definition} and $(I+s\Lambda_\tau F)^{-1}(\cdot)$ denotes the inverse mapping of $z\mapsto z+s\Lambda_\tau F(z)$.
Under Assumption~\ref{as:inv}(ii), for any fixed \(z\in\mathbb{R}^{n+m}\), define
$
T_z(u):=z-s\Lambda_\tau F(u).
$
Then, for any \(u,v\in\mathbb{R}^{n+m}\),
$
|T_z(u)-T_z(v)|
\le |s|L_\tau|u-v|.
$
Hence, if \(|s|<1/L_\tau\), \(T_z\) is a contraction with constant \(|s|L_\tau<1\). Since \(\mathbb{R}^{n+m}\) is complete, Banach's fixed-point theorem implies that \(T_z\) has a unique fixed point for every \(z\in\mathbb{R}^{n+m}\). Consequently, \eqref{dyn:TTPPM} is well defined on \(\mathbb{R}^{n+m}\).
The contraction argument is used here only to guarantee that the implicit update is well defined for sufficiently small $s.$ The stability conclusions below are obtained independently through the resolution-ODE analysis.
The dynamics of the proximal point method for nonconvex-nonconcave min-max problems are studied in \cite{grimmer2023landscape}, where convergence is characterised in terms of the interaction between the two players. Here, we focus on the local stability of saddle points for the two-timescale variant.
To study the limiting properties, we focus on stable fixed points of~\eqref{TTPPM}. First, we show that with random initialisation, \eqref{TTPPM} will almost surely escape unstable fixed points.

\begin{proposition}\label{prop:TTPPM_uns}
	Let $f$ be $C^2$ and let Assumption~\ref{as:inv}(ii) hold. Let $\tau>0$, $\alpha_y=s\in(0,1/L_\tau)$ and $\alpha_x=s/\tau$. Then the set of initial points $z_0\in\R^{n+m}$ for which the iterates of \eqref{TTPPM} converge to  a linearly unstable fixed point has Lebesgue measure zero.
\end{proposition}

A proof of Proposition~\ref{prop:TTPPM_uns} can be found in Appendix~\ref{app:main}. Proposition~\ref{prop:TTPPM_uns} shows that the set of initial points from which the iterates converge to  a linearly unstable fixed point has Lebesgue measure zero.
Next, we will analyse the ODE approximations of \eqref{TTPPM}. We can use Theorem~\ref{th:ODE} to construct the $O(1)$- and $O(s)$-resolution ODEs of \eqref{TTPPM}, which are given by
\begin{align}\label{o1:TTPPM}
	\dot{z} &= -\Lambda_\tau F(z) \qquad \text{and} \\
\begin{split}\label{os:TTPPM}
	\dot{z} & = -\Lambda_\tau F(z) +s(\Lambda_\tau\nabla F(z)\Lambda_\tau F(z)-\tfrac{1}{2}\Lambda_\tau\nabla F(z)\Lambda_\tau F(z)) \\
	&=(-I+\tfrac{s}{2}\Lambda_\tau \nabla F(z))\Lambda_\tau F(z),
\end{split}
\end{align}
respectively.
The details of constructing \eqref{o1:TTPPM} and \eqref{os:TTPPM} can be found in Appendix~\ref{app:o}.
A special case of these constructions restricted to $\tau=1$ can be found in \cite{lu2022sr}.

\begin{theorem}\label{thm:TTPPM_con}
	Let Assumption~\ref{as:inv} hold, let $f$ be $C^3$, and let $\tau>0$.
	\begin{itemize}
		\item[(i)] If $\Re(\kappa)<0$ for all $\kappa\in\sigma(N)$, then $z^\ast$ is a locally exponentially stable equilibrium of \eqref{o1:TTPPM}. If $N$ has an eigenvalue on the imaginary axis, then $z^\ast$ is not exponentially stable for \eqref{o1:TTPPM}.
		\item[(ii)] The point $z^\ast$ is a locally exponentially stable equilibrium of \eqref{os:TTPPM} for every $s\in(0,2/\nu)$, an interval that contains $(0,2/L_\tau)$. In particular, if all eigenvalues of $N$ lie on the imaginary axis, this holds for every $s>0$.
	\end{itemize}
\end{theorem}
A proof of Theorem~\ref{thm:TTPPM_con} can be found in Appendix~\ref{app:main}. As for GEG, the $O(s)$ term in \eqref{os:TTPPM} is stabilising.

\begin{corollary}\label{coro:TTPPM}
	Let Assumption~\ref{as:inv} hold and let $\tau>0$.
	\begin{itemize}
		\item[(i)] If $f$ is $C^3$ and $\Re(\kappa)<0$ for all $\kappa\in\sigma(N)$, then there exists $s^\ast>0$ such that $z^\ast$ is a locally exponentially stable equilibrium of \eqref{dyn:TTPPM} for every $s\in(0,s^\ast]$.
		\item[(ii)] If $f$ is $C^5$, then there exists $s^\ast>0$ such that $z^\ast$ is a locally exponentially stable equilibrium of \eqref{dyn:TTPPM} for every $s\in(0,s^\ast]$, including saddle points at which $N$ has eigenvalues on the imaginary axis.
	\end{itemize}
\end{corollary}


 \subsubsection{Damped Newton}
 In this section, we investigate the properties of the fixed points of
DN, which can be formulated as a discrete-time dynamical system as
\begin{align}\tag{DN}\label{dyn:DN}
	z_{k+1} = z_k - s (\nabla F(z_k))^{-1} F(z_{k}),
\end{align}
where $F$ is defined in \eqref{eq:F_definition}.
Since $\nabla F(z)=\diag(I,-I)\nabla^2f(z)$, the update \eqref{dyn:DN} is well defined wherever $\nabla^2 f(z)$ is invertible. Under Assumption~\ref{as:inv}(i), this holds on a neighbourhood of $z^\ast$, whereas global well-posedness requires $\nabla^2f(z)$ to be invertible for all $z$. Wherever $\nabla F(z)$ is invertible, we write $d(z)=(\nabla F(z))^{-1}F(z)$ for the Newton direction and $Q(z)$ for the matrix of the linear map $v\mapsto(\nabla F(z))^{-1}\big(\mathrm{D}\nabla F(z)[v]\big)d(z)$, where $\mathrm{D}\nabla F(z)[v]$ denotes the derivative of $\nabla F$ at $z$ in the direction $v$. Differentiating $\nabla F(z)d(z)=F(z)$ gives
\begin{align}\label{eq:DN_newton_jac}
	\tfrac{\partial d}{\partial z}(z)=I-Q(z).
\end{align}
If $f$ is $C^3$, then, at every critical point $\bar z$ with invertible Hessian, $d(\bar z)=0$, hence $Q(\bar z)=0$ and $\frac{\partial w_s}{\partial z}(\bar z)=(1-s)I$. Consequently, for $s\in(0,2)$, every nondegenerate critical point of $f$, whether a saddle point or not, is a locally exponentially stable fixed point of DN, and DN has no unstable fixed points with invertible Hessian. DN therefore does not distinguish saddle points from other nondegenerate critical points, which reflects the well-known attraction of Newton-type methods to arbitrary critical points, and an analogue of Proposition~\ref{prop:TTPPM_uns} does not provide information for DN.

With Theorem~\ref{th:ODE}, the $O(1)$- and $O(s)$-resolution ODEs of \eqref{dyn:DN} are
\begin{align}\label{o1:DN}
	\dot{z} &= -d(z) \qquad \text{and} \\
	\dot{z} &= -d(z)+\tfrac{s}{2}\big(Q(z)-I\big)d(z)=-\big(1+\tfrac{s}{2}\big)d(z)+\tfrac{s}{2}Q(z)d(z), \label{os:DN}
\end{align}
respectively. The details of the construction of \eqref{o1:DN} and \eqref{os:DN} can be found in Appendix~\ref{app:o}.

\begin{theorem}\label{thm:DN_con}
	Let Assumption~\ref{as:inv}(i) hold and let $f$ be $C^4$. Then $z^\ast$ is a locally exponentially stable equilibrium of \eqref{o1:DN} and, for every $s>0$, of \eqref{os:DN}.
\end{theorem}
A proof of Theorem~\ref{thm:DN_con} can be found in Appendix~\ref{app:main}.

\begin{corollary}\label{coro:DN}
	Let Assumption~\ref{as:inv}(i) hold and let $f$ be $C^4$. Then there exists $s^\ast>0$ such that $z^\ast$ is a locally exponentially stable equilibrium of \eqref{dyn:DN} for every $s\in(0,s^\ast]$. If $f$ is $C^3$, this holds for every $s\in(0,2)$.
\end{corollary}

The observation that damping yields the linearisation
$(1-s)I$ at a nondegenerate critical point is classical in the analysis
of damped Newton-type methods. Here it serves as an additional example
of how the resolution-ODE viewpoint places optimisation algorithms with
very different update structures under the same stability framework.

 \subsubsection{Regularised Damped Newton}
In this section, we investigate the properties of the fixed points of the RDN, which can be formulated as a discrete-time dynamical system as
\begin{align}\tag{RDN} \label{dyn:RDN}
	z_{k+1} = z_k - s(\nabla F(z_k)+\phi_k I)^{-1}F(z_k),
\end{align}
where $F$ is defined in \eqref{eq:F_definition}.

Throughout, $\phi_k=\phi>0$ is fixed. By Lemma~\ref{negeig} with $\tau=1$, the eigenvalues of $H(z^\ast)=-\nabla F(z^\ast)$ have non-positive real parts, so $\nabla F(z^\ast)+\phi I=\phi I-H(z^\ast)$ is invertible and RDN is well defined near $z^\ast$ for every $\phi>0$. If Assumption~\ref{as:inv}(ii) holds and $\phi>L$, then $\nabla F(z)+\phi I$ is invertible for all $z$ with $\|(\nabla F(z)+\phi I)^{-1}\|\leq(\phi-L)^{-1}$. Wherever $\nabla F(z)+\phi I$ is invertible, we write $d_\phi(z)=(\nabla F(z)+\phi I)^{-1}F(z)$ and $Q_\phi(z)$ for the matrix of the linear map $v\mapsto(\nabla F(z)+\phi I)^{-1}\big(\mathrm{D}\nabla F(z)[v]\big)d_\phi(z)$, so that
\begin{align}\label{eq:RDN_newton_jac}
	\tfrac{\partial d_\phi}{\partial z}(z)=(\nabla F(z)+\phi I)^{-1}\nabla F(z)-Q_\phi(z).
\end{align}

\begin{proposition}\label{prop:RDN_uns}
		Let $f$ be $C^3$ and let Assumption~\ref{as:inv}(ii) hold. Suppose that there exists $G\geq0$ with $\|\nabla^3f(z)\|\,|\nabla f(z)|\leq G$ for all $z$, and let $\phi>L$ and
	\begin{align*}
		q_\phi=\frac{L}{\phi-L}+\frac{G}{(\phi-L)^2}.
	\end{align*}
	Then, for every $s\in(0,1/q_\phi)$, the set of initial points $z_0\in\R^{n+m}$ for which the iterates of \eqref{dyn:RDN} converge to  a linearly unstable fixed point has Lebesgue measure zero. Moreover, every critical point $\bar z$ such that $H(\bar z)$ has a real eigenvalue in $(0,\phi)$ is  a linearly unstable fixed point of \eqref{dyn:RDN}.
\end{proposition}
A proof of Proposition~\ref{prop:RDN_uns} can be found in Appendix~\ref{app:main}.
With Theorem~\ref{th:ODE}, the $O(1)$- and $O(s)$-resolution ODEs of \eqref{dyn:RDN} are
\begin{align}\label{o1:RDN}
	\dot{z} &= -d_\phi(z) \qquad \text{and} \\
	\dot{z} &= -d_\phi(z)-\tfrac{s}{2}\Big((\nabla F(z)+\phi I)^{-1}\nabla F(z)-Q_\phi(z)\Big)d_\phi(z), \label{os:RDN}
\end{align}
respectively. The details of constructing \eqref{o1:RDN} and \eqref{os:RDN} can be found in Appendix~\ref{app:o}.

\begin{theorem}\label{thm:RDN_con}
	Let Assumption~\ref{as:inv}(i) hold, let $f$ be $C^4$, and let $\phi>0$. Then $z^\ast$ is a locally exponentially stable equilibrium of \eqref{o1:RDN} and, for every $s\in(0,2)$, of \eqref{os:RDN}.
\end{theorem}
A proof of Theorem~\ref{thm:RDN_con} can be found in Appendix~\ref{app:main}. No lower bound on the regularisation parameter $\phi$ is needed for the local analysis, and the condition $\phi>L$ is only used for global well-posedness in Proposition~\ref{prop:RDN_uns}.

\begin{corollary}\label{coro:RDN}
	Let Assumption~\ref{as:inv}(i) hold, let $f$ be $C^4$, and let $\phi>0$. Then there exists $s^\ast>0$ such that $z^\ast$ is a locally exponentially stable equilibrium of \eqref{dyn:RDN} for every $s\in(0,s^\ast]$.
\end{corollary}
A proof of Corollary~\ref{coro:RDN} can be found in Appendix~\ref{app:main}.

\begin{remark}[Jacobian method]\label{rem:JM}
	The Jacobian method (JM) was introduced in \cite[Sec.~2.3]{lu2022sr} and is given by $z_{k+1}=z_k+s\nabla F(z_k)F(z_k)$. Its $O(1)$-resolution ODE is $\dot z=\nabla F(z)F(z)$ \cite{lu2022sr}. Let $f$ be $C^3$ and let Assumption~\ref{as:inv}(i) hold. Since $F(z^\ast)=0$ and $\nabla F(z^\ast)=-H(z^\ast)$, the Jacobian of the update map at $z^\ast$ is $I+sH(z^\ast)^2$, and the Jacobian of the $O(1)$-resolution ODE at $z^\ast$ is $H(z^\ast)^2$. For $\kappa=a+\jmath b\in\sigma(H(z^\ast))$, we have $\Re(\kappa^2)=a^2-b^2$ and $|1+s\kappa^2|^2=1+2s(a^2-b^2)+s^2|\kappa|^4$. Hence, by Theorem~\ref{thm:stability_of_equilibria}, if $|\Re(\kappa)|<|\Im(\kappa)|$ for all $\kappa\in\sigma(H(z^\ast))$, then $z^\ast$ is locally exponentially stable for the $O(1)$-resolution ODE and, for every $s\in\big(0,\min_{\kappa\in\sigma(H(z^\ast))}2(b^2-a^2)/|\kappa|^4\big)$, for JM. Since $\kappa\neq0$ by Assumption~\ref{as:inv}(i), if $|\Re(\kappa)|\geq|\Im(\kappa)|$ for some $\kappa$, then $|1+s\kappa^2|^2\geq1+s^2|\kappa|^4>1$, and $z^\ast$ is unstable for JM for every $s>0$. If $|\Re(\kappa)|>|\Im(\kappa)|$ for some $\kappa$, then $H(z^\ast)^2$ has an eigenvalue with positive real part, and $z^\ast$ is also unstable for the $O(1)$-resolution ODE. In the case $|\Re(\kappa)|=|\Im(\kappa)|$, the eigenvalue $\kappa^2=2ab\jmath$ of $H(z^\ast)^2$ is purely imaginary, and linearisation is inconclusive for the $O(1)$-resolution ODE. This is a restrictive limitation. For example, for the strongly convex-strongly concave function $f(x,y)=x^2-y^2$, we have $H(0)=-2I$, so the saddle point $(0,0)$ is unstable for JM for every $s>0$.
\end{remark}

\begin{remark}[Relation to direct linearisation]\label{rem:novelty41}
	Under Assumption~\ref{as:inv}(i), the conclusions of this section can also be obtained by computing the spectrum of $\frac{\partial w_s}{\partial z}(z^\ast)$ directly, and some of them are known. Corollary~\ref{coro:TTGDA} states the same conclusion as \cite[Prop.~26]{jin2020local}, and Corollary~\ref{coro:GEG} recovers \cite[Thm.~4]{farzin2025properties}. The instability of gradient descent ascent at equilibria whose linearisation has purely imaginary eigenvalues is known \cite{mescheder2018training}, and the spectra of extragradient and proximal point updates are analysed in \cite{azizian2020tight}. Proposition~\ref{prop:saddle_transfer} makes this mechanism rigorous at every saddle point with an invertible Hessian. The resolution ODEs explain the mechanism behind these results. The three first-order methods have the same $O(1)$-resolution ODE up to the factor $a_1$, which cannot distinguish them at saddle points where $N$ has eigenvalues on the imaginary axis. Their stability is then decided by the sign of the coefficient $a_2$ of the $O(s)$ term in Proposition~\ref{prop:saddle_transfer}, namely $a_2=-\frac12$ for TT-GDA, $a_2=\gamma(1-\frac\gamma2)$ for GEG, and $a_2=\frac12$ for TT-PPM. The main advantage of the framework in min-max optimisation, however, appears at degenerate saddle points, where linearisation is inconclusive. This case is treated next.
\end{remark}

\subsection{Min-Max Problems with Rank Deficient Hessians}\label{sec:gen}

The results in Section~\ref{sec:main_p1} rely on Assumption~\ref{as:inv}(i). This assumption is common in the literature, for example \cite[As.~1.7]{daskalakis2018limit}, \cite[As.~$2^\prime$]{chae2023two}, and \cite{jin2020local}, where only local saddle points with $\nabla^2_{yy} f(z^\ast)\prec0$ and $\nabla^2_{xx} f(z^\ast)\succ0$ are considered. If $\nabla^2f(z^\ast)$ is singular, then $0\in\sigma(N)$, so the linearisations at $z^\ast$ of the $O(1)$-resolution ODEs and of the first-order DTAs have the eigenvalues $0$ and $1$, respectively, and linearisation is inconclusive. In this section, we use Theorem~\ref{thm:stab1_A} and Corollary~\ref{cor:rates} to establish asymptotic stability, together with convergence rates, of degenerate saddle points and of compact sets of critical points. Recall that $\mathcal{B}_\delta(\mathcal{A})=\{z \mid |z|_\mathcal{A}\leq\delta\}$.

\begin{theorem}\label{thm:degenerate}
	Let $f:\R^{n+m}\to\R$ be $C^2$, let $\tau>0$, and let $\mathcal{A}\subset\R^{n+m}$ be nonempty and compact with $F(z)=0$ for all $z\in\mathcal{A}$. Let $\delta>0$ be such that $F(z)\neq0$ for all $z\in\mathcal{B}_\delta(\mathcal{A})\setminus\mathcal{A}$,  and let $\psi:U\to\R_{\geq0}$ be $C^2$ on an open set $U\supset\mathcal{B}_{2\delta}(\mathcal{A})$. Suppose that
	\begin{itemize}
		\item[(a)] there exist $\alpha_1,\alpha_2\in\mathcal{K}_\infty$ with $\alpha_1(|z|_\mathcal{A})\leq\psi(z)\leq\alpha_2(|z|_\mathcal{A})$ for all $z\in\mathcal{B}_{2\delta}(\mathcal{A})$,
		\item[(b)] there exists $\beta>0$ with $\langle\nabla\psi(z),\Lambda_\tau F(z)\rangle\geq\beta|F(z)|^2$ for all $z\in\mathcal{B}_\delta(\mathcal{A})$,
		\item[(c)] there exists $c_\psi\geq0$ with $|\nabla\psi(z)|\leq c_\psi|F(z)|$ for all $z\in\mathcal{B}_\delta(\mathcal{A})$.
	\end{itemize}
	Consider \eqref{dyn:TTGDA}, \eqref{dyn:GEG} with $\gamma>0$, and \eqref{dyn:TTPPM}, 
    where the implicit update of \eqref{dyn:TTPPM} is understood as the unique solution in the ball $ {\mathcal{B}_\delta(z_k)}$, and set $\gamma=1$ for \eqref{dyn:TTGDA} and \eqref{dyn:TTPPM}. For each of these algorithms, there exists $s^\ast>0$ such that, for every $s\in(0,s^\ast]$, the set $\mathcal{A}$ is locally asymptotically stable and
	\begin{align}\label{eq:degenerate_decrease}
		\psi(z_{k+1})\leq\psi(z_k)-\tfrac{s\gamma\beta}{2}|F(z_k)|^2\qquad\text{whenever } z_k\in\mathcal{B}_\delta(\mathcal{A}) .
	\end{align}
	If, in addition, $|F(z)|^2\geq c_F\,\psi(z)^p$ for all $z\in\mathcal{B}_\delta(\mathcal{A})$ with constants $c_F>0$ and $p\geq1$, then the convergence rates of Corollary~\ref{cor:rates} hold with $V=\psi$ and $c=\gamma\beta c_F$. For \eqref{dyn:TTGDA}, assumption (c) is not needed.
\end{theorem}

A proof of Theorem~\ref{thm:degenerate} is given in Appendix~\ref{app:main}. Assumptions (b) and (c) require $\psi$ to be a Lyapunov function whose gradient is aligned with $F$ and comparable to it in size near $\mathcal{A}$, and no rate is imposed on the decrease. An important class of problems satisfying these assumptions are separable objective functions.

\begin{corollary}\label{cor:separable}
	Let $g\in C^2(\R^n)$ and $h\in C^2(\R^m)$, let $f(x,y)=g(x)-h(y)$, and let $\tau>0$. Let $\mathcal{A}_x\subset\R^n$ and $\mathcal{A}_y\subset\R^m$ be nonempty and compact, and let $\delta>0$. Suppose that $g\equiv g^\ast$ on $\mathcal{A}_x$ and $h\equiv h^\ast$ on $\mathcal{A}_y$ for some constants $g^\ast,h^\ast$, that $g(x)>g^\ast$ and $\nabla g(x)\neq0$ for all $x$ with $0<|x|_{\mathcal{A}_x}\leq3\delta$, and that $h(y)>h^\ast$ and $\nabla h(y)\neq0$ for all $y$ with $0<|y|_{\mathcal{A}_y}\leq3\delta$. Then every point of $\mathcal{A}=\mathcal{A}_x\times\mathcal{A}_y$ is a local saddle point of $f$, and the conclusions of Theorem~\ref{thm:degenerate} hold with $\psi(x,y)=g(x)-g^\ast+h(y)-h^\ast$, $\beta=\min\{1,\tau^{-1}\}$ and $c_\psi=1$.
\end{corollary}

A proof of Corollary~\ref{cor:separable} is given in Appendix~\ref{app:main}. The Hessian of $f$ may be singular on $\mathcal{A}$, so Corollary~\ref{cor:separable} applies to degenerate saddle points and to sets of saddle points. For separable $f$, the three algorithms decouple into the corresponding methods applied to $g$ and $h$ separately, so Corollary~\ref{cor:separable} mainly serves as an illustration. Theorem~\ref{thm:degenerate} itself does not require separability. We illustrate these results with three examples. The first, $f(x,y)=x^2-y^4$, has an isolated degenerate saddle point and is separable, so Corollary~\ref{cor:separable} applies directly and yields asymptotic stability together with an explicit sublinear convergence rate. This example also shows the limitations of second-order methods at degenerate saddle points, since the update of DN is not defined on a whole line through the saddle point. The second example also has an isolated degenerate saddle point, but it is not separable, and Theorem~\ref{thm:degenerate} applies with an explicit Lyapunov function. The third example has a compact set of critical points whose interior consists of local saddle points at which the Hessian vanishes. It is not separable and assumption~(c) of Theorem~\ref{thm:degenerate} fails, so we verify the compatibility condition \eqref{eq:compatibility} of Theorem~\ref{thm:stab1_A} directly for TT-GDA and GEG, and we treat TT-PPM with Proposition~\ref{prop:implicit}. This illustrates how the general results of Section~\ref{sec:stab} can be used when the structural assumptions of Theorem~\ref{thm:degenerate} are not available.

\subsubsection{Objective function \texorpdfstring{$f(x,y)=x^2-y^4$}{f(x,y) = x\textasciicircum 2 - y\textasciicircum 4}}\label{sec:x2y4}

Consider the $C^\infty$ objective function $f(x,y)=x^2-y^4$ with saddle point $z^\ast=(0,0)$. Then
\begin{align*}
    \nabla f(z) = \begin{bmatrix}
        2x\\-4y^3
    \end{bmatrix},\quad
    \nabla^2f(z)=\begin{bmatrix}
    \begin{array}{cc}
		2 & 0  \\
		0 & -12y^2
	\end{array}
    \end{bmatrix},
\end{align*}
so $\nabla^2f(z^\ast)$ is singular and the results of Section~\ref{sec:main_p1} do not apply. It can be observed numerically that TT-GDA, GEG, and TT-PPM converge to $z^\ast$ for sufficiently small step sizes. This behaviour is explained by Corollary~\ref{cor:separable} with $g(x)=x^2$, $h(y)=y^4$ and $\mathcal{A}_x=\mathcal{A}_y=\{0\}$, whose assumptions hold for every $\delta>0$. Hence, for every $\tau>0$ (and every $\gamma>0$ for GEG), there exists $s^\ast>0$ such that $z^\ast$ is locally asymptotically stable for these algorithms for all $s\in(0,s^\ast]$.

To obtain a rate, note that $\psi(z)=x^2+y^4$ and $|F(z)|^2=4x^2+16y^6$. Let $\delta>0$. For $z\in\mathcal{B}_\delta(\mathcal{A})$, we have $|x|\leq\delta$, hence $4x^2\geq\frac{4}{\delta}(x^2)^{3/2}$, and $16y^6=16(y^4)^{3/2}$. Convexity of $t\mapsto t^{3/2}$ gives $a^{3/2}+b^{3/2}\geq2^{-1/2}(a+b)^{3/2}$ for $a,b\geq0$, and therefore
\begin{align*}
	|F(z)|^2\geq c_F(\delta)\,\psi(z)^{3/2},\qquad c_F(\delta)=2^{-1/2}\min\Big\{\frac{4}{\delta},16\Big\},\qquad z\in\mathcal{B}_\delta(\mathcal{A}) .
\end{align*}
Theorem~\ref{thm:degenerate} and Corollary~\ref{cor:rates} with $p=\tfrac32$ and $c=\gamma\beta c_F(\delta)$, where $\beta=\min\{1,\tau^{-1}\}$, yield $s^\ast(\delta)>0$ such that
\begin{align*}
	\psi(z_k)\leq\Big(\psi(z_0)^{-1/2}+\tfrac14\gamma\beta c_F(\delta)\,s\,k\Big)^{-2}
\end{align*}
for all $k\in\N$, all $s\in(0,s^\ast(\delta)]$ and all $z_0\in\mathcal{B}_\delta(\mathcal{A})$ with $0<\psi(z_0)\leq\ell$, where $0<\ell<\alpha_1(\delta)$ as in Corollary~\ref{cor:rates}. Since $\delta>0$ is arbitrary and the sublevel sets of $\psi$ are bounded, every bounded set of initial conditions is covered for sufficiently small step sizes, and only the constants $s^\ast(\delta)$ and $c_F(\delta)$ depend on the size of this set.  
The same arguments apply to $f(x,y)=x^4-y^4$ with $\psi(z)=x^4+y^4$, where $|F(z)|^2=16\big((x^4)^{3/2}+(y^4)^{3/2}\big)\geq2^{7/2}\psi(z)^{3/2}$ holds on all of $\R^2$, so that only $s^\ast(\delta)$ depends on $\delta$. For DN, $\nabla F(z)=\diag(2,12y^2)$ is singular on the whole line $\{y=0\}$, so \eqref{dyn:DN} is not defined there, and neither Section~\ref{sec:main_p1} nor Theorem~\ref{thm:degenerate} applies.


\subsubsection{A non-separable degenerate saddle point}\label{sec:nonsep}
Consider the $C^\infty$ objective function
\begin{align*}
	f(x,y)=\tfrac12x^2+bxy^3-\tfrac14y^4,\qquad b\in\R .
\end{align*}
For $b\neq0$, the coupling term $bxy^3$ prevents the algorithms from decoupling, so Corollary~\ref{cor:separable} does not apply. We have $F(z)=(x+by^3,\,y^3-3bxy^2)^\top$, $f(0,y)=-\tfrac14y^4\leq0=f(0,0)\leq\tfrac12x^2=f(x,0)$ for all $x,y\in\R$, and $\nabla^2f(0,0)=\diag(1,0)$. Hence $z^\ast=(0,0)$ is a degenerate saddle point.
Writing $v(z)=(x,y^3)^\top$, we have $F(z)=M(y)v(z)$ with $M(y)=\left[\begin{smallmatrix}1&b\\-3by^2&1\end{smallmatrix}\right]$ and $\det M(y)=1+3b^2y^2\geq1$.

\begin{proposition}\label{prop:nonsep}
		Let $\tau>0$, $b\in\R$ and $c_y>b^2/4$, and let $\psi(x,y)=\tfrac{\tau}{2}x^2+\tfrac{c_y}{4}y^4$ and $\mathcal{A}=\{z^\ast\}$. Then there exist $\delta>0$, $\beta>0$ and $c_\psi>0$ such that the assumptions of Theorem~\ref{thm:degenerate} hold, and $|F(z)|^2\geq c_F\,\psi(z)^{3/2}$ for all $z\in\mathcal{B}_\delta(\mathcal{A})$ with
	\begin{align*}
		c_F=\frac{1}{\sqrt2\,c_M^2}\min\Big\{\frac{(2/\tau)^{3/2}}{\delta},\Big(\frac{4}{c_y}\Big)^{3/2}\Big\},\qquad c_M=\max_{|y|\leq\delta}\|M(y)^{-1}\| .
	\end{align*}
	Consequently, for \eqref{dyn:TTGDA}, \eqref{dyn:GEG} with $\gamma>0$, and \eqref{dyn:TTPPM}, with $\gamma=1$ for \eqref{dyn:TTGDA} and \eqref{dyn:TTPPM}, there exists $s^\ast>0$ such that, for every $s\in(0,s^\ast]$, the saddle point $z^\ast$ is locally asymptotically stable and
	\begin{align*}
		\psi(z_k)\leq\Big(\psi(z_0)^{-1/2}+\tfrac14\gamma\beta c_F\,s\,k\Big)^{-2}\qquad\forall\, k\in\N
	\end{align*}
	for all $z_0\in\mathcal{B}_\delta(\mathcal{A})$ with $0<\psi(z_0)\leq\ell$, where $0<\ell<\alpha_1(\delta)$ as in Corollary~\ref{cor:rates}.
\end{proposition}

A proof of Proposition~\ref{prop:nonsep} is given in Appendix~\ref{app:main}, where $\beta$ and $c_\psi$ are given explicitly.

\subsubsection{Convergence to a compact attractor}
\label{sec:cas}

In this section, we consider the objective function $f(x,y)=\chi(r)(x^2-y^2)$, where $r=\sqrt{x^2+y^2}$, $R,\epsilon>0$, and
\begin{align*}
	\chi(r)=\left\{
	\begin{array}{ll}
		0 & r\leq R,\\
		\vartheta\big(\frac{r-R}{\epsilon}\big) & R\leq r\leq R+\epsilon,\\
		1 & r\geq R+\epsilon,
	\end{array}
	\right.
	\qquad \vartheta(p)=3p^2-2p^3 .
\end{align*}
In the numerical experiments, $R=1$ and $\epsilon=0.2$. Since $\vartheta(0)=\vartheta'(0)=\vartheta'(1)=0$ and $\vartheta(1)=1$, the function $\chi$ is $C^1$ with derivative $\chi'(r)=\epsilon^{-1}\vartheta'\big(\frac{r-R}{\epsilon}\big)$ for $R\leq r\leq R+\epsilon$ and $\chi'(r)=0$ otherwise, where $\vartheta'(p)=6p-6p^2$. Hence $f$ is $C^1$ with a globally Lipschitz continuous gradient, but $f$ is not $C^2$ at $r=R$ and $r=R+\epsilon$. We have
\begin{align*}
    \nabla_x f(x,y) = 2x\chi(r) + \chi^\prime(r)\frac{x}{r}(x^2-y^2)\;\text{and}\;\nabla_y f(x,y) = -2y\chi(r) + \chi^\prime(r)\frac{y}{r}(x^2-y^2),
\end{align*}
so that $F(z)=2\chi(r)z+\chi'(r)\frac{x^2-y^2}{r}(x,-y)^\top$ for $z\neq0$. The gradient of $f$ vanishes on $\mathcal{A}= {\mathcal{B}_R(0)}$, and every point $z$ with $|z|<R$ is a local saddle point of $f$ at which $\nabla^2f$ vanishes. Points with $|z|=R$ are critical points of $f$, but not all of them are saddle points, since, for example, $f(R,y)>0=f(R,0)$ for all sufficiently small $y\neq0$. Assumption (c) of Theorem~\ref{thm:degenerate} fails for $\psi(z)=|z|_\mathcal{A}^2$ near the diagonals $x^2=y^2$, where $|\nabla\psi(z)|=2|z|_\mathcal{A}$ while $|F(z)|=2\chi(r)r=O(|z|_\mathcal{A}^2)$. We therefore apply Theorem~\ref{thm:stab1_A} directly to TT-GDA and GEG, and Proposition~\ref{prop:implicit} to TT-PPM.

\begin{proposition}\label{prop:bump}
	Let $f$ be as above, let $0<\delta\leq\epsilon/2$, and let $L>0$ be a Lipschitz constant of $F$ on $\R^2$.
	\begin{itemize}
		\item[(i)] Let $\gamma>0$, and consider \eqref{dyn:TTGDA} and \eqref{dyn:GEG} with $\tau=1$, where $\gamma=1$ for \eqref{dyn:TTGDA}. There exists $s^\ast>0$ such that, for every $s\in(0,s^\ast]$, the set $\mathcal{A}=\mathcal{B}_R(0)$ is locally asymptotically stable, and every solution with $0<|z_0|_\mathcal{A}<\delta$ satisfies
		\begin{align*}
			|z_k|_\mathcal{A}\leq\Big(|z_0|_\mathcal{A}^{-1}+\frac{\gamma R\,s\,k}{\epsilon^2}\Big)^{-1}\qquad\forall\, k\in\N .
		\end{align*}
		\item[(ii)] Consider \eqref{dyn:TTPPM} with $\tau=1$. For every $s\in(0,1/L)$, the set $\mathcal{A}$ is locally asymptotically stable, every solution converges to $\mathcal{A}$, and every solution with $0<|z_0|_\mathcal{A}\leq\epsilon$ satisfies
		\begin{align*}
			|z_k|_\mathcal{A}\leq\Big(|z_0|_\mathcal{A}^{-1}+\frac{2R\,s\,k}{c_0\,\epsilon^2}\Big)^{-1}\qquad\forall\, k\in\N,
		\end{align*}
		where $c_0=\tfrac12(1+b_0)(2+b_0)$ and $b_0=4Rs|z_0|_\mathcal{A}/\epsilon^2$.
	\end{itemize}
	In both cases, if $z_0\in\mathcal{A}$, then $z_k=z_0$ for all $k\in\N$.
\end{proposition}

A proof of Proposition~\ref{prop:bump} is given in Appendix~\ref{app:main}. Besides assumption~(c), Theorem~\ref{thm:degenerate} also fails to apply because $f$ and $\psi=|z|_\mathcal{A}^2$ are not $C^2$.

\section{Conclusion and Future Work}\label{sec:conc}
We have established conditions under which stability properties of a continuous-time system are inherited by a discrete-time system. For exponential stability of an equilibrium, the one-step error must be small relative to the distance to the equilibrium and to the decrease of a Lyapunov function over one step, and we have shown that this condition is sharp in order. The result applies to ODEs that depend on the step size, including $O(s)$-resolution ODEs whose decay rate vanishes as the step size tends to zero. For asymptotic stability of equilibria and compact sets, which is not inherited in general, we have provided a comparison criterion and a checkable compatibility condition, together with the transfer of sublinear convergence rates. For resolution ODEs, joint regularity of the update map guarantees the required error bound.
The theoretical results have been applied to min-max algorithms, and we have shown that, for sufficiently small step sizes and proper choice of parameters, saddle points of \eqref{eq:eq20} with invertible Hessian are locally exponentially stable fixed points of GEG, TT-PPM, DN, and RDN. In contrast, such saddle points are unstable for TT-GDA when $\Lambda_\tau H(z^\ast)$ has eigenvalues on the imaginary axis, and they can be unstable for JM even for strongly convex-strongly concave objectives. Moreover, degenerate saddle points and compact sets of critical points are asymptotically stable for TT-GDA, GEG and TT-PPM under the conditions of Theorem~\ref{thm:degenerate}, with sublinear convergence rates under an additional growth condition. When these conditions fail, Theorem~\ref{thm:stab1_A} and Proposition~\ref{prop:implicit} can be applied directly, as illustrated in Proposition~\ref{prop:bump} for a compact set of non-isolated critical points.
Numerical simulations confirming the theoretical results are available online in a \href{https://mybinder.org/v2/gh/amirali78frz/MCSS-minmax-stability-numerics.git/HEAD?urlpath=%2Fdoc%2Ftree%2FMCSS_binder.ipynb}{\texttt{Binder notebook}}.
For future work, an interesting direction is to use the $O(s)$ resolution ODE to go back from an ODE to its corresponding DTA. In this 
way, one 
can design a flow with intended 
properties and obtain its corresponding DTA.

\backmatter

\begin{appendices}
\section{Additional Theorems and Lemmas}\label{app:extra}
In this section, results from the literature that are used in the proofs of the main results are summarised.
 \begin{theorem}[\mbox{\cite[Thm. 3.3.52]{hinrichsen2005mathematical}}] \label{thm:stability_of_equilibria}
Consider the discrete-time system \eqref{eq:discrete_time_sys} and the continuous-time system \eqref{eq:continuous_sys}, respectively. Assume that $w_s(\cdot)$ and $W_s(\cdot)$ are continuously differentiable in $z^e\in \R^d$, where $z^e$ denotes an equilibrium.
\begin{itemize}
    \item[(i)] If $|\lambda| < 1$ for all $\lambda \in \sigma(\frac{\partial w_s}{\partial z}(z^e))$, ($\Re(\lambda) < 0$ for all $\lambda \in \sigma(\frac{\partial W_s}{\partial z}(z^e))$) then the equilibrium $z^e$ is exponentially stable.      
    \item[(ii)] If there exists $\lambda \in \sigma(\frac{\partial w_s}{\partial z}(z^e))$ such that $|\lambda| > 1$, (if there exists $\lambda \! \in \! \sigma(\frac{\partial W_s}{\partial z}(z^e))$ such that $\Re(\lambda) > 0$) then the equilibrium $z^e$ is unstable.
\end{itemize}
\end{theorem}


 \begin{proposition}[\mbox{\cite[Prop. 2]{lu2022sr}}]
Consider a DTA with iterate update 
$z^+=w_s(z)$ and its $O(s^r)$-resolution ODE $\dot{Z} = f^{(r)}(Z,s)$ in \eqref{eq:ode}. Let $z^\ast$ be a fixed point of the DTA, then $z^\ast$ is 
a fixed point of the $O(s^r)$-resolution ODE for any degree $r.$
\end{proposition}
 The above proposition indicates that any equilibrium of the DTA is also an equilibrium of the $O(s^r)$-resolution ODE, but it does not mean they are the same. 
The resolution ODE may introduce new equilibria 
in the construction process. This can be seen in \eqref{os:TTGDA}, where the $O(s)$-resolution ODE of the TT-GDA may admit equilibria other than the fixed points of \eqref{dyn:TTGDA}, for instance.
 \section{Proof of Main Results}\label{app:main}
In this section, we give proofs of the main results presented in Sections~\ref{sec:stab} and~\ref{sec:main}.

\begin{proof}[Proof of Lemma~\ref{lem:comparison}]
	The identity $V(w_s(z))-V(z)=\big[V(w_s(z))-V(Z_s(s;z))\big]-\big[V(z)-V(Z_s(s;z))\big]$ together with \eqref{eq:comparison_hyp} yields \eqref{eq:comparison_concl}.
\end{proof}

\begin{proof}[Proof of Theorem~\ref{thm:stab}]
	Without loss of generality, let $z^e=0$. Fix $s\in(0,\bar s]$ satisfying \eqref{eq:step_condition}, and write $V=V_s$, $\eta=\eta(s)$, $\tilde\rho=\tilde\rho(s)$ and $\delta=\delta(s)$. Since $\tilde\rho\leq\rho(s)$, inequality \eqref{eq:lyap_decay} also holds with $\tilde\rho$ in place of $\rho(s)$, and $s\tilde\rho\leq1$ by the definition of $\tilde\rho(s)$ in \eqref{eq:thm_constants}.

	\emph{Step 0 (bound on $\eta$).} For $z\in\mathcal B_\delta(0)$, \eqref{eq:lyap_bounds} gives $V(0)=0$, and \eqref{eq:lyap_grad} gives $V(z)=\int_0^1\langle\nabla V(tz),z\rangle\,dt\leq\int_0^1c_3|z|^2tdt\leq\tfrac{c_3}{2}|z|^2$. Together with \eqref{eq:lyap_bounds} this implies that $c_1\leq c_3/2$, and hence $\kappa\leq\frac{1}{8(M+1)}<1$ and $\eta\leq\kappa s\tilde\rho\leq 1$.


    \emph{Step 1 (flow).} Let $|z|<\delta/M$. By Assumption~\ref{as:error}, $Z_s(t;z)$ exists for $t\in[0,s]$. Let $T^\ast$ be the supremum of all $T\in[0,s]$ such that $Z_s(t;z)\in\mathcal B_\delta(0)$ for all $t\in[0,T]$. For $t\in[0,T^\ast)$, \eqref{eq:lyap_decay} yields $\frac{d}{dt}V(Z_s(t;z))\leq-\tilde\rho V(Z_s(t;z))$, and therefore $\frac{d}{dt}\big(e^{\tilde\rho t}V(Z_s(t;z))\big)\leq0$. Hence $V(Z_s(t;z))\leq e^{-\tilde\rho t}V(z)$ and, by \eqref{eq:lyap_bounds}, $|Z_s(t;z)|\leq Me^{-\tilde\rho t/2}|z|\leq M|z|<\delta$. By continuity, this bound also holds at $t=T^\ast$. If $T^\ast<s$, it would then hold on $[0,T^\ast+\varepsilon]$ for some $\varepsilon>0$, which contradicts the definition of $T^\ast$. Hence $T^\ast=s$, and $V(Z_s(t;z))\leq e^{-\tilde\rho t}V(z)$ and $|Z_s(t;z)|\leq Me^{-\tilde\rho t/2}|z|$ hold for all $t\in[0,s]$.

	\emph{Step 2 (decrease along the flow).} Let $|z|<\delta/(M+1)$ and $Z:=Z_s(s;z)$. By Step 1, $V(Z)\leq e^{-s\tilde\rho}V(z)$. Since $1-e^{-x}\geq x/2$ for $x\in[0,1]$ and $s\tilde\rho\leq1$,
	\begin{align}\label{eq:pf_flow_decrease}
		V(z)-V(Z)\geq\big(1-e^{-s\tilde\rho}\big)V(z)\geq\tfrac{s\tilde\rho}{2}V(z).
	\end{align}

	\emph{Step 3 (effect of the error).} Let $w:=w_s(z)=Z+e_s(z)$. For every $\vartheta\in[0,1]$, Step 0, Step 1 and \eqref{eq:defect_bound} give $|Z+\vartheta e_s(z)|\leq M|z|+\eta|z|\leq(M+1)|z|<\delta$. By the mean value theorem applied to $\vartheta\mapsto V(Z+\vartheta e_s(z))$, there exists $\vartheta^\ast\in(0,1)$ such that
	\begin{align*}
		V(w)-V(Z)\!=\!\langle\nabla V(Z+\vartheta^\ast e_s(z)),e_s(z)\rangle\!\leq\! c_3(M+1)|z|\,\eta|z|\leq\frac{c_3(M+1)\eta}{c_1}V(z)\!\leq\!\frac{s\tilde\rho}{4}V(z),
	\end{align*}
	where we used \eqref{eq:lyap_bounds}, \eqref{eq:lyap_grad}, and \eqref{eq:step_condition}. Combined with \eqref{eq:pf_flow_decrease}, this yields $V(w)-V(Z)\leq\tfrac12\big[V(z)-V(Z)\big]$.

	\emph{Step 4 (one-step decrease).} Lemma~\ref{lem:comparison} with $\theta=\tfrac12$ and \eqref{eq:pf_flow_decrease} give
	\begin{align}\label{eq:pf_one_step}
		V(w_s(z))\leq V(z)-\tfrac12\big[V(z)-V(Z)\big]\leq q\,V(z),\qquad q:=1-\tfrac{s\tilde\rho}{4}\in\big[\tfrac34,1\big),
	\end{align}
	for all $z$ with $|z|<\delta/(M+1)$. Moreover, Step 3 with $\vartheta=1$ gives $|w_s(z)|\leq(M+1)|z|<\delta$, that is, $w_s(z)\in\mathcal B_\delta(0)$ for all such $z$.

	\emph{Step 5 (iteration).} Let $|z_0|<\delta/(M(M+1))$. We show by induction that $z_k\in\mathcal B_\delta(0)$ and $V(z_k)\leq q^kV(z_0)$ for all $k\in\N$. These properties hold for $k=0$, since $|z_0|<\delta$. If it holds for some $k$, then \eqref{eq:lyap_bounds} applies at $z_k$ and gives $|z_k|\leq\sqrt{V(z_k)/c_1}\leq\sqrt{q^kV(z_0)/c_1}\leq\sqrt{V(z_0)/c_1}\leq\sqrt{c_2/c_1}\,|z_0|=M|z_0|<\delta/(M+1)$. Hence Step 4 applies to $z_k$ and gives $z_{k+1}=w_s(z_k)\in\mathcal B_\delta(0)$ and $V(z_{k+1})\leq qV(z_k)\leq q^{k+1}V(z_0)$. Consequently,
	\begin{align*}
		|z_k|\leq\sqrt{q^kV(z_0)/c_1}\leq Mq^{k/2}|z_0| \qquad\forall\, k\in\N,
	\end{align*}
	which is \eqref{eq:exp_estimate}. Local stability in the sense of Definition~\ref{def:stab_discr} follows by choosing, for given $\varepsilon>0$, the radius $\min\{\delta/(M(M+1)),\varepsilon/M\}$, and the exponential estimate holds with constants $M$ and $\gamma=q^{1/2}\in(0,1)$.
\end{proof}

\begin{proof}[Proof of Corollary~\ref{cor:fixed_margin}]
	Without loss of generality, let $z^e=0$, and fix $s\in(0,\bar s]$. We construct $V_s$ following \cite[Thm.~4.14]{khalil2002nonlinear}. Let $\delta_1=\delta_0/K$, and define
	\begin{align*}
		T=\frac{\ln(2K^2)}{2\lambda}>0,\qquad V_s(z)=\int_0^T {|Z_s(t;z)|^2}\,dt,\qquad |z|<\delta_1.
	\end{align*}
	Let $\mathcal{O}=\{z \mid |z|<\delta_1\}$.  {By \eqref{eq:unif_exp}, for every $z\in\mathcal{O}$, the solution $Z_s(t;z)$ exists for all $t\geq0$ and satisfies $|Z_s(t;z)|\leq K|z|<\delta_0$.} Hence these solutions remain in the interior of $\mathcal B_{\delta_0}(0)$, where the bounds on $W_s$ apply. Since $W_s$ is $C^1$ on $\mathcal{O}$,  {$Z_s(t;\cdot)$} is $C^1$ on $\mathcal{O}$ \cite[Ch.~V]{hartman2002ordinary}, and hence $V_s$ is $C^1$ on $\mathcal{O}$.

	\emph{Upper bound  {on $V_s(\cdot)$}.} By \eqref{eq:unif_exp}, $V_s(z)\leq K^2|z|^2\int_0^Te^{-2\lambda t}dt\leq c_2|z|^2$ with $c_2=K^2/(2\lambda)$.

	\emph{Lower bound  {on $V_s(\cdot)$}.} Since $W_s(0)=0$ and $\|\frac{\partial W_s}{\partial z}\|\leq L$ on the convex set $\mathcal B_{\delta_0}(0)$, we have $|W_s(\zeta)|\leq L|\zeta|$ on $\mathcal B_{\delta_0}(0)$. Hence  {$\frac{d}{dt}|Z_s(t;z)|^2=2\langle Z_s(t;z),W_s(Z_s(t;z))\rangle\geq-2L|Z_s(t;z)|^2.$ This implies that $|Z_s(t;z)|^2\geq e^{-2Lt}|z|^2$} and $V_s(z)\geq c_1|z|^2$ with $c_1=(1-e^{-2LT})/(2L)$.

    \emph{Gradient bound.}  {Since $W_s$ is $C^1$, the sensitivity matrix $\Phi(t)=\frac{\partial Z_s}{\partial z}(t;z)$ exists, is continuous in $(t,z)$, and satisfies the variational equation $\dot\Phi=\frac{\partial W_s}{\partial z}(Z_s(t;z))\Phi$ with $\Phi(0)=I$, obtained by differentiating $Z_s(t;z)=z+\int_0^tW_s(Z_s(\tau;z))\,d\tau$ with respect to $z$ \cite[Ch.~V, Thm.~3.1]{hartman2002ordinary}. Since $Z_s(t;z)\in\mathcal B_{\delta_0}(0)$, where $\|\frac{\partial W_s}{\partial z}\|\leq L$, we have $\|\Phi(t)\|\leq1+\int_0^tL\|\Phi(\tau)\|\,d\tau$, and the Gr\"onwall--Bellman inequality \cite[Lem.~A.1]{khalil2002nonlinear} yields $\|\Phi(t)\|\leq e^{Lt}$. Differentiating under the integral sign, which is justified since the integrand is $C^1$ in $z$ with derivative continuous in $(t,z)$, and using $\nabla_z|Z_s(t;z)|^2=2\Phi(t)^\top Z_s(t;z)$ and \eqref{eq:unif_exp}, we obtain}

\begin{align*}
	|\nabla V_s(z)|=\Big|2\int_0^T\Phi(t)^\top Z_s(t;z)\,dt\Big|\leq2K|z|\int_0^Te^{(L-\lambda)t}dt\leq c_3|z|,\qquad c_3=2KTe^{LT}.
\end{align*}

    \emph{Decay.} Let $z\in\mathcal{O}$ and let $\tau\geq0$ be sufficiently small such that $Z_s(\tau;z)\in\mathcal{O}$.  {Since $W_s$ does not depend on $t$ and solutions of \eqref{eq:continuous_sys} are unique, the flow satisfies $Z_s(t;Z_s(\tau;z))=Z_s(t+\tau;z)$ for all $t\geq0$. Substituting $u=t+\tau$ gives}
\begin{align*}
	 {V_s(Z_s(\tau;z))=\int_0^T|Z_s(t+\tau;z)|^2\,dt=\int_\tau^{T+\tau}|Z_s(u;z)|^2\,du .}
\end{align*}
 {We differentiate both sides with respect to $\tau$ at $\tau=0$. By the chain rule, the left-hand side gives $\langle\nabla V_s(z),W_s(z)\rangle$. By the fundamental theorem of calculus, the right-hand side gives $|Z_s(T;z)|^2-|z|^2$.} Using \eqref{eq:unif_exp},  {$K^2e^{-2\lambda T}=\frac12$ by the definition of $T$, and $V_s(z)\leq c_2|z|^2$,} we obtain
\begin{align*}
	\langle\nabla V_s(z),W_s(z)\rangle&=|Z_s(T;z)|^2-|z|^2\leq\big(K^2e^{-2\lambda T}-1\big)|z|^2\\&=-\tfrac12|z|^2\leq-\frac{1}{2c_2}V_s(z)=-\frac{\lambda}{K^2}V_s(z).
\end{align*}

	Hence Assumption~\ref{as:lyap} holds with $\delta_V(s)\equiv\delta_1/2$, since $\mathcal B_{\delta_1/2}(0)\subset\mathcal{O}$, the constants $c_1,c_2,c_3$ above, which depend only on $K$, $\lambda$ and $L$, and the constant rate $\rho=\lambda/K^2$. For $s\leq K^2/\lambda$ we have $\tilde\rho(s)=\rho$. Since $\eta(s)=o(s)$, there exists $s^\ast\in(0,\min\{\bar s,K^2/\lambda\}]$ such that $\eta(s)\leq\kappa\rho\, s$ for all $s\in(0,s^\ast]$, and Theorem~\ref{thm:stab} yields the first assertion.

	If $W_s=W$ is independent of $s$ and $0$ is locally exponentially stable for $\dot Z=W(Z)$, then \eqref{eq:unif_exp} holds for some $K\geq1$, $\lambda>0$ and $\delta_0>0$ {, including the existence of solutions for all $t\geq0$ (Definition~\ref{def:stab_cont})}, and $\|\frac{\partial W}{\partial z}\|$ is bounded on the compact set $\mathcal B_{\delta_0}(0)$ by continuity. This proves the second assertion.
\end{proof}

 {Before we proceed with a proof of Theorem~\ref{thm:stab_set_general}, we state the following lemma, which will simplify the following derivations.}

\begin{lemma}\label{lem:discrete_lyap_set}
	Let $\mathcal{A}\subset\R^d$ be nonempty and compact, $0<\delta\leq\delta_2$, and let $V:\mathcal{B}_{\delta_2}(\mathcal{A})\to\R_{\geq0}$ satisfy $\alpha_1(|z|_\mathcal{A})\leq V(z)\leq\alpha_2(|z|_\mathcal{A})$ for all $z\in\mathcal{B}_{\delta_2}(\mathcal{A})$ with $\alpha_1,\alpha_2\in\mathcal{K}_\infty$. Let $g:\mathcal{B}_\delta(\mathcal{A})\to\mathcal{B}_{\delta_2}(\mathcal{A})$, and let $\psi:\mathcal{B}_\delta(\mathcal{A})\to\R_{\geq0}$ be continuous with $\psi(z)>0$ for all $z\in\mathcal{B}_\delta(\mathcal{A})\setminus\mathcal{A}$ and
	\begin{align}\label{eq:lem_set_decrease}
		V(g(z))-V(z)\leq-\psi(z)\qquad\forall\, z\in\mathcal{B}_\delta(\mathcal{A}) .
	\end{align}
	Then, for every $\ell\in(0,\alpha_1(\delta))$, the set $\Omega_\ell=\{z\in\mathcal{B}_\delta(\mathcal{A}) \mid V(z)\leq\ell\}$ is forward invariant for $z^+=g(z)$, and $\mathcal{A}$ is locally asymptotically stable for $z^+=g(z)$.
\end{lemma}

\begin{proof}
	Let $\ell\in(0,\alpha_1(\delta))$ and $z\in\Omega_\ell$. Then $g(z)\in\mathcal{B}_{\delta_2}(\mathcal{A})$ and $V(g(z))\leq V(z)\leq\ell$ by \eqref{eq:lem_set_decrease}. Hence $\alpha_1(|g(z)|_\mathcal{A})\leq\ell<\alpha_1(\delta)$, that is, $|g(z)|_\mathcal{A}<\delta$ and $g(z)\in\Omega_\ell$. This shows forward invariance.

	\emph{Stability.} Let $\varepsilon>0$, $\varepsilon'=\min\{\varepsilon,\delta\}$ and $\ell=\alpha_1(\varepsilon')/2$. Let $\delta'=\alpha_2^{-1}(\ell)$, which satisfies $\delta'\leq\alpha_1^{-1}(\ell)<\varepsilon'$ since $\alpha_1(t)\leq\alpha_2(t)$ for all $t\in[0,\delta_2]$. If $|z_0|_\mathcal{A}<\delta'$, then $z_0\in\mathcal{B}_\delta(\mathcal{A})$ and $V(z_0)\leq\alpha_2(\delta')=\ell$, so $z_k\in\Omega_\ell$ for all $k$, and $\alpha_1(|z_k|_\mathcal{A})\leq\ell<\alpha_1(\varepsilon')$ gives $|z_k|_\mathcal{A}<\varepsilon$.

	\emph{Attractivity.} Let $\ell\in(0,\alpha_1(\delta))$ and $|z_0|_\mathcal{A}<\alpha_2^{-1}(\ell)$, so that $z_k\in\Omega_\ell$ for all $k\in\mathbb N$. Summing \eqref{eq:lem_set_decrease} gives $\sum_{i=0}^{k}\psi(z_i)\leq V(z_0)-V(z_{k+1})\leq V(z_0)$ for all $k\in\mathbb N$, hence $\psi(z_k)\to0$  {for $k\to\infty$}. Suppose that $|z_k|_\mathcal{A}\not\to0$  {for $k\to\infty$}. Then there exist $\varepsilon''\in(0,\delta)$ and a subsequence  {$\{k_j\}_{j\in\mathbb N}$} with $|z_{k_j}|_\mathcal{A}\geq\varepsilon''$  {for all $j\in\mathbb N$}. The set $\{z \mid \varepsilon''\leq|z|_\mathcal{A}\leq\delta\}$ is compact and disjoint from $\mathcal{A}$, so the continuous function $\psi$ attains a positive minimum $m$ on it, and $\psi(z_{k_j})\geq m$ for all $j\in\mathbb N$, a contradiction. Hence we can conclude that $|z_k|_\mathcal{A}\to0$  {for $k\to\infty$}.
\end{proof}

\begin{proof}[Proof of Theorem~\ref{thm:stab_set_general}]
    Let $D_s(z)=V(z)-V(Z_s(s;z))$ for $z\in\mathcal{B}_\delta(\mathcal{A})$. Since $W_s$ is locally Lipschitz continuous on  {an open set containing $\mathcal{B}_{\delta_2}(\mathcal{A})$}, $z\mapsto Z_s(s;z)$ is continuous on $\mathcal{B}_\delta(\mathcal{A})$, and so is $D_s: {\mathcal{B}_\delta(\mathcal{A})\to\R}$. Every $a\in\mathcal{A}$ is an interior point of $U$ at which $V$ attains its minimum $0$, so $\nabla V(a)=0$. Together with (b), $\frac{d}{dt}V(Z_s(t;z))=\langle\nabla V(Z_s(t;z)),W_s(Z_s(t;z))\rangle\leq0$ for $t\in[0,s]$ by (c), hence $D_s(z)\geq0$  {for all $z\in \mathcal{B}_\delta(\mathcal{A})$}. If $z\in\mathcal{B}_\delta(\mathcal{A})\setminus\mathcal{A}$, then the derivative is negative at $t=0$ by (b), hence $V(Z_s(t_1;z))<V(z)$ for some $t_1\in(0,s]$.  {Since $t\mapsto V(Z_s(t;z))$ is nonincreasing on $[0,s]$, it follows that $V(Z_s(s;z))\leq V(Z_s(t_1;z))<V(z)$, that is,} $D_s(z)>0$. By (c) and Lemma~\ref{lem:comparison},  {we can conclude that}
	\begin{align*}
		V(w_s(z))-V(z)\leq-(1-\theta)D_s(z)\qquad\forall\, z\in\mathcal{B}_\delta(\mathcal{A}) .
	\end{align*}
	The assertion follows from Lemma~\ref{lem:discrete_lyap_set} with $g=w_s$ and $\psi=(1-\theta)D_s$.
\end{proof}

\begin{proof}[Proof of Theorem~\ref{thm:stab1_A}]
	Let $s\in(0,s^\ast]$ and $z\in\mathcal{B}_\delta(\mathcal{A})$. By (i) and \eqref{eq:euler_decomp}, $|w_s(z)-z|=s|W(z)+sR_s(z)|\leq sM_w\leq\delta$, so every point $z+t(w_s(z)-z)$ with $t\in[0,1]$ lies in $\mathcal{B}_{2\delta}(\mathcal{A})$. Since $\nabla V$ is Lipschitz continuous on $\mathcal{B}_{2\delta}(\mathcal{A})$,
	\begin{align*}
		V(w_s(z))-V(z)-\langle\nabla V(z),w_s(z)-z\rangle&=\int_0^1\big\langle\nabla V(z+t(w_s(z)-z))-\nabla V(z),w_s(z)-z\big\rangle dt\\&\leq\tfrac{L_V}{2}|w_s(z)-z|^2 .
	\end{align*}
	Inserting $w_s(z)-z=sW(z)+s^2R_s(z)$ (which is a reformulation of \eqref{eq:euler_decomp}) and using (iii) and \eqref{eq:compatibility}, it further holds that
		\begin{align*}
		V(w_s(z))-V(z)&\leq s\langle\nabla V(z),W(z)\rangle+s^2\Big[\langle\nabla V(z),R_s(z)\rangle+\tfrac{L_V}{2}|W(z)+sR_s(z)|^2\Big]\\
		&\leq-s\sigma(z)+s^2C\sigma(z)=-s\sigma(z)\big(1-sC\big)\leq-\tfrac{s}{2}\sigma(z),
	\end{align*}
	where the last inequality uses $\sigma(z)\geq0$ and $sC\leq\tfrac12$, which holds since $s\leq s^\ast\leq1/(2C)$.
    Accordingly, we have proven
    \eqref{eq:set_decrease}, and thus local
    asymptotic stability follows from Lemma~\ref{lem:discrete_lyap_set} with $\delta_2=2\delta$, $g=w_s$ and $\psi=\tfrac{s}{2}\sigma$.

    Let $a\in\mathcal{A}$. Then $\nabla V(a)=0$, since $V(\cdot)$ attains its minimum $0$ at the interior point $a$ of $U$, and $\sigma(a)=0$ by (iii). Moreover, $L_V>0$, since otherwise $\nabla V\equiv0$ on $\mathcal{B}_{2\delta}(\mathcal{A})$, so that $V(\cdot)$ vanishes on the ball $\mathcal{B}_{2\delta}(a')\subset\mathcal{B}_{2\delta}(\mathcal{A})$ around a boundary point $a'$ of $\mathcal{A}$, which contradicts (ii). Hence \eqref{eq:compatibility} at $z=a$ gives $|W(a)+sR_s(a)|=0$, that is, $w_s(a)=a$.
\end{proof}

\begin{proof}[Proof of Corollary~\ref{cor:rates}]
	By Lemma~\ref{lem:discrete_lyap_set} with $g=w_s$ and $\psi=\tfrac{s}{2}\sigma$, the solution remains in $\Omega_\ell {=\{z\in\mathcal{B}_\delta(\mathcal{A}) \mid V(z)\leq\ell\}}\subset\mathcal{B}_\delta(\mathcal{A})$  {and hence} 
    \eqref{eq:set_decrease} is satisfied for solutions starting in $\mathcal{B}_\delta(\mathcal{A})$.
    If $z_0\in\mathcal{A}$, then $z_k=z_0$ for all $k$, since every point of $\mathcal{A}$ is a fixed point of $w_s$ by Theorem~\ref{thm:stab1_A}. 
    Let now $z_0\notin\mathcal{A}$. With $v_k=V(z_k)$ and $a=\tfrac{cs}{2}$, we obtain $0\leq v_{k+1}\leq v_k-a\,v_k^p$ for all $k\in \N$ from  \eqref{eq:set_decrease}.

    For $p=1$, this gives $v_{k+1}\leq(1-a)v_k$ for all $k\in\N$. Since $v_0>0$, the case $k=0$ yields $0\leq(1-a)v_0$, that is, $a=\tfrac{cs}{2}\leq1$. Hence $1-a\in[0,1)$, and multiplying $v_k\leq(1-a)^kv_0$ by $1-a$ gives the claim by induction.

	For $p>1$, note that $v_0>0$ since $z_0\notin\mathcal{A}$. If $v_j=0$ for some $j$, then $v_i=0$ for all $i\geq j$, and the bound holds trivially for these indices. It therefore suffices to consider indices $k$ with $v_{k+1}>0$, so that $v_0,\ldots,v_{k+1}>0$. The function $\varphi(v)=v^{1-p}$ is convex and decreasing on $(0,\infty)$, so
	\begin{align*}
		\varphi(v_{k+1})-\varphi(v_k)\geq  {\tfrac{d}{dv}}\varphi
        (v_k)(v_{k+1}-v_k)=(p-1)v_k^{-p}(v_k-v_{k+1})\geq(p-1)a .
	\end{align*}
	Summing over $j=0,\ldots,k-1$ gives $v_k^{1-p}\geq v_0^{1-p}+(p-1)ak$. Since $1-p<0$, this yields $v_k\leq\big(v_0^{1-p}+(p-1)ak\big)^{-1/(p-1)}$, which is the claim with $a=\tfrac{cs}{2}$.
\end{proof}

\begin{proof}[Proof of Proposition~\ref{prop:implicit}]
	By convexity of $V$, 
    it holds that
    $$V(z)\!=\!V\big(w_s(z)-sW(w_s(z))\big)\!\geq \!V(w_s(z))-s\langle\nabla V(w_s(z)),W(w_s(z))\rangle\!\geq\! V(w_s(z))+s\sigma(w_s(z)),$$ i.e.,
    \eqref{eq:implicit_decrease} follows. For the second assertion, \eqref{eq:implicit_decrease} shows that $V(z_k)$ is nonincreasing, which implies stability, and that $\sum_{k\geq1}s\sigma(z_k)\leq V(z_0)$. Since the iterates remain in the compact sublevel set $\{z \mid V(z)\leq V(z_0)\}$, the attractivity argument in the proof of Lemma~\ref{lem:discrete_lyap_set} yields $|z_k|_\mathcal{A}\to0$ for every $z_0\in\R^d$.
\end{proof}

\begin{proof}[Proof of Proposition~\ref{prop:defect}]
		The function $f^{(r)}$ in \eqref{eq:ode} is of class $C^k$ on $U\times\R$. Let $\Phi(t,z,\tilde s)$ denote the solution of $\dot Z=f^{(r)}(Z,\tilde s)$ with $\Phi(0,z,\tilde s)=z$. By differentiable dependence on initial conditions and parameters \cite[Ch.~V]{hartman2002ordinary}, $\Phi$ is of class $C^k$ on its open domain of definition  $\mathcal D$ . Since $w_s(z^e)=z^e$ for $s\in[0,\bar s)$, \cite[Prop.~2]{lu2022sr} gives $f_i(z^e)=0$ for $i=0,\ldots,r$. Hence, it holds that $\Phi(t,z^e,\tilde{s})=z^e$ for all $t$ and {for all} $\tilde{s}$. 
   Fix
$s_1\in(0,\bar s)$. Since $\Phi(t,z^e,\tilde{s})=z^e$ for all $t,\tilde{s}$, the
compact segment $[0,s_1]\times\{z^e\}\times[0,s_1]$ lies in $\mathcal D$, so
there is a $\delta>0$ with $\mathcal B_\delta(z^e)\subset U$ and
$[0,s_1]\times\mathcal B_\delta(z^e)\times[0,s_1]\subset\mathcal D$.
    Define
	\begin{align*}
		E(z,s)=w_s(z)-\Phi(s,z,s),\qquad (z,s)\in {\mathcal B_\delta(z^e)}\times[0,s_1],
	\end{align*}
	which is of class $C^k$ and satisfies $E(z,s)=e_s(z)$ for $s\in(0,s_1]$  {(according to the error definition in \eqref{eq:defect})}, since $W_s=f^{(r)}(\cdot\,,s)$ and $Z_s(s;z)=\Phi(s,z,s)$.

	By the construction of the coefficient functions in Theorem~\ref{th:ODE} (see the proof of \cite[Thm.~1]{lu2022sr}), the Taylor coefficients of $s\mapsto w_s(z)$ and $s\mapsto\Phi(s,z,s)$ at $s=0$ coincide up to order $r+1$ for every $z$, that is,
	\begin{align}\label{eq:pf_taylor_zero}
         {\tfrac{\partial^j}{\partial s^j}}E(z,0)=0\qquad\forall\, j\in\{0,\ldots,r+1\},\ \forall\, z\in\mathcal B_\delta(z^e).
	\end{align}
	Derivatives with respect to $s$ at $s=0$ are one-sided. Let $G(\cdot,\cdot)= {\frac{\partial}{\partial z}}E(\cdot,\cdot)$, which is of class $C^{k-1}$. Since $j+1\leq r+2\leq k$ for $j\leq r+1$, mixed partial derivatives of $E$ of these orders commute, and differentiating \eqref{eq:pf_taylor_zero} with respect to $z$ gives $\frac{\partial^j}{\partial s^j}G(z,0)=0$ for $j\in\{0,\ldots,r+1\}$.

	(i) Let $k=r+2$, so that $G$ is of class $C^{r+1}$. Taylor's theorem with integral remainder in $s$ and $\frac{\partial^j}{\partial s^j}G(z,0)=0$ for $j\leq r+1$ give, for $s\in[0,s_1]$,
	\begin{align*}
		G(z,s)=\frac{s^{r+1}}{r!}\int_0^1(1-\tau)^r\left[\frac{\partial^{r+1}}{\partial s^{r+1}}G(z,\tau s)-\frac{\partial^{r+1}}{\partial s^{r+1}}G(z,0)\right]\,d\tau .
	\end{align*}
	Hence $\|G(z,s)\|\leq s^{r+1}\omega(s)$ with
	\begin{align*}
		\omega(s)=\frac{1}{(r+1)!}\sup\left\{\left\|\frac{\partial^{r+1}}{\partial s^{r+1}}G(z,\tilde{s})-\frac{\partial^{r+1}}{\partial s^{r+1}}G(z,0)\right\| \;\Big|\; z\in {\mathcal B_\delta(z^e)},\ \tilde{s}\in[0,s]\right\},
	\end{align*}
	and $\omega(s)\to0$ as $s\to0^+$ by uniform continuity of $\frac{\partial^{r+1}}{\partial s^{r+1}}G(\cdot,\cdot)$ on the compact set $ {\mathcal B_\delta(z^e)}\times[0,s_1]$. Since $E(z^e,s)=0$, the mean value inequality on the convex set $\mathcal B_\delta(z^e)$ yields $|e_s(z)|=|E(z,s)-E(z^e,s)|\leq s^{r+1}\omega(s)|z-z^e|$.

	(ii) Let $k=r+3$, so that $G(\cdot)$ is of class $C^{r+2}$. We consider again Taylor's theorem 
    and $\frac{\partial^{j}}{\partial s^j}G(z,0)=0$ for $j\leq r+1$ gives $\|G(z,s)\|\leq K_es^{r+2}$ with $K_e=\frac{1}{(r+2)!}\max\{\|\frac{\partial^{r+2}}{\partial s^{r+2}}G(z,\tilde{s})\| \mid z\in {\mathcal B_\delta(z^e)},\ \tilde{s}\in[0,s_1]\}$, and the conclusion follows as in (i).

	In both cases, $Z_s(\cdot\,;z)=\Phi(\cdot\,,z,s)$ exists on $[0,s]$ for $z\in\mathcal B_\delta(z^e)$ and $s\in(0,s_1]$, and $z^e$ is a common equilibrium. Hence Assumption~\ref{as:error} holds with $\delta_e=\delta$.
\end{proof}

\begin{proof}[Proof of Corollary~\ref{cor:res_fixed}]
	By Proposition~\ref{prop:defect}(i), Assumption~\ref{as:error} holds with $\eta(s)=s\,\omega(s)=o(s)$. Since $W_s=f_0$ is independent of $s$ and of class $C^2$, the assertion follows from Corollary~\ref{cor:fixed_margin}.
\end{proof}

\begin{proof}[Proof of Corollary~\ref{cor:res_vanishing}]
	By Proposition~\ref{prop:defect}(i), Assumption~\ref{as:error} holds with $\eta(s)=s^2\omega(s)$ and $\omega(s)\to0$ as $s\to0^+$. For $s\leq\rho_1^{-1/2}$ we have $\tilde\rho(s)=\min\{\rho(s),s^{-1}\}\geq\rho_1 s$, so \eqref{eq:step_condition} holds whenever $\omega(s)\leq\kappa\rho_1$. This is the case for all sufficiently small $s>0$, and Theorem~\ref{thm:stab} yields the assertion.
\end{proof}

\begin{proof}[Proof of Lemma~\ref{negeig}]
	Write $\nabla^2f(z^\ast)=\left[\begin{smallmatrix}A&B\\B^\top&-C\end{smallmatrix}\right]$ with $A=\nabla^2_{xx}f(z^\ast)$, $B=\nabla^2_{xy}f(z^\ast)$ and $C=-\nabla^2_{yy}f(z^\ast)$. By Proposition~\ref{locNash}, $A\succeq0$ and $C\succeq0$. By \eqref{eq:DF_def},
	\begin{align*}
		N=\Lambda_\tau H(z^\ast)=\begin{bmatrix}-\tau^{-1}A&-\tau^{-1}B\\B^\top&-C\end{bmatrix}.
	\end{align*}
	With $D=\diag(\sqrt{\tau}I,I)$, we obtain $DND^{-1}=-S+K$, where $S=\diag(\tau^{-1}A,C)\succeq0$ and $K=\left[\begin{smallmatrix}0&-\tau^{-1/2}B\\\tau^{-1/2}B^\top&0\end{smallmatrix}\right]=-K^\top$. Let $M=-S+K$.

	(i) If $Mv=\kappa v$ with $v\in\mathbb{C}^{n+m}\setminus\{0\}$, then $\kappa|v|^2=v^{\mathsf H}Mv=-v^{\mathsf H}Sv+v^{\mathsf H}Kv$. Since $v^{\mathsf H}Sv\geq0$ is real and $v^{\mathsf H}Kv$ is purely imaginary, $\Re(\kappa)|v|^2=-v^{\mathsf H}Sv\leq0$ (see also \cite[Lem.~2]{farzin2025properties}).

	(ii) Let $\kappa=\jmath b$ with $b\in\R$ be an eigenvalue of $M$ with eigenvector $v$. By (i), $v^{\mathsf H}Sv=0$, hence $Sv=0$ since $S\succeq0$, and therefore $Kv=\jmath bv$. It follows that $(M-\jmath bI)^{\mathsf H}v=(-S-K+\jmath bI)v=-\jmath bv+\jmath bv=0$. If there were a vector $u$ with $(M-\jmath bI)u=v$, then $|v|^2=v^{\mathsf H}(M-\jmath bI)u=\big((M-\jmath bI)^{\mathsf H}v\big)^{\mathsf H}u=0$, a contradiction. Hence $\jmath b$ is a semisimple eigenvalue of $M$, and thus of $N$, since $N$ and $M$ are similar.

	If $\nabla^2f(z^\ast)$ is invertible, then $H(z^\ast)=\diag(-I,I)\nabla^2f(z^\ast)$ and $N=\Lambda_\tau H(z^\ast)$ are invertible, so $0\notin\sigma(N)$.
\end{proof}

\begin{proof}[Proof of Proposition~\ref{prop:saddle_transfer}]
	(i) Since $\Re(\kappa)<0$ for all $\kappa\in\sigma(N)$ and $a_1>0$, the matrix $\frac{\partial f_0}{\partial z}(z^\ast)=a_1N$ is Hurwitz, and $z^\ast$ is a locally exponentially stable equilibrium of $\dot z=f_0(z)$ by Theorem~\ref{thm:stability_of_equilibria}(i). The hypotheses of Proposition~\ref{prop:defect} hold with $r=0$ and $k=2$, and the assertion follows from Corollary~\ref{cor:res_fixed}.

	(ii) By Lemma~\ref{negeig} and Assumption~\ref{as:inv}(i), $0\notin\sigma(N)$, all eigenvalues of $N$ have non-positive real parts, and those on the imaginary axis are semisimple. Hence there exists an invertible real matrix $T$ with $T^{-1}NT=\diag(N_-,R_0)$, where $\sigma(N_-)\subset\{\kappa \mid \Re(\kappa)<0\}$ and $R_0=\diag(b_1E,\ldots,b_qE)$ with $E=\left[\begin{smallmatrix}0&1\\-1&0\end{smallmatrix}\right]$ and $b_1,\ldots,b_q>0$. Either block may be absent, and the modifications below are then obvious. Let $\mathcal{J}_s=a_1N+sa_2N^2$ denote the Jacobian of $f_0+sf_1$ at $z^\ast$ and $p_s(X)=a_1X+sa_2X^2$, so that $T^{-1}\mathcal{J}_sT=\diag(p_s(N_-),p_s(R_0))$.

	Since $a_1N_-$ is Hurwitz, there exists $P_-\succ0$ with $a_1(N_-^\top P_-+P_-N_-)=-I$. Let $s_1=\min\{\bar s,\,1/(4a_2\|P_-\|\|N_-\|^2)\}$. For $s\in(0,s_1]$,
	\begin{align*}
		p_s(N_-)^\top P_-+P_-p_s(N_-)=-I+sa_2\big((N_-^2)^\top P_-+P_-N_-^2\big)\preceq-\tfrac12I\preceq-\tfrac{1}{2\|P_-\|}P_- .
	\end{align*}
	Since $E^\top=-E$ and $E^2=-I$, we have $R_0^\top=-R_0$ and $R_0^2=-\diag(b_1^2I_2,\ldots,b_q^2I_2)$, so that
	\begin{align*}
		p_s(R_0)^\top+p_s(R_0)=2sa_2R_0^2\preceq-2sa_2b_{\min}^2I,\qquad b_{\min}=\min_jb_j .
	\end{align*}
	Let $P=T^{-\top}\diag(P_-,I)T^{-1}\succ0$ and $\hat\rho(s)=\min\{\frac{1}{4\|P_-\|},\,sa_2b_{\min}^2\}$. Then
	\begin{align*}
		&\mathcal{J}_s^\top P+P\mathcal{J}_s\\&=T^{-\top}\Big(\diag\big(p_s(N_-),p_s(R_0)\big)^\top\diag(P_-,I)+\diag(P_-,I)\diag\big(p_s(N_-),p_s(R_0)\big)\Big)T^{-1}\\&\preceq-2\hat\rho(s)P .
	\end{align*}
	Let $\zeta=z-z^\ast$. Since $f_0(z^\ast)=f_1(z^\ast)=0$ by \cite[Prop.~2]{lu2022sr} and $f_0,f_1$ are $C^2$, there exist $\delta_0>0$ and $K_r>0$ such that $f_0(z)+sf_1(z)=\mathcal{J}_s\zeta+r_s(\zeta)$ with $|r_s(\zeta)|\leq K_r|\zeta|^2$ for all $|\zeta|\leq\delta_0$ and $s\in(0,s_1]$. For $V(z)=\zeta^\top P\zeta$, we obtain
	\begin{align*}
		\langle\nabla V(z),f_0(z)\!+\!sf_1(z)\rangle&\!=\!\zeta^\top(\mathcal{J}_s^\top P\!+\!P\mathcal{J}_s)\zeta\!+\!2\zeta^\top Pr_s(\zeta)\!\leq\!-2\hat\rho(s)V(z)\!+\!\frac{2\lambda_{\max}(P)K_r|\zeta|}{\lambda_{\min}(P)}\!V(z)\\&\leq-\hat\rho(s)V(z)
	\end{align*}
	for $|\zeta|\leq\delta_V(s):=\min\{\delta_0,\hat\rho(s)\lambda_{\min}(P)/(2\lambda_{\max}(P)K_r)\}$. Hence Assumption~\ref{as:lyap} holds for $W_s=f_0+sf_1$ on $(0,s_1]$ with $V_s=V$, $c_1=\lambda_{\min}(P)$, $c_2=\lambda_{\max}(P)$, $c_3=2\lambda_{\max}(P)$, radius $\delta_V(s)$, and rate $\hat\rho(s)\geq\rho_1s$, where $\rho_1=\min\{\frac{1}{4\|P_-\|s_1},a_2b_{\min}^2\}$. The hypotheses of Proposition~\ref{prop:defect} hold with $r=1$ and $k=3$, and Corollary~\ref{cor:res_vanishing} yields the assertion.

	(iii) Let $\jmath b\in\sigma(N)$ with $b\neq0$. Then $a_1\jmath b+sa_2(\jmath b)^2=\jmath a_1b-sa_2b^2$ is an eigenvalue of $\mathcal{J}_s$ with real part $-sa_2b^2>0$, and $z^\ast$ is unstable by Theorem~\ref{thm:stability_of_equilibria}(ii).
\end{proof}

\begin{proof}[Proof of Proposition~\ref{prop:TTGDA_uns}]
	The proof follows that of \cite[Thm.~2]{lee2019first}. The map $w_s(z)=z-s\Lambda_\tau F(z)$ is $C^1$ with $\frac{\partial w_s}{\partial z}(z)=I+s\Lambda_\tau H(z)$. Since $\|\Lambda_\tau H(z)\|\leq\ell_\tau\|\nabla^2f(z)\|\leq L_\tau$ and $sL_\tau<1$, this matrix is invertible for all $z$, and $w_s$ is a local diffeomorphism by the inverse function theorem \cite[Thm.~9.24]{rudin1964principles}. Hence \cite[Thm.~2]{lee2019first} applies and shows that the set of initial points for which the iterates converge to a fixed point $\bar z$ with $\rho(\frac{\partial w_s}{\partial z}(\bar z))>1$ has Lebesgue measure zero.
\end{proof}

\begin{proof}[Proof of Theorem~\ref{thm:TTGDA_con}]
	(i) The Jacobian of \eqref{o1:TTGDA} at $z^\ast$ is $\Lambda_\tau H(z^\ast)=N$, which is Hurwitz, so $z^\ast$ is locally exponentially stable by Theorem~\ref{thm:stability_of_equilibria}(i). By Appendix~\ref{app:o}, the Jacobian of \eqref{os:TTGDA} at $z^\ast$ is $N-\frac{s}{2}N^2$, whose eigenvalues are $\kappa-\frac{s}{2}\kappa^2$ with real parts $a-\frac{s}{2}(a^2-b^2)$ for $\kappa=a+\jmath b\in\sigma(N)$. If $a^2\geq b^2$, this is at most $a<0$. If $b^2>a^2$, then $\frac{s}{2}(b^2-a^2)<\frac{s}{2}L_\tau^2<\mu\leq|a|$ for $s\in(0,2\mu/L_\tau^2)$, so the real part is negative. Since $f$ is $C^3$, \eqref{os:TTGDA} is $C^1$, and Theorem~\ref{thm:stability_of_equilibria}(i) yields the assertion.

	(ii) If $N$ has an eigenvalue on the imaginary axis, then the Jacobian $N$ of \eqref{o1:TTGDA} at $z^\ast$ is not Hurwitz. Since $f$ is $C^3$, \eqref{o1:TTGDA} has a locally Lipschitz continuous Jacobian, and $z^\ast$ is not exponentially stable by \cite[Thm.~4.15]{khalil2002nonlinear}. If $\jmath b\in\sigma(N)$ with $b\neq0$, then $\jmath b-\frac{s}{2}(\jmath b)^2=\jmath b+\frac{s}{2}b^2$ is an eigenvalue of the Jacobian of \eqref{os:TTGDA} with positive real part, so $z^\ast$ is unstable by Theorem~\ref{thm:stability_of_equilibria}(ii).
\end{proof}

\begin{proof}[Proof of Corollary~\ref{coro:TTGDA}]
	(i) The map $(s,z)\mapsto z-s\Lambda_\tau F(z)$ is $C^2$, $f_0=-\Lambda_\tau F$ is $C^2$, and \eqref{eq:poly_jacobian} holds with $a_1=1$ and $a_2=-\tfrac12$ (Appendix~\ref{app:o}). The assertion follows from Proposition~\ref{prop:saddle_transfer}(i). 
    
    (ii) We have $\frac{\partial w_s}{\partial z}(z^\ast)=I+sN$. If $\jmath b\in\sigma(N)$ with $b\in\R\setminus\{0\}$, then $1+\jmath sb\in\sigma(I+sN)$ and $|1+\jmath sb|>1$, so $z^\ast$ is unstable by Theorem~\ref{thm:stability_of_equilibria}(ii).
\end{proof}

\begin{proof}[Proof of Theorem~\ref{thm:GEG_con}]
	(i) The Jacobian of \eqref{o1:GEG} at $z^\ast$ is $\gamma N$, and the argument is the same as in the proof of Theorem~\ref{thm:TTGDA_con}.

	(ii) By Appendix~\ref{app:o}, the Jacobian of \eqref{os:GEG} at $z^\ast$ is $\gamma N+s\gamma(1-\frac{\gamma}{2})N^2$, whose eigenvalues have real parts $\gamma\big(a+s(1-\frac{\gamma}{2})(a^2-b^2)\big)$ for $\kappa=a+\jmath b\in\sigma(N)$, where $a\leq0$ by Lemma~\ref{negeig}. If $a=0$, then $b\neq0$ since $0\notin\sigma(N)$, and the real part is $-\gamma s(1-\frac{\gamma}{2})b^2<0$ for $\gamma\in(0,2)$. If $a<0$ and $a^2\leq b^2$, the real part is at most $\gamma a<0$. If $a<0$ and $a^2>b^2$, then $|a|\leq\nu$ and $s(1-\frac{\gamma}{2})(a^2-b^2)\leq s(1-\frac{\gamma}{2})\nu|a|<|a|$ for $s$ in the stated range, so the real part is negative. Since $f$ is $C^3$, \eqref{os:GEG} is $C^1$, and Theorem~\ref{thm:stability_of_equilibria}(i) yields the assertion.
\end{proof}
\begin{proof}[Proof of Corollary~\ref{coro:GEG}]
	If $f$ is $C^{k+1}$, then $F$ is $C^k$ and $(s,z)\mapsto z-\gamma s\Lambda_\tau F(z-s\Lambda_\tau F(z))$ is $C^k$. By Appendix~\ref{app:o}, $f_0=-\gamma\Lambda_\tau F$ and $f_1=\gamma(1-\frac{\gamma}{2})\Lambda_\tau\nabla F\Lambda_\tau F$, so \eqref{eq:poly_jacobian} holds with $a_1=\gamma$ and $a_2=\gamma(1-\frac{\gamma}{2})$. For (i), $f\in C^3$ gives the regularity required in Proposition~\ref{prop:saddle_transfer}(i). For (ii), $f\in C^5$ implies that the update map and $f_0$ are $C^4$ and that $f_1$ is $C^3$, and $a_2>0$, so Proposition~\ref{prop:saddle_transfer}(ii) applies.
\end{proof}

\begin{proof}[Proof of Proposition~\ref{prop:TTPPM_uns}]
	We need to prove that $w_s$ is a local diffeomorphism.
	Considering the inverse function theorem \cite[Thm. 9.24]{rudin1964principles}, it is sufficient to show that the Jacobian of \eqref{dyn:TTPPM} is invertible.
	Since $I+s\Lambda_\tau\nabla F(w)=I-s\Lambda_\tau H(w)$ is invertible for all $w$ when $s\|\Lambda_\tau H(w)\|<1$, the implicit function theorem shows that $w_s$ is $C^1$. Differentiating $w_s(z) + s\Lambda_\tau F(w_s(z)) = z$ w.r.t. $z$ results in
\begin{align}\label{eq:J_PPM}
	J(z) + s\Lambda_\tau\nabla F(w_s(z))J(z) = I,
\end{align}
and consequently
	\(
    J(z) = (I+s\Lambda_\tau\nabla F(w_s(z)))^{-1} = (I-s\Lambda_\tau H(w_s(z)))^{-1},
    \)
	which is invertible. From Assumption~\ref{as:inv}(ii) we know \(\|H(z)\|\leq\|\nabla^2f(z)\|\leq L\), and from \eqref{eq:F_definition} it follows that $\|\Lambda_\tau\|\leq\max\{\tau^{-1},1\}.$ Thus, for $s<1/L_\tau=\frac{\min\{1,\tau\}}{L}$, the right-hand side of \eqref{dyn:TTPPM} is a local diffeomorphism.
	Then, the remainder of the proof follows along the lines of the arguments in \cite[Thm 2]{lee2019first}.
\end{proof}

\begin{proof}[Proof of Theorem~\ref{thm:TTPPM_con}]
	The ODE \eqref{o1:TTPPM} coincides with \eqref{o1:GEG} and \eqref{os:TTPPM} coincides with \eqref{os:GEG} for $\gamma=1$. The assertions therefore follow from Theorem~\ref{thm:GEG_con} with $\gamma=1$, for which $1/((1-\frac{\gamma}{2})\nu)=2/\nu$.
\end{proof}

\begin{proof}[Proof of Corollary~\ref{coro:TTPPM}]
	For $|s|<1/L_\tau$, $w_s(z)$ is the unique solution $w$ of $w+s\Lambda_\tau F(w)=z$. Since $I+s\Lambda_\tau\nabla F(w)$ is invertible for such $s$, the implicit function theorem shows that $(s,z)\mapsto w_s(z)$ is $C^k$ if $F$ is $C^k$. By Appendix~\ref{app:o}, \eqref{eq:poly_jacobian} holds with $a_1=1$ and $a_2=\tfrac12$. The assertions follow from Proposition~\ref{prop:saddle_transfer}(i) and (ii) as in the proof of Corollary~\ref{coro:GEG}.
\end{proof}

\begin{proof}[Proof of Theorem~\ref{thm:DN_con}]
	Since $f$ is $C^4$, the vector fields in \eqref{o1:DN} and \eqref{os:DN} are $C^1$ near $z^\ast$. By Appendix~\ref{app:o}, their Jacobians at $z^\ast$ are $-I$ and $-(1+\frac{s}{2})I$, respectively, and the assertion follows from Theorem~\ref{thm:stability_of_equilibria}(i).
\end{proof}
\begin{proof}[Proof of Corollary~\ref{coro:DN}]
	If $f$ is $C^4$, then $d$ is $C^2$ near $z^\ast$, so $(s,z)\mapsto z-sd(z)$ and $f_0=-d$ are $C^2$, and the first assertion follows from Theorem~\ref{thm:DN_con} and Corollary~\ref{cor:res_fixed}. If $f$ is $C^3$, then $w_s$ is $C^1$ near $z^\ast$ with $\frac{\partial w_s}{\partial z}(z^\ast)=(1-s)I$, and the second assertion follows from Theorem~\ref{thm:stability_of_equilibria}(i).
\end{proof}

\begin{proof}[Proof of Proposition~\ref{prop:RDN_uns}]
		Since $\|\nabla F(z)\|=\|\nabla^2f(z)\|\leq L<\phi$, the smallest singular value of $\nabla F(z)+\phi I$ is at least $\phi-L$, so $\|(\nabla F(z)+\phi I)^{-1}\|\leq(\phi-L)^{-1}$ and, since $|F(z)|=|\nabla f(z)|$, $|d_\phi(z)|\leq|\nabla f(z)|/(\phi-L)$ for all $z$. Moreover, $\|\mathrm{D}\nabla F(z)\|=\|\nabla^3f(z)\|$, so $\|Q_\phi(z)\|\leq\|\nabla^3f(z)\|\,|\nabla f(z)|/(\phi-L)^2\leq G/(\phi-L)^2$. By \eqref{eq:RDN_newton_jac}, the map $w_s(z)=z-sd_\phi(z)$ is $C^1$ with
	\begin{align*}
		\Big\|\tfrac{\partial w_s}{\partial z}(z)-I\Big\|=s\Big\|(\nabla F(z)+\phi I)^{-1}\nabla F(z)-Q_\phi(z)\Big\|\leq s\,q_\phi<1
	\end{align*}
	for $s\in(0,1/q_\phi)$, so $\frac{\partial w_s}{\partial z}(z)$ is invertible for all $z$ and $w_s$ is a local diffeomorphism. The first assertion follows from \cite[Thm.~2]{lee2019first}. At a critical point $\bar z$, $d_\phi(\bar z)=0$ and $Q_\phi(\bar z)=0$, so $\frac{\partial w_s}{\partial z}(\bar z)=I+s(\phi I-H(\bar z))^{-1}H(\bar z)$. If $\kappa\in(0,\phi)$ is an eigenvalue of $H(\bar z)$, then $1+s\kappa/(\phi-\kappa)>1$ is an eigenvalue of $\frac{\partial w_s}{\partial z}(\bar z)$, so $\bar z$ is linearly unstable.
\end{proof}

\begin{proof}[Proof of Theorem~\ref{thm:RDN_con}]
	Let $G_\phi=(\phi I-H(z^\ast))^{-1}H(z^\ast)$. By Appendix~\ref{app:o}, the Jacobians of \eqref{o1:RDN} and \eqref{os:RDN} at $z^\ast$ are $G_\phi$ and $G_\phi-\frac{s}{2}G_\phi^2$, and both vector fields are $C^1$ near $z^\ast$ since $f$ is $C^4$. The eigenvalues of $G_\phi$ are $q=\kappa/(\phi-\kappa)$ for $\kappa\in\sigma(H(z^\ast))$, where $\Re(\kappa)\leq0$ and $\kappa\neq0$ by Lemma~\ref{negeig} with $\tau=1$. Hence
	\begin{align*}
		\Re(q)=\frac{\phi\Re(\kappa)-|\kappa|^2}{|\phi-\kappa|^2}<0,\qquad\frac{|\Re(q)|}{|q|^2}=\frac{|\kappa|^2-\phi\Re(\kappa)}{|\kappa|^2}\geq1 .
	\end{align*}
	The eigenvalues of $G_\phi-\frac{s}{2}G_\phi^2$ are $q-\frac{s}{2}q^2$, and $\Re(q-\frac{s}{2}q^2)\leq\Re(q)+\frac{s}{2}|q|^2=-|\Re(q)|+\frac{s}{2}|q|^2<0$ for $s\in(0,2)$. The assertion follows from Theorem~\ref{thm:stability_of_equilibria}(i).
\end{proof}

\begin{proof}[Proof of Corollary~\ref{coro:RDN}]
	If $f$ is $C^4$, then $d_\phi$ is $C^2$ near $z^\ast$, so $(s,z)\mapsto z-sd_\phi(z)$ and $f_0=-d_\phi$ are $C^2$, and the assertion follows from Theorem~\ref{thm:RDN_con} and Corollary~\ref{cor:res_fixed}.
\end{proof}

\begin{proof}[Proof of Theorem~\ref{thm:degenerate}]
	Let $M_F=\max_{z\in\mathcal{B}_\delta(\mathcal{A})}|F(z)|$, let $L>0$ be an upper bound of $\|\nabla F(z)\|=\|\nabla^2f(z)\|$ on $\mathcal{B}_{2\delta}(\mathcal{A})$, and let $L_\psi$ be a Lipschitz constant of $\nabla\psi$ on $\mathcal{B}_{2\delta}(\mathcal{A})$, which exists since $\psi$ is $C^2$ on a neighbourhood of the compact set $\mathcal{B}_{2\delta}(\mathcal{A})$. Let $\ell_\tau=\max\{1,\tau^{-1}\}$, $W=-\gamma\Lambda_\tau F$ and $\sigma(z)=\gamma\beta|F(z)|^2$. Then $\sigma$ is continuous, $\sigma(z)>0$ for $z\in\mathcal{B}_\delta(\mathcal{A})\setminus\mathcal{A}$, and, by (b), $\langle\nabla\psi(z),W(z)\rangle\leq-\sigma(z)$ on $\mathcal{B}_\delta(\mathcal{A})$. We verify the remaining assumptions of Theorem~\ref{thm:stab1_A} with $V=\psi$ for $s\in(0,\bar s]$, where $\bar s=\min\{\delta/(\ell_\tau(M_F+L\delta)),1/(2\ell_\tau L)\}$. Let $z\in\mathcal{B}_\delta(\mathcal{A})$, so that $|F(z)|\leq M_F$ and $ {\mathcal{B}_\delta(z)}\subset\mathcal{B}_{2\delta}(\mathcal{A})$.

	\emph{TT-GDA.} Here $w_s(z)=z+sW(z)$ with $R_s\equiv0$, and $\frac{L_\psi}{2}|W(z)|^2\leq\frac{L_\psi\ell_\tau^2}{2}|F(z)|^2=\frac{L_\psi\ell_\tau^2}{2\beta}\sigma(z)$, which is \eqref{eq:compatibility}. Assumption (c) is not used.

	\emph{GEG.} Let $u=z-s\Lambda_\tau F(z)$. Then $|u-z|\leq s\ell_\tau M_F\leq\delta$, so the segment between $z$ and $u$ lies in $ {\mathcal{B}_\delta(z)}$. We have $w_s(z)=z+sW(z)+s^2R_s(z)$ with $R_s(z)=\frac{\gamma}{s}\Lambda_\tau\big(F(z)-F(u)\big)$, and $|F(z)-F(u)|\leq Ls\ell_\tau|F(z)|$. Hence $|R_s(z)|\leq\gamma L\ell_\tau^2|F(z)|$ and $|W(z)+sR_s(z)|=\gamma|\Lambda_\tau F(u)|\leq\gamma\ell_\tau(1+sL\ell_\tau)|F(z)|\leq\frac32\gamma\ell_\tau|F(z)|$. Using (c),
	\begin{align*}
		\langle\nabla\psi(z),R_s(z)\rangle+\tfrac{L_\psi}{2}|W(z)+sR_s(z)|^2&\leq\Big(c_\psi\gamma L\ell_\tau^2+\tfrac98L_\psi\gamma^2\ell_\tau^2\Big)|F(z)|^2\\&=\frac{c_\psi L\ell_\tau^2+\frac98L_\psi\gamma\ell_\tau^2}{\beta}\sigma(z).
	\end{align*}

	\emph{TT-PPM.} The map $v\mapsto z-s\Lambda_\tau F(v)$ maps $ {\mathcal{B}_\delta(z)}$ into itself, since $s\ell_\tau|F(v)|\leq s\ell_\tau(M_F+L\delta)\leq\delta$, and it is a contraction with constant $s\ell_\tau L\leq\frac12$. Hence it has a unique fixed point $w_s(z)\in {\mathcal{B}_\delta(z)}$, which depends continuously on $z$. Since $|w_s(z)-z|\leq s\ell_\tau|F(w_s(z))|$, we have $|F(w_s(z))|\leq|F(z)|+sL\ell_\tau|F(w_s(z))|$, hence $|F(w_s(z))|\leq2|F(z)|$. Writing $w_s(z)=z+sW(z)+s^2R_s(z)$ with $W=-\Lambda_\tau F$ and $R_s(z)=\frac{1}{s}\Lambda_\tau\big(F(z)-F(w_s(z))\big)$, we obtain $|R_s(z)|\leq\frac{\ell_\tau L}{s}|w_s(z)-z|\leq2L\ell_\tau^2|F(z)|$ and $|W(z)+sR_s(z)|=|\Lambda_\tau F(w_s(z))|\leq2\ell_\tau|F(z)|$, so that \eqref{eq:compatibility} holds with $C=(2c_\psi L\ell_\tau^2+2L_\psi\ell_\tau^2)/\beta$.

    In each case, the bounds above show that (i) of Theorem~\ref{thm:stab1_A} holds with some $M_w>0$, and $R_s$ is continuous on $\mathcal{B}_\delta(\mathcal{A})$, and assumption (a) provides the bounds in Theorem~\ref{thm:stab1_A}(ii). Theorem~\ref{thm:stab1_A} yields local asymptotic stability of $\mathcal{A}$ and \eqref{eq:set_decrease}, which is \eqref{eq:degenerate_decrease}. If $|F|^2\geq c_F\psi^p$ on $\mathcal{B}_\delta(\mathcal{A})$, then $\sigma\geq\gamma\beta c_F\psi^p$, and Corollary~\ref{cor:rates} applies.
\end{proof}

\begin{proof}[Proof of Corollary~\ref{cor:separable}]
	Let $z=(x,y)$. Since $\nabla_xf(z)=\nabla g(x)$ and $\nabla_yf(z)=-\nabla h(y)$, we have $F(z)=(\nabla g(x),\nabla h(y))$ and $\nabla\psi(z)=F(z)$. Moreover, $|z|_\mathcal{A}^2=|x|_{\mathcal{A}_x}^2+|y|_{\mathcal{A}_y}^2$, so that $|x|_{\mathcal{A}_x}\leq|z|_\mathcal{A}$ and $|y|_{\mathcal{A}_y}\leq|z|_\mathcal{A}$.

	\emph{Saddle points.} Let $(a,b)\in\mathcal{A}$. For $x\in\mathcal{B}_{3\delta}(a)$, we have $|x|_{\mathcal{A}_x}\leq3\delta$ and hence $g(x)\geq g^\ast=g(a)$. Similarly, $h(y)\geq h^\ast=h(b)$ for $y\in\mathcal{B}_{3\delta}(b)$. Thus $a$ is a local minimiser of $g$ and $b$ is a local minimiser of $h$, so that $\nabla g(a)=0$, $\nabla h(b)=0$ and $F(a,b)=0$. Moreover, for all $(x,y)\in\mathcal{B}_{3\delta}(a)\times\mathcal{B}_{3\delta}(b)$,
	\begin{align*}
		f(a,y)=g^\ast-h(y)\leq g^\ast-h^\ast=f(a,b)\leq g(x)-h^\ast=f(x,b),
	\end{align*}
	so $(a,b)$ is a local saddle point of $f$ by \eqref{eq:saddle_local}. It remains to verify the assumptions of Theorem~\ref{thm:degenerate}.

		\emph{Nonvanishing of $F$.} If $z\in\mathcal{B}_\delta(\mathcal{A})\setminus\mathcal{A}$, then $x\notin\mathcal{A}_x$ or $y\notin\mathcal{A}_y$, and $|x|_{\mathcal{A}_x},|y|_{\mathcal{A}_y}\leq|z|_\mathcal{A}\leq\delta$, so $\nabla g(x)\neq0$ or $\nabla h(y)\neq0$, that is, $F(z)\neq0$.

	\emph{Assumptions (b) and (c).} For all $z\in\R^{n+m}$,
	\begin{align*}
		\langle\nabla\psi(z),\Lambda_\tau F(z)\rangle=\tau^{-1}|\nabla g(x)|^2+|\nabla h(y)|^2\geq\beta|F(z)|^2,\qquad|\nabla\psi(z)|=|F(z)|.
	\end{align*}

	\emph{Assumption (a).} The function $\psi$ is $C^2$ on $\R^{n+m}$ and vanishes on $\mathcal{A}$. On the open set $U=\{(x,y) \mid |x|_{\mathcal{A}_x}<3\delta,\ |y|_{\mathcal{A}_y}<3\delta\}$, which contains $\mathcal{B}_{2\delta}(\mathcal{A})$, we have $g(x)\geq g^\ast$ and $h(y)\geq h^\ast$, so $\psi\geq0$ on $U$, and we regard $\psi$ as a function on $U$. Let $z\in\mathcal{B}_{2\delta}(\mathcal{A})\setminus\mathcal{A}$. Then $|x|_{\mathcal{A}_x}\leq2\delta$ and $|y|_{\mathcal{A}_y}\leq2\delta$, so $g(x)\geq g^\ast$ and $h(y)\geq h^\ast$, and at least one of these inequalities is strict. Hence $\psi(z)>0$. For $t\in(0,2\delta]$, let
	\begin{align*}
		\varphi(t)=\min\{\psi(z) \mid t\leq|z|_\mathcal{A}\leq2\delta\}.
	\end{align*}
	The minimum is attained on a compact set disjoint from $\mathcal{A}$, so $\varphi(t)>0$, and $\varphi$ is nondecreasing. Define $\alpha_1(t)=\frac{1}{2\delta}\int_0^t\varphi(r)\,dr$ for $t\in[0,2\delta]$ and $\alpha_1(t)=\alpha_1(2\delta)+t-2\delta$ for $t>2\delta$. Then $\alpha_1\in\mathcal{K}_\infty$, and for $z\in\mathcal{B}_{2\delta}(\mathcal{A})$ with $t=|z|_\mathcal{A}>0$,
	\begin{align*}
		\alpha_1(t)\leq\tfrac{t}{2\delta}\varphi(t)\leq\varphi(t)\leq\psi(z).
	\end{align*}
	For the upper bound, let $\bar L\geq1$ be an upper bound of $\|\nabla^2f\|$ on $\mathcal{B}_{2\delta}(\mathcal{A})$, which exists since $f$ is $C^2$ and $\mathcal{B}_{2\delta}(\mathcal{A})$ is compact. Let $z\in\mathcal{B}_{2\delta}(\mathcal{A})$ and let $a\in\mathcal{A}$ satisfy $|z-a|=|z|_\mathcal{A}$. Every point $w$ on the segment from $a$ to $z$ satisfies $|w|_\mathcal{A}\leq|w-a|\leq|z-a|$, so the segment lies in $\mathcal{B}_{2\delta}(\mathcal{A})$. Since $F(a)=0$ and $\|\nabla F\|=\|\nabla^2f\|$, this gives $|F(w)|\leq\bar L|w-a|$ on the segment, and therefore
	\begin{align*}
		\psi(z)=\psi(z)-\psi(a)=\int_0^1\big\langle F(a+t(z-a)),z-a\big\rangle\,dt\leq\bar L|z-a|^2\int_0^1t\,dt=\tfrac{\bar L}{2}|z|_\mathcal{A}^2 .
	\end{align*}
	Hence (a) holds with $\alpha_2(t)=\frac{\bar L}{2}t^2$. All assumptions of Theorem~\ref{thm:degenerate} are satisfied with $\beta=\min\{1,\tau^{-1}\}$ and $c_\psi=1$, which proves the assertion.
\end{proof}

\begin{proof}[Proof of Proposition~\ref{prop:nonsep}]
	Let $D=\diag(\tau,c_y)$, so that $\nabla\psi(z)=Dv(z)$ and $F(z)=M(y)v(z)$. Then $\langle\nabla\psi(z),\Lambda_\tau F(z)\rangle=v^\top D\Lambda_\tau M(y)v=v^\top S(y)v$, where $S(y)$ is the symmetric part of
	\begin{align*}
		D\Lambda_\tau M(y)=\begin{bmatrix}1&b\\-3bc_yy^2&c_y\end{bmatrix}.
	\end{align*}
	Since $S(0)=\left[\begin{smallmatrix}1&b/2\\b/2&c_y\end{smallmatrix}\right]$ is positive definite for $c_y>b^2/4$, continuity yields $\delta>0$ and $m>0$ with $S(y)\succeq mI$ for all $|y|\leq\delta$. Let $\kappa_M=\max_{|y|\leq\delta}\|M(y)\|$. For $z\in\mathcal{B}_\delta(\mathcal{A})=\mathcal{B}_\delta(0)$, we have $|x|,|y|\leq\delta$, and hence
	\begin{align*}
		\langle\nabla\psi(z),\Lambda_\tau F(z)\rangle\!\geq\! m|v|^2\!\geq\!\frac{m}{\kappa_M^2}|F(z)|^2,\;\; |\nabla\psi(z)|=|DM(y)^{-1}F(z)|\leq\max\{\tau,c_y\}\,c_M|F(z)|.
	\end{align*}
	These are assumptions (b) and (c) of Theorem~\ref{thm:degenerate} with $\beta=m/\kappa_M^2$ and $c_\psi=\max\{\tau,c_y\}c_M$. Since $M(y)$ is invertible and $v(z)=0$ only for $z=0$, we have $F(z)\neq0$ for $z\in\mathcal{B}_\delta(\mathcal{A})\setminus\mathcal{A}$. If $|z|=r$, then $x^2\geq r^2/2$ or $y^2\geq r^2/2$, so assumption (a) holds with $\alpha_1(r)=\min\{\tau r^2/4,c_yr^4/16\}$ and $\alpha_2(r)=\frac{\tau}{2}r^2+\frac{c_y}{4}r^4$. Moreover, $f$ and $\psi$ are $C^\infty$.

	For the rate, let $z\in\mathcal{B}_\delta(\mathcal{A})$. Then $x^2\geq|x|^3/\delta=\frac{(2/\tau)^{3/2}}{\delta}\big(\frac{\tau}{2}x^2\big)^{3/2}$ and $y^6=\big(\frac{4}{c_y}\big)^{3/2}\big(\frac{c_y}{4}y^4\big)^{3/2}$. Since $a^{3/2}+b^{3/2}\geq2^{-1/2}(a+b)^{3/2}$ for $a,b\geq0$ and $|v|\leq c_M|F|$,
	\begin{align*}
		|F(z)|^2\geq c_M^{-2}\big(x^2+y^6\big)\geq c_F\,\psi(z)^{3/2}.
	\end{align*}
	The remaining assertions follow from Theorem~\ref{thm:degenerate} and Corollary~\ref{cor:rates} with $p=\frac32$ and $c=\gamma\beta c_F$, since $(p-1)\frac{cs}{2}=\frac14\gamma\beta c_Fs$.
\end{proof}

\begin{proof}[Proof of Proposition~\ref{prop:bump}]
	\emph{TT-GDA.} Let $V(z)=|z|_\mathcal{A}^2=u^2$ with $u=(r-R)_+$. Since $\mathcal{A}$ is convex, $V$ is convex and $C^1$ with $\nabla V(z)=2(z-P_\mathcal{A}(z))$, where $P_\mathcal{A}$ is the projection onto $\mathcal{A}$, and $\nabla V$ is Lipschitz continuous with $L_V=2$. Moreover, $\alpha_1(t)=\alpha_2(t)=t^2$. We apply Theorem~\ref{thm:stab1_A} with $W=-F$ and $R_s\equiv0$. For $z\in\mathcal{B}_\delta(\mathcal{A})\setminus\mathcal{A}$, we have $R<r\leq R+\delta\leq R+\epsilon$, and with $\langle z,F(z)\rangle=2\chi(r)r^2+\chi'(r)(x^2-y^2)^2/r$,
	\begin{align*}
		\sigma(z):=-\langle\nabla V(z),W(z)\rangle=\frac{2u}{r}\langle z,F(z)\rangle=2u\,g(z),\quad g(z)=2\chi(r)r+\chi'(r)\frac{(x^2-y^2)^2}{r^2}\geq0 .
	\end{align*}
	We set $\sigma(z)=0$ for $z\in\mathcal{A}$, which agrees with $-\langle\nabla V(z),W(z)\rangle=0$ there. Since $F$, $\nabla V$ and $\sigma$ vanish on $\mathcal{A}$, it suffices to verify the estimates below for $z\notin\mathcal{A}$. The function $\sigma$ is continuous, and $\sigma(z)>0$ for $z\in\mathcal{B}_\delta(\mathcal{A})\setminus\mathcal{A}$ since $\chi(r)>0$ for $r>R$. Let $Q=(x^2-y^2)^2$. Since $\chi(r)=\vartheta(u/\epsilon)$ with $u/\epsilon\in[0,1]$, we have $\chi(r)\leq3u^2/\epsilon^2\leq3u/\epsilon$ and $0\leq\chi'(r)\leq6u/\epsilon^2$, and
	\begin{align*}
		|F(z)|^2=4\chi(r)^2r^2+\frac{4\chi(r)\chi'(r)Q}{r}+\chi'(r)^2Q .
	\end{align*}
	Writing $\chi=\chi(r)$, $\chi'=\chi'(r)$ and $\chi''=\chi''(r)$, and using $\sigma\geq4u\chi r$ and $\sigma\geq2u\chi'Q/r^2$, we obtain $4\chi^2r^2\leq\frac{\chi r}{u}\sigma\leq\frac{3r}{\epsilon}\sigma$, $\frac{4\chi\chi'Q}{r}\leq\frac{2\chi r}{u}\sigma\leq\frac{6r}{\epsilon}\sigma$, and $\chi'^2Q\leq\frac{\chi'r^2}{2u}\sigma\leq\frac{3r^2}{\epsilon^2}\sigma$. Hence
	\begin{align}\label{eq:bump_F2}
		|F(z)|^2\leq C\sigma(z),\qquad C=\frac{9(R+\delta)}{\epsilon}+\frac{3(R+\delta)^2}{\epsilon^2},\qquad z\in\mathcal{B}_\delta(\mathcal{A}),
	\end{align}
	and $\frac{L_V}{2}|W(z)|^2=|F(z)|^2\leq C\sigma(z)$, which is \eqref{eq:compatibility}. Since $\nabla f$ is continuous, $M_w=\max_{\mathcal{B}_\delta(\mathcal{A})}|F|<\infty$, and Theorem~\ref{thm:stab1_A} yields local asymptotic stability of $\mathcal{A}$ for $s\in(0,s^\ast]$ with $s^\ast=\min\{\delta/M_w,1/(2C)\}$.

	For the rate, $\vartheta(\xi)-\xi^2=2\xi^2(1-\xi)\geq0$ for $\xi\in[0,1]$ gives $\chi(r)\geq u^2/\epsilon^2$ whenever $0\leq u\leq\epsilon$, so, with $\sigma$ given by the same formula on $\mathcal{B}_\epsilon(\mathcal{A})\setminus\mathcal{A}$ and $\sigma=0$ on $\mathcal{A}$,
	\begin{align}\label{eq:bump_rate}
		\sigma(z)\geq4u\chi(r)r\geq\frac{4R}{\epsilon^2}u^3=\frac{4R}{\epsilon^2}V(z)^{3/2},\qquad z\in\mathcal{B}_\epsilon(\mathcal{A}) .
	\end{align}
	If $z_0\in\mathcal{A}$, then $F(z_0)=0$ and $z_k=z_0$ for all $k$. Let $0<|z_0|_\mathcal{A}<\delta$. Then Corollary~\ref{cor:rates} with $\ell=V(z_0)<\delta^2=\alpha_1(\delta)$, $p=\frac32$ and $4R/\epsilon^2$ in place of $c$ gives $V(z_k)\leq\big(V(z_0)^{-1/2}+\frac{Rs}{\epsilon^2}k\big)^{-2}$, which is the assertion since $|z_k|_\mathcal{A}=V(z_k)^{1/2}$.

	\emph{GEG.} Let $\mathcal{U}=\{z \mid R<|z|<R+\epsilon\}$. On $\mathcal{U}$, $f(z)=\vartheta\big((|z|-R)/\epsilon\big)(x^2-y^2)$ is $C^\infty$, and its derivatives up to order three are bounded on $\mathcal{U}$, since $\vartheta$ is a polynomial and $R<|z|<R+\epsilon$. Hence there exists $L_2>0$ with $\|\nabla F(z_1)-\nabla F(z_2)\|\leq L_2|z_1-z_2|$ whenever the segment from $z_1$ to $z_2$ lies in $\mathcal{U}$. For $z\in\mathcal{U}$ with $r=|z|$ and angle $\theta$, let $\hat r=z/r$, $\hat\theta=(-\sin\theta,\cos\theta)^\top$ and $c=\cos2\theta$, so that $Q=r^4c^2$ and $g=2\chi r+\chi'r^2c^2$. A direct computation gives
	\begin{align*}
		F(z)=g\,\hat r+h\,\hat\theta,\qquad h=-\chi'(r)r^2c\sin2\theta .
	\end{align*}
	Since $\chi\leq3u^2/\epsilon^2$ and $u\leq r$, we have $g\leq2\chi r+\chi'r^2\leq12ur^2/\epsilon^2$ and $|h|\leq\chi'r^2\leq6ur^2/\epsilon^2$, and therefore $|F(z)|\leq K_Fu$ with $K_F=6\sqrt5(R+\epsilon)^2/\epsilon^2$. Moreover, $\nabla F(z)F(z)=D_FF(z)$, where $D_F\varphi(z)=\frac{d}{dt}\varphi(z+tF(z))\big|_{t=0}$ denotes the derivative of a function $\varphi$ in the direction $F(z)$. Since $F=g\hat r+h\hat\theta$, we have $D_Fr=g$ and $D_F\theta=h/r$, hence $D_F\hat r=\frac hr\hat\theta$ and $D_F\hat\theta=-\frac hr\hat r$, so that
	\begin{align*}
		\langle\hat r,\nabla F(z)F(z)\rangle=g\,\partial_rg+\frac hr\partial_\theta g-\frac{h^2}{r}\leq g\,|\partial_rg|+4\chi'^2r^3c^2\sin^22\theta\leq g\big(|\partial_rg|+4\chi'r\big),
	\end{align*}
	where we used $\partial_\theta g=-4\chi'r^2c\sin2\theta$ and $\chi'r^2c^2\leq g$. Since $\partial_rg=2\chi+2\chi'r+(\chi''r^2+2\chi'r)c^2$, $0\leq\chi\leq1$, $0\leq\chi'\leq\frac{3}{2\epsilon}$, $|\chi''|\leq\frac{6}{\epsilon^2}$ and $\nabla V(z)=2u\hat r$, we obtain
	\begin{align}\label{eq:bump_lead}
		\langle\nabla V(z),\nabla F(z)F(z)\rangle\leq C_1\sigma(z),\qquad C_1=2+\frac{12(R+\epsilon)}{\epsilon}+\frac{6(R+\epsilon)^2}{\epsilon^2},\qquad z\in\mathcal{B}_\delta(\mathcal{A})\setminus\mathcal{A}.
	\end{align}
	Let $s\leq s_1:=\min\{1/L,1/(2K_F)\}$ and $z\in\mathcal{B}_\delta(\mathcal{A})\setminus\mathcal{A}$, and let $z_\lambda=z-\lambda sF(z)$ for $\lambda\in[0,1]$. Then $|z_\lambda-z|\leq sK_Fu\leq u/2$, so $R+u/2\leq|z_\lambda|\leq R+\tfrac32\delta<R+\epsilon$, and the segment from $z$ to $z-sF(z)$ lies in $\mathcal{U}$. The update of \eqref{dyn:GEG} with $\tau=1$ satisfies \eqref{eq:euler_decomp} with $W=-\gamma F$ and
	\begin{align*}
		R_s(z)=\frac{\gamma}{s}\big(F(z)-F(z-sF(z))\big)=\gamma\int_0^1\nabla F(z_\lambda)\,d\lambda\,F(z),
	\end{align*}
	hence $|R_s(z)-\gamma\nabla F(z)F(z)|\leq\frac{\gamma L_2s}{2}|F(z)|^2$. With \eqref{eq:bump_lead}, $|\nabla V(z)|\leq2\delta$ and \eqref{eq:bump_F2}, we obtain $\langle\nabla V(z),R_s(z)\rangle\leq\gamma(C_1+\delta L_2s_1C)\sigma(z)$. Moreover, $|W(z)+sR_s(z)|=\gamma|F(z-sF(z))|\leq\gamma(1+sL)|F(z)|\leq2\gamma|F(z)|$, so $\frac{L_V}{2}|W(z)+sR_s(z)|^2\leq4\gamma^2C\sigma(z)$. Since $\langle\nabla V,W\rangle=-\gamma\sigma$, assumptions (iii) and (iv) of Theorem~\ref{thm:stab1_A} hold with $\gamma\sigma$ in place of $\sigma$, $\bar s=s_1$ and $C_{\mathrm{GEG}}=C_1+\delta L_2s_1C+4\gamma C$. Both sides of \eqref{eq:compatibility} vanish on $\mathcal{A}$, and $\sup\{|W(z)+sR_s(z)| \mid z\in\mathcal{B}_\delta(\mathcal{A}),\ s\in(0,s_1]\}\leq2\gamma M_w<\infty$. Theorem~\ref{thm:stab1_A} yields local asymptotic stability of $\mathcal{A}$ for all sufficiently small $s>0$. By \eqref{eq:bump_rate}, $\gamma\sigma\geq\frac{4\gamma R}{\epsilon^2}V^{3/2}$ on $\mathcal{B}_\delta(\mathcal{A})$, and Corollary~\ref{cor:rates} with $p=\tfrac32$ and $4\gamma R/\epsilon^2$ in place of $c$ gives $V(z_k)\leq\big(V(z_0)^{-1/2}+\frac{\gamma Rs}{\epsilon^2}k\big)^{-2}$ for $0<|z_0|_\mathcal{A}<\delta$, which is the assertion.

		\emph{TT-PPM.} Let $s\in(0,1/L)$. As in Section~\ref{sec:main_p1}, $w_s(z)$ is the unique solution of $w+sF(w)=z$, that is, $w_s(z)=z+sW(w_s(z))$ with $W=-F$, and $w_s$ is continuous. Let $\hat\sigma(z)=\langle\nabla V(z),F(z)\rangle=-\langle\nabla V(z),W(z)\rangle$. Then $\hat\sigma=\sigma$ on $\mathcal{B}_\epsilon(\mathcal{A})$, and $\hat\sigma(z)=4u|z|$ for $|z|\geq R+\epsilon$, where $F(z)=2z$. Hence $\hat\sigma$ is continuous on $\R^2$, $\hat\sigma\geq0$, and $\hat\sigma(z)>0$ if and only if $z\notin\mathcal{A}$. Since $V$ is convex and $C^1$ with $V(z)=|z|_\mathcal{A}^2$ on $\R^2$, all assumptions of Proposition~\ref{prop:implicit} hold, so $\mathcal{A}$ is locally asymptotically stable, every solution converges to $\mathcal{A}$, and $v_k:=V(z_k)$ satisfies $v_k-v_{k+1}\geq s\hat\sigma(z_{k+1})$ by \eqref{eq:implicit_decrease}.
		For the rate, note that $z_{k+1}\in\mathcal{A}$ implies $F(z_{k+1})=0$ and hence $z_k=z_{k+1}$, so $z_0\notin\mathcal{A}$ implies $v_k>0$ for all $k$. Let $0<|z_0|_\mathcal{A}\leq\epsilon$. Since $v_k$ is nonincreasing, $z_k\in\mathcal{B}_\epsilon(\mathcal{A})$ for all $k$, and \eqref{eq:bump_rate} gives $\hat\sigma(z_{k+1})\geq\frac{4R}{\epsilon^2}v_{k+1}^{3/2}$. With $a=4Rs/\epsilon^2$, we obtain $v_k-v_{k+1}\geq av_{k+1}^{3/2}$. Let $q_k=(v_k/v_{k+1})^{1/2}\geq1$ and $b_k=av_k^{1/2}$. Dividing the previous inequality by $v_{k+1}$ gives $q_k^2-1\geq av_{k+1}^{1/2}=b_k/q_k$, that is, $q_k-1\geq b_k/\big(q_k(q_k+1)\big)$. Moreover,
	\begin{align*}
		v_{k+1}^{-1/2}-v_k^{-1/2}=v_k^{-1/2}(q_k-1).
	\end{align*}
	If $q_k\geq1+b_k$, then $v_k^{-1/2}(q_k-1)\geq v_k^{-1/2}b_k=a$. Otherwise, $q_k(q_k+1)<(1+b_k)(2+b_k)$, and therefore $v_k^{-1/2}(q_k-1)\geq v_k^{-1/2}b_k/\big((1+b_k)(2+b_k)\big)=a/\big((1+b_k)(2+b_k)\big)$. Since $v_k\leq v_0$, we have $b_k\leq b_0=av_0^{1/2}=4Rs|z_0|_\mathcal{A}/\epsilon^2$, and in both cases
	\begin{align*}
		v_{k+1}^{-1/2}-v_k^{-1/2}\geq\frac{a}{(1+b_0)(2+b_0)}=\frac{a}{2c_0}.
	\end{align*}
	Summing gives $v_k^{-1/2}\geq v_0^{-1/2}+\frac{a}{2c_0}k$, which is the stated rate, since $|z_k|_\mathcal{A}=v_k^{1/2}$ and $\frac{a}{2c_0}=\frac{2Rs}{c_0\epsilon^2}$.
\end{proof}

\section{\texorpdfstring{$O(1)$- and $O(s)$}{O(1)- and O(s)}-resolution ODE Constructions}\label{app:o}

In this section, we provide the details of the construction of the $O(1)$- and $O(s)$-resolution ODEs of the DTAs studied in this paper using Theorem~\ref{th:ODE}. In all cases, $w_0(z)=z$. With
\begin{align*}
 g_1(z) =    \frac{\partial w_s(z)}{\partial s}\Big|_{s=0}, \qquad
g_2(z)= \frac{1}{2} \frac{\partial^2 w_s(z)}{\partial s^{2}}\Big|_{s=0},
\end{align*}
Theorem~\ref{th:ODE} gives $f_0=g_1$ and $f_1=g_2-\frac12h_{2,0}$ with $h_{2,0}=\nabla f_0\,f_0$. The $O(1)$-resolution ODE is $\dot z=f_0(z)$ and the $O(s)$-resolution ODE is $\dot z=f_0(z)+sf_1(z)$. We also record the Jacobians of $f_0$ and $f_1$ at a critical point $z^\ast$, that is, $F(z^\ast)=0$. For this, we use $\Lambda_\tau\nabla F(z^\ast)=-N$ with $N$ in \eqref{eq:N_notation}, and the fact that the Jacobian of $z\mapsto A(z)b(z)$ at $z^\ast$ equals $A(z^\ast)\frac{\partial b}{\partial z}(z^\ast)$ if $b(z^\ast)=0$.

\subsection{\texorpdfstring{\eqref{TTGDA}}{(TT-GDA)}}
The update is $w_s(z)=z-s\Lambda_\tau F(z)$. Thus $g_1=-\Lambda_\tau F$, $g_2=0$, and
\begin{align*}
	f_0=-\Lambda_\tau F,\qquad f_1=-\tfrac12\nabla f_0\,f_0=-\tfrac12\Lambda_\tau\nabla F\,\Lambda_\tau F,
\end{align*}
which leads to \eqref{o1:TTGDA} and \eqref{os:TTGDA}. At $z^\ast$, $\frac{\partial f_0}{\partial z}(z^\ast)=N$ and $\frac{\partial f_1}{\partial z}(z^\ast)=-\frac12(\Lambda_\tau\nabla F(z^\ast))^2=-\frac12N^2$.

\subsection{\texorpdfstring{\eqref{GEG}}{(GEG)}}
The update is $w_s(z)=z-\gamma s\Lambda_\tau F(u_s)$ with $u_s=z-s\Lambda_\tau F(z)$. Differentiating with respect to $s$ gives $g_1=-\gamma\Lambda_\tau F$ and $\frac{\partial^2w_s}{\partial s^2}\big|_{s=0}=-2\gamma\Lambda_\tau\nabla F(z)\frac{\partial u_s}{\partial s}\big|_{s=0}=2\gamma\Lambda_\tau\nabla F\Lambda_\tau F$, that is, $g_2=\gamma\Lambda_\tau\nabla F\Lambda_\tau F$. Hence
\begin{align*}
	f_0=-\gamma\Lambda_\tau F,\qquad f_1=\gamma\Lambda_\tau\nabla F\Lambda_\tau F-\tfrac{\gamma^2}{2}\Lambda_\tau\nabla F\Lambda_\tau F=\gamma\big(1-\tfrac{\gamma}{2}\big)\Lambda_\tau\nabla F\Lambda_\tau F,
\end{align*}
which leads to \eqref{o1:GEG} and \eqref{os:GEG}. At $z^\ast$, $\frac{\partial f_0}{\partial z}(z^\ast)=\gamma N$ and $\frac{\partial f_1}{\partial z}(z^\ast)=\gamma(1-\frac{\gamma}{2})N^2$.

\subsection{\texorpdfstring{\eqref{TTPPM}}{(TT-PPM)}}
The update $w_s(z)$ is defined by $w_s(z)+s\Lambda_\tau F(w_s(z))=z$. Differentiating this identity once and twice with respect to $s$ and evaluating at $s=0$, where $w_0(z)=z$, gives $\frac{\partial w_s}{\partial s}\big|_{s=0}=-\Lambda_\tau F$ and $\frac{\partial^2w_s}{\partial s^2}\big|_{s=0}=-2\Lambda_\tau\nabla F\frac{\partial w_s}{\partial s}\big|_{s=0}=2\Lambda_\tau\nabla F\Lambda_\tau F$. Thus $g_1=-\Lambda_\tau F$, $g_2=\Lambda_\tau\nabla F\Lambda_\tau F$, and
\begin{align*}
	f_0=-\Lambda_\tau F,\qquad f_1=\Lambda_\tau\nabla F\Lambda_\tau F-\tfrac12\Lambda_\tau\nabla F\Lambda_\tau F=\tfrac12\Lambda_\tau\nabla F\Lambda_\tau F,
\end{align*}
which leads to \eqref{o1:TTPPM} and \eqref{os:TTPPM}. At $z^\ast$, $\frac{\partial f_0}{\partial z}(z^\ast)=N$ and $\frac{\partial f_1}{\partial z}(z^\ast)=\frac12N^2$. For $\tau=1$, these constructions reduce to those in \cite[App.~B]{lu2022sr}.

\subsection{\texorpdfstring{\eqref{dyn:DN}}{(DN)}}
The update is $w_s(z)=z-sd(z)$. Thus $g_1=-d$, $g_2=0$, and, by \eqref{eq:DN_newton_jac},
\begin{align*}
	f_0=-d,\qquad f_1=-\tfrac12\nabla f_0\,f_0=-\tfrac12\big(I-Q\big)d=\tfrac12\big(Q-I\big)d,
\end{align*}
which leads to \eqref{o1:DN} and \eqref{os:DN}. At $z^\ast$, $d(z^\ast)=0$ and $Q(z^\ast)=0$, so $\frac{\partial f_0}{\partial z}(z^\ast)=-I$ and $\frac{\partial f_1}{\partial z}(z^\ast)=-\frac12\big(\frac{\partial d}{\partial z}(z^\ast)\big)^2=-\frac12I$.

\subsection{\texorpdfstring{\eqref{dyn:RDN}}{(RDN)}}
The update is $w_s(z)=z-sd_\phi(z)$. Thus $g_1=-d_\phi$, $g_2=0$, and, by \eqref{eq:RDN_newton_jac},
\begin{align*}
	f_0=-d_\phi,\qquad f_1=-\tfrac12\nabla f_0\,f_0=-\tfrac12\Big((\nabla F+\phi I)^{-1}\nabla F-Q_\phi\Big)d_\phi,
\end{align*}
which leads to \eqref{o1:RDN} and \eqref{os:RDN}. At $z^\ast$, $d_\phi(z^\ast)=0$, $Q_\phi(z^\ast)=0$ and $\frac{\partial d_\phi}{\partial z}(z^\ast)=(\nabla F(z^\ast)+\phi I)^{-1}\nabla F(z^\ast)=-G_\phi$ with $G_\phi=(\phi I-H(z^\ast))^{-1}H(z^\ast)$, so $\frac{\partial f_0}{\partial z}(z^\ast)=G_\phi$ and $\frac{\partial f_1}{\partial z}(z^\ast)=-\frac12G_\phi^2$.

\end{appendices}

\bibliography{MCSS}

\end{document}